\documentclass{amsart}
\usepackage{amscd}
\usepackage{amssymb}
\usepackage{graphicx, color}
\usepackage{hyperref}


\def\xx{\mathbf{x}}
\def\oo{\mathbf{0}}

\theoremstyle{plain}
\newtheorem{theorem}{Theorem}
\newtheorem{proposition}{Proposition}
\newtheorem{conjecture}{Conjecture}
\newtheorem{lemma}{Lemma}
\newtheorem{fact}{Fact}
\newtheorem{corollary}{Corollary}
\newtheorem{problem}{Question}

\theoremstyle{definition}
\newtheorem{definition}{Definition}
\newtheorem{example}{Example}
\newtheorem{notation}{Notation}

\theoremstyle{remark}
\newtheorem{remark}{Remark}

\newcommand{\rbullet}{\textcolor{red}{\bullet}}
\newcommand{\knot}{\operatorname{knot}}

\newcommand{\cNsum}{ \sum_{\check{n}_{b-1} + 1 \le i \le \check{n}_d
}
}

\newcommand{\cat}{\alpha_i \tilde{x}^*_i}
\newcommand{\caFsum}{\sum_{\check{n}_{b-1} + 1 \le i \le \check{n}_{d}} \alpha_i \tilde{x}^*_i }

\newcommand{\cFsum}{\sum_{\check{n}_{b-1} + 1 \le i \le \check{n}_{d}} \alpha_i x^*_i}
\newcommand{\scFsum}{\sum_{\check{n}_{b-1} + 1 \le i \le \check{n}_{d}} \alpha_i}

\newcommand{\caFsumIrr}{\sum_{i \in \cirri} \alpha_i \tilde{x}^*_i
}

\newcommand{\cFsumIrr}{\sum_{i \in \cirri} \alpha_i x^*_i}

\newcommand{\caFsumConn}{\sum_{i \in \cconni} \alpha_i \tilde{x}^*_i 
}
\newcommand{\cFsumConn}{\sum_{i \in \cconni} \alpha_i x^*_i}

\newcommand{\cineq}{\check{n}_{b-1} + 1 \le i \le \check{n}_{d}}

\newcommand{\cxi}{{x}^*_i}
\newcommand{\cxj}{{x}^*_j}

\newcommand{\cyc}{\operatorname{cyc}}
\newcommand{\rev}{\operatorname{rev}}
\newcommand{\Rep}{R_{\epsilon_1 \epsilon_2 \epsilon_3 \epsilon_4 \epsilon_5}}
\newcommand{\cRep}{\check{R}_{\epsilon_1 \epsilon_2 \epsilon_3 \epsilon_4 \epsilon_5}}

\newcommand{\vs}{\vspace{1mm}}

\newcommand{\ax}{~
\begin{picture}(10,12)
\qbezier(5,10.5)(5.6,10.8)(6.2,11.1)
\qbezier(5,10.3)(5.5,9.8)(6,9.3)
\put(5,3){\circle{15}}
\put(0,-2){\vector(1,1){10}}
\put(10,-2){\vector(-1,1){10}}
\end{picture}~~~
}

\newcommand{\atra}{~
\,\begin{picture}(10,13)
\qbezier(6,10.5)(6.6,10.8)(7.2,11.1)
\qbezier(6,10.3)(6.5,9.8)(7,9.3)
\put(5,3){\circle{15}}
\put(0,-2){\vector(1,1){10}}
\put(10,-2){\vector(-1,1){10}}
\put(5,-5){\vector(0,1){15}}
\end{picture}~~~
}

\newcommand{\atrb}{~
\,\begin{picture}(10,13)
\qbezier(6,10.5)(6.6,10.8)(7.2,11.1)
\qbezier(6,10.3)(6.5,9.8)(7,9.3)
\put(5,3){\circle{15}}
\put(0,-2){\vector(1,1){10}}
\put(10,-2){\vector(-1,1){10}}
\put(5,10){\vector(0,-1){15}}
\end{picture}~~~
}

\newcommand{\aha}{~
\begin{picture}(10,13)
\qbezier(5,10.5)(5.6,10.8)(6.2,11.1)
\qbezier(5,10.3)(5.5,9.8)(6,9.3)
\put(5,3){\circle{15}}
\put(2,-4){\vector(0,1){14}}
\put(8,10){\vector(0,-1){14}}
\put(-3,3){\vector(1,0){16}}
\end{picture}~~~
}

\newcommand{\ahb}{~
\begin{picture}(10,13)
\qbezier(5,10.5)(5.6,10.8)(6.2,11.1)
\qbezier(5,10.3)(5.5,9.8)(6,9.3)
\put(5,3){\circle{15}}
\put(2,10){\vector(0,-1){14}}
\put(8,-4){\vector(0,1){14}}
\put(-3,3){\vector(1,0){16}}
\end{picture}~~~
}

\newcommand{\ahc}{~
\begin{picture}(10,13)
\qbezier(5,10.5)(5.6,10.8)(6.2,11.1)
\qbezier(5,10.3)(5.5,9.8)(6,9.3)
\put(5,3){\circle{15}}
\put(2,-4){\vector(0,1){14}}
\put(8,-4){\vector(0,1){14}}
\put(-3,3){\vector(1,0){16}}
\end{picture}~~~
}

\newcommand{\ahd}{~
\begin{picture}(10,13)
\qbezier(5,10.5)(5.6,10.8)(6.2,11.1)
\qbezier(5,10.3)(5.5,9.8)(6,9.3)
\put(5,3){\circle{15}}
\put(2,10){\vector(0,-1){14}}
\put(8,10){\vector(0,-1){14}}
\put(-3,3){\vector(1,0){16}}
\end{picture}~~~
}

\newcommand{\achordtwo}{\,
\begin{picture}(10,12)
\qbezier(5,10.5)(5.6,10.8)(6.2,11.1)
\qbezier(5,10.3)(5.5,9.8)(6,9.3)
\put(5,3){\circle{15}}
\put(-2,1){\vector(2,1){13}}
\put(12,1){\vector(-2,1){13}}
\put(-2,-1){\vector(1,0){13}}
\end{picture}~
}

\newcommand{\bchordtwo}{\,
\begin{picture}(10,12)
\qbezier(5,10.5)(5.6,10.8)(6.2,11.1)
\qbezier(5,10.3)(5.5,9.8)(6,9.3)
\put(5,3){\circle{15}}
\put(-2,1){\vector(2,1){13}}
\put(12,1){\vector(-2,1){13}}
\put(12,-1){\vector(-1,0){13}}
\end{picture}~
}

\newcommand{\cchordtwo}{\,
\begin{picture}(10,12)
\qbezier(5,10.5)(5.6,10.8)(6.2,11.1)
\qbezier(5,10.3)(5.5,9.8)(6,9.3)
\put(5,3){\circle{15}}
\put(-2,1){\vector(2,1){13}}
\put(-1,8){\vector(2,-1){13}}
\put(-2,-1){\vector(1,0){13}}
\end{picture}~
}

\newcommand{\dchordtwo}{\,
\begin{picture}(10,12)
\qbezier(5,10.5)(5.6,10.8)(6.2,11.1)
\qbezier(5,10.3)(5.5,9.8)(6,9.3)
\put(5,3){\circle{15}}
\put(-2,1){\vector(2,1){13}}
\put(-1,8){\vector(2,-1){13}}
\put(12,-1){\vector(-1,0){13}}
\end{picture}~
}

\newcommand{\echordtwo}{\,
\begin{picture}(10,12)
\qbezier(5,10.5)(5.6,10.8)(6.2,11.1)
\qbezier(5,10.3)(5.5,9.8)(6,9.3)
\put(5,3){\circle{15}}
\put(11,8){\vector(-2,-1){13}}
\put(-1,8){\vector(2,-1){13}}
\put(-2,-1){\vector(1,0){13}}
\end{picture}~
}

\newcommand{\fchordtwo}{\,
\begin{picture}(10,12)
\qbezier(5,10.5)(5.6,10.8)(6.2,11.1)
\qbezier(5,10.3)(5.5,9.8)(6,9.3)
\put(5,3){\circle{15}}
\put(11,8){\vector(-2,-1){13}}
\put(-1,8){\vector(2,-1){13}}
\put(12,-1){\vector(-1,0){13}}
\end{picture}~
}

\newcommand{\gchordtwo}{\,
\begin{picture}(10,12)
\qbezier(5,10.5)(5.6,10.8)(6.2,11.1)
\qbezier(5,10.3)(5.5,9.8)(6,9.3)
\put(5,3){\circle{15}}
\put(11,8){\vector(-2,-1){13}}
\put(12,1){\vector(-2,1){13}}
\put(-2,-1){\vector(1,0){13}}
\end{picture}~
}

\newcommand{\hchordtwo}{\,
\begin{picture}(10,12)
\qbezier(5,10.5)(5.6,10.8)(6.2,11.1)
\qbezier(5,10.3)(5.5,9.8)(6,9.3)
\put(5,3){\circle{15}}
\put(11,8){\vector(-2,-1){13}}
\put(12,1){\vector(-2,1){13}}
\put(12,-1){\vector(-1,0){13}}
\end{picture}~
}

\newcommand{\achordth}{~
\begin{picture}(10,12)
\qbezier(5,10.5)(5.6,10.8)(6.2,11.1)
\qbezier(5,10.3)(5.5,9.8)(6,9.3)
\put(5,3){\circle{15}}
\put(2,-4){\vector(0,1){14}}
\put(8,10){\vector(0,-1){14}}
\end{picture}~~~}

\newcommand{\bchordth}{~
\begin{picture}(10,12)
\qbezier(5,10.5)(5.6,10.8)(6.2,11.1)
\qbezier(5,10.3)(5.5,9.8)(6,9.3)
\put(5,3){\circle{15}}
\put(2,10){\vector(0,-1){14}}
\put(8,-4){\vector(0,1){14}}
\end{picture}~~~}

\newcommand{\bchordthb}{~
\begin{picture}(10,12)
\qbezier(5,10.5)(5.6,10.8)(6.2,11.1)
\qbezier(5,10.3)(5.5,9.8)(6,9.3)
\put(5,3){\circle{15}}
\put(2,-4){\vector(0,1){14}}
\put(8,-4){\vector(0,1){14}}
\end{picture}~~~}

\newcommand{\ai}{\begin{picture}(10,12)
\qbezier(5,10.5)(5.6,10.8)(6.2,11.1)
\qbezier(5,10.3)(5.5,9.8)(6,9.3)
\put(5,3){\circle{15}}
\put(11,6){\vector(-1,0){13}}
\put(12,2){\vector(-1,0){15}}
\put(12,-1){\vector(-1,0){13}}
\end{picture}
}

\newcommand{\aiii}{
\begin{picture}(10,12)
\qbezier(5,10.5)(5.6,10.8)(6.2,11.1)
\qbezier(5,10.3)(5.5,9.8)(6,9.3)
\put(5,3){\circle{15}}
\put(11,6){\vector(-1,0){13}}
\put(-2,2){\vector(1,0){15}}
\put(12,-1){\vector(-1,0){13}}
\end{picture}
}

\newcommand{\av}{
\begin{picture}(10,12)
\qbezier(5,10.5)(5.6,10.8)(6.2,11.1)
\qbezier(5,10.3)(5.5,9.8)(6,9.3)
\put(5,3){\circle{15}}
\put(-2,6){\vector(1,0){13}}
\put(-2,2){\vector(1,0){15}}
\put(12,-1){\vector(-1,0){13}}
\end{picture}
}

\newcommand{\aii}{
\begin{picture}(10,12)
\qbezier(5,10.5)(5.6,10.8)(6.2,11.1)
\qbezier(5,10.3)(5.5,9.8)(6,9.3)
\put(5,3){\circle{15}}
\put(12,6){\vector(-1,0){13}}
\put(12,2){\vector(-1,0){15}}
\put(-2,-1){\vector(1,0){13}}
\end{picture}
}

\newcommand{\bi}{
\begin{picture}(10,12)
\qbezier(5,10.5)(5.6,10.8)(6.2,11.1)
\qbezier(5,10.3)(5.5,9.8)(6,9.3)
\put(5,3){\circle{15}}
\put(4,10){\vector(-1,-2){5}}
\put(12,0){\vector(-1,2){5}}
\put(11,-2){\vector(-1,0){12}}
\end{picture}
}

\newcommand{\bii}{
\begin{picture}(10,12)
\qbezier(5,10.5)(5.6,10.8)(6.2,11.1)
\qbezier(5,10.3)(5.5,9.8)(6,9.3)
\put(5,3){\circle{15}}
\put(4,10){\vector(-1,-2){5}}
\put(12,0){\vector(-1,2){5}}
\put(-1,-2){\vector(1,0){12}}
\end{picture}
}

\newcommand{\biii}{
\begin{picture}(10,12)
\qbezier(5,10.5)(5.6,10.8)(6.2,11.1)
\qbezier(5,10.3)(5.5,9.8)(6,9.3)
\put(5,3){\circle{15}}
\put(-2,0){\vector(1,2){5}}
\put(12,0){\vector(-1,2){5}}
\put(11,-2){\vector(-1,0){12}}
\end{picture}
}

\newcommand{\bv}{
\begin{picture}(10,12)
\qbezier(4,10.5)(4.6,10.8)(5.2,11.1)
\qbezier(4,10.3)(4.5,9.8)(5,9.3)
\put(5,3){\circle{15}}
\put(-2,0){\vector(1,2){5}}
\put(6,10){\vector(1,-2){5}}
\put(11,-2){\vector(-1,0){12}}
\end{picture}
}

\newcommand{\conn}{\operatorname{Conn}}
\newcommand{\irr}{\operatorname{Irr}}

\newcommand{\cirri}{\check{I}^{(\operatorname{Irr})}_{b, d}}

\newcommand{\cconni}{\check{I}^{(\operatorname{Conn})}_{b, d}}
\newcommand{\sub}{\operatorname{Sub}}
\newcommand{\ii}{{\rm{I}}}
\newcommand{\sss}{{\rm{SI\!I\!I}}}
\newcommand{\www}{{\rm{WI\!I\!I}}}
\newcommand{\s}{{\rm{SI\!I}}}
\newcommand{\w}{{\rm{WI\!I}}}

\begin{document}
\title[Arrow diagrams on spherical curves and computations]{Arrow diagrams on spherical curves and computations}
\author{Noboru Ito}
\address{Graduate School of Mathematical Sciences, The University of Tokyo, 3-8-1, Komaba, Meguro-ku, Tokyo 153-8914, Japan}
\email{noboru@ms.u-tokyo.ac.jp}
\author{Masashi Takamura}
\address{School of Social Informatics, Aoyama Gakuin University,
5-10-1 Fuchinobe Chuo-ku, Sagamihara-shi, Kanagawa-ken@252-5258, Japan}
\email{takamura@si.aoyama.ac.jp}
\keywords{spherical curve; chord diagram; Reidemeister move; dimensions of finite type invariants}
\date{December 20, 2018}
\maketitle

\begin{abstract}
We give a definition of an integer-valued function $\sum_i \alpha_i x^*_i$ derived from arrow diagrams for the ambient isotopy classes of oriented spherical curves.  Then, we introduce certain elements of the free $\mathbb{Z}$-module generated by the arrow diagrams with at most $l$ arrows, called relators of Type~($\check{\ii}$) (($\check{\s}$), ($\check{\w}$), ($\check{\sss}$), or ($\check{\www}$), resp.), and introduce another function $\sum_i \alpha_i \tilde{x}^*_i$ to obtain $\sum_i \alpha_i x^*_i$.  One of the main results shows that if $\sum_i \alpha_i \tilde{x}^*_i$ vanishes on finitely many relators of Type~($\check{\ii}$) (($\check{\s}$), ($\check{\w}$), ($\check{\sss}$), or ($\check{\www}$), resp.), then $\sum_i \alpha_i \tilde{x}$ is invariant under the deformation of type RI (strong~RI\!I, weak~RI\!I, strong~RI\!I\!I, or weak~RI\!I\!I, resp.).   The other main result is that we obtain functions of arrow diagrams with up to six arrows.  This computation is done with the aid of computers.  
\end{abstract}


\section{Introduction}\label{intro}
A spherical curve is the image of a generic immersion of a circle into a $2$-sphere.   Any two spherical curves are transformed into each other by a finite sequence of deformations of three types RI, RI\!I, and RI\!I\!I, each of which is a replacement of a part of the curve as Fig.~\ref{reide1}.   
\begin{figure}[h!]
\includegraphics[width=12cm]{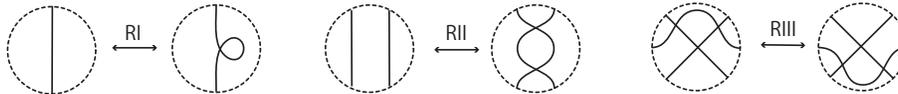}
\caption{Deformations of three types: RI, RI\!I, and RI\!I\!I.}\label{reide1}
\end{figure}
The deformations are obtained from Reidemeister moves of type $\Omega_1$, $\Omega_2$, and $\Omega_3$ on knot diagrams by ignoring over/under information of double points.   It is known that a finite number of repetitions of Reidemeister moves suffice to take any one diagram of a knot to any other.  
\begin{figure}[h!]
\includegraphics[width=12cm]{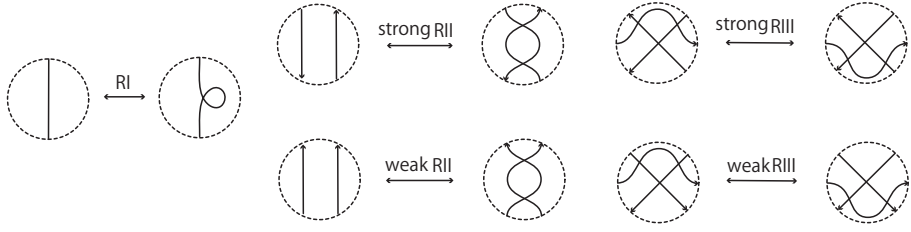}
\caption{Deformations of five types: RI (left), strong~RI\!I (upper center), weak~RII (lower center), strong~RI\!I\!I (upper right), weak~RI\!I\!I (lower right).}\label{refined_reide1}
\end{figure}
\begin{figure}[h!]
\includegraphics[width=5cm]{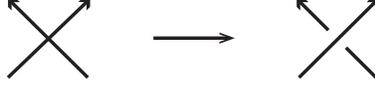}
\caption{The replacement of a double point to the positive crossing}
\label{positive}
\end{figure}

Deformations of $5$ types arise $32$ ($=2^5$) equivalence relations.  By \cite{IT32}, it is known that $32$ equivalence relations arise a non-trivial classification problem for each of $20$ equivalence relations that are mutually different.  
For these equivalence relations, some interesting problems are known \cite{ItoKP}.  
For example, we call an equivalence relation generated by deformations of types RI and weak~RI\!I\!I \emph{weak~(1, 3) homotopy}, and for it, Question~\ref{13weak} had remained unanswered by 2018 \footnote{Only for the unknot, the positive answer was given \cite{ITw13}.}.   
For a spherical curve $P$, we choose an orientation and replace every double point with the positive crossing, with respect to the orientation, as shown in Fig.~\ref{positive}.  

Here, every crossing of the type as in the rightmost figure of Fig.~\ref{positive} is called a \emph{positive crossing}.   Traditionally, if every crossing of a knot diagram is positive, the diagram is called \emph{positive knot diagram} and the knot having the diagram is called a \emph{positive knot}.   
By the replacement, there exists a map from $P$ to the (unoriented) knot diagram.  It is known that the map $f$ induces a map $[f]$ from a weak~(1, 3) homotopy class to the positive knot isotopy class.  
In 2013, Question~\ref{13weak} was given by S.~Kamada and independently Y.~Nakanishi \cite{ItoKP}.    
\begin{problem}\label{13weak}
Is the induced map $[f]$ injective?
\end{problem}
The above map $[f]$ also immediately implies that every knot invariant is a positive knot invariant, which is a weak (1, 3) homotopy invariant of spherical curves.   

Let $V_{\knot}^{(n)}$ ($V_{w13}^{(n)}$, resp.) be a vector space generated by signed arrow diagrams (arrow diagrams, resp.) with at most $n$ arrows, where each element of the vector space is a knot invariant (an invariant under weak (1, 3) homotopy, resp.).   
Bar-Natan, Halacheva, Leung, and Roukema \cite{bar-natan} compute $\dim V_{\knot}^{(n)} / V_{\knot}^{(n-1)}$ of spaces of finite type invariants obtained from the Polyak Algebra, which is generated by arrow diagrams and which has relations induced by Reidemeister moves $\Omega_1$, $\Omega_2$, and $\Omega_3$.  We obtain the computation of $\dim V_{w13}^{(n)} / V_{w13}^{(n-1)}$ which is compared to $\dim V_{\knot}^{(n)} / V_{\knot}^{(n-1)}$.

\begin{equation*}
       \begin{array}{|c|c|c|c|c|c|}
	 \hline
	  n & 2 & 3 & 4 & 5 & 6 \\
	 \hline
	  \dim V_{\knot}^{(n)} / V_{\knot}^{(n-1)} & 0 & 1 & 4 & 17 & \\
	 \hline	  	  
	  \dim V_{w13}^{(n)} / V_{w13}^{(n-1)} & 0 & 1 & 3 & 13 & 31   \\
	 \hline	  
       \end{array}
     \end{equation*}

\begin{conjecture}\label{conjecture1}
The pattern of $\dim V_{w13}^{(n)} / V_{w13}^{(n-1)} \le \dim V_{\knot}^{(n)} / V_{\knot}^{(n-1)}$ would continue.  
\end{conjecture}
  
The difficulty of Question~\ref{13weak} motivates this paper because there may be no arrow diagram formula that is invariant under weak (1, 3) homotopy and is not invariant under knot isotopy.    
In this paper, we introduce a systematic computation (Theorem~\ref{gg_thm2} of Section~\ref{rr2_main}) and an actual computation with an aid of computers (Section~\ref{takamura}) to obtain functions under weak (1, 3) homotopy.  

\section{Preliminaries and main results}\label{s_def}

\subsection{Definitions and notations}\label{ss_def}
In \cite{PV}, Polyak and Viro introduced a method for representing a knot diagram via what is called an oriented Gauss word.  
In this section, we first give a formal treatment of this concept (Definitions~\ref{ori_gg} and \ref{dfn2_arrow}).  Then we apply the idea of \cite{PV} to introduce integer-valued functions of spherical curves.    
\begin{definition}[Gauss word]\label{ori_g}
Let $\hat{n}$ $=$ $\{1, 2, 3, \dots, n\}$.  
A {\it{word}} $w$ of length $n$ is a map $\hat{n}$ $\to$ $\mathbb{N}$.  The word $w$ is represented by $w(1)w(2)w(3) \cdots w(n)$.  
For a word $w : \hat{n}$ $\to$ $\mathbb{N}$, each element of $w(\hat{n})$ is called a {\it{letter}}.  
A word $u$ of length $q$ ($q \le n$) is a \emph{sub-word} of $w$ if there exists an integer $p$ ($q \le p \le n$) such that $u(j)$ $=$ $w(n-p+j)$ ($1 \le j \le q$).  
A {\it{Gauss word}} of length $2n$ is a word $w$ of length $2n$ satisfying that each letter in $w(\hat{2n})$ appears exactly twice in $w(1)w(2)w(3) \cdots w(2n)$.
Let $\cyc$ and $\rev$ be maps $\hat{2n} \to \hat{2n}$ satisfying that $\cyc(p) \equiv p+1$ (mod $2n$) and $\rev(p) \equiv -p+1$ (mod $2n$).   
For two Gauss word $v$ and $w$ of length $2n$, $v$ and $w$ are {\it{isomorphic}} if there exists a bijection $f : v(\hat{2n})$ $\to$ $w(\hat{2n})$ satisfying the following: there exists $t \in \mathbb{Z}$ such that $w \circ (\cyc)^t \circ (\rev)^{\epsilon} = f \circ v$ ($\epsilon = 0$ or $1$).  The isomorphisms obtain an equivalence relation on the Gauss words.  For a Gauss word $v$ of length $2n$, denoted by $[v]$ the equivalence class containing $v$.  
\end{definition}
\begin{definition}[chord diagram]\label{def_chord}
An {\it{chord diagram}} is a configuration of $n$ pair(s) of points up to ambient isotopy and reflection of a circle.  This integer $n$ is called the {\it{length}} of the chord diagram.  
Traditionally, two points of each pair are connected by a simple arc.  This arc is called a \emph{chord}.  Let $G_{\le d}$ be the set of chord diagrams consisting of at most $d$ chords.  
\end{definition}
\begin{figure}[htbp]
\begin{center}
\includegraphics[width=8cm]{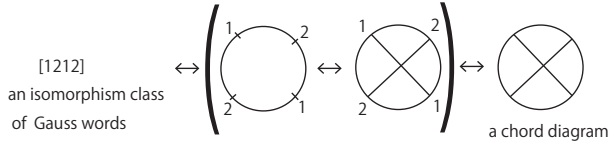}
\caption{Four expressions (Gauss words and chord diagrams).}
\label{four}
\end{center}
\end{figure}  
We note that the equivalence classes of the Gauss words of length $2n$ have one to one correspondence with a chord diagrams, each of which has $n$ chords.  We identify these four expressions (Fig.~\ref{four}) and freely use either one of them in Section~\ref{un}.    
\begin{definition}[oriented Gauss word]\label{ori_gg}
For a given Gauss word $v$ and for each letter $k$ of $v$, we distinguish the two $k$'s in $v$ by calling one $k$ a {\it{starting point}} and the other an {\it{end point}}.  We express the assignments by adding extra informations to $v$ $=$ $v(1)v(2) \cdots v(2n)$, that is, we add ``$\bar{~}$'' on the letters which are assigned end points.      
Then the new word $v^*$ is called an {\it{oriented Gauss word}}.   Each letter of an oriented Gauss word is called an {\it{oriented letter}}.  
Let $v^*$, $w^*$ be oriented Gauss words of length $2n$ induced from $v$, $w$.  
Without loss of generality, we may suppose that the set of the letters in $v(\hat{2n})$ is $\{ 1, 2, \dots, n \}$ (Clearly, $v^*$ is a word of length $2n$ with letters $\{ 1, 2, \ldots, n, \overline{1}, \overline{2}, \ldots, \overline{n} \}$).       
We say that $v^*$ is {\it{isomorphic}} to $w^*$ if there exists a bijection $f : v(\hat{2n})$ $\to$ $w(\hat{2n})$ such that there exists $t \in \mathbb{Z}$ 
such that $w^* \circ (\cyc)^{t} = f^* \circ v^*$ where $f^* :$ $v^*(\hat{2n})$ $=$ $\{ 1, 2, \ldots, n, \overline{1}, \overline{2}, \ldots, \overline{n} \}$ $\to$ $w^*(\hat{2n})$ is the bijection such that $f^*(i)$ $=$ $f(i)$ and $f^*(\overline{i})$ $=$ $\overline{f(i)}$ ($i=1, 2, \dots, n$).  
The isomorphisms give an equivalence relation on the oriented Gauss words.  
For an oriented Gauss word $v^*$ of length $2n$, $[[v^*]]$ denotes the equivalence class containing $v^*$.  An oriented Gauss word ${v^*}'$ is an {\emph{oriented sub-Gauss word}} of the oriented Gauss word $v^*$ if ${v^*}'$ is obtained from $v^*$ by ignoring some pairs of letters, each of which arises from a common number.  Then $\sub(v^*)$ denotes the set of the oriented sub-Gauss words of $v^*$.    
\end{definition}
\begin{definition}[arrow diagram]\label{dfn2_arrow}  
An {\it{arrow diagram}} is a configuration of $n$ pair(s) of points up to ambient isotopy on a circle where each pair of points consists of a starting point and an end point.  The integer $n$ is called the {\it{length}} of the arrow diagram.  
Traditionally, two points of each pair are connected by a simple arc and an assignment of starting and end points on the boundary points of the chord is represented by an arrow on the chord from the starting point to the end point.  This oriented arc is called an \emph{arrow}.      
\end{definition}
\begin{figure}[htbp]
\begin{center}
\includegraphics[width=8cm]{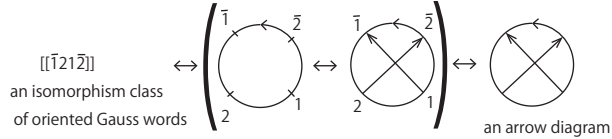}
\caption{Four expressions (oriented Gauss words and arrow diagrams).}
\label{four_arrow}
\end{center}
\end{figure}
Here, we note that, in general, $v^*$ and $v^* \circ \rev$ are not isomorphic, and this implies that an arrow diagram and its reflection image are not ambient isotopic in general.  
Thus, in this paper, we express this condition by assigning an arrow on the circle representing the counter clockwise direction (Fig.~\ref{four_arrow}).  
We note that the equivalence classes of the oriented Gauss words of length $2n$ have one to one correspondence with the arrow diagrams, each of which has $n$ arrows.  
In the rest of this paper, we identify these four expressions, and freely use either one of them depending on situations. 
\begin{notation}[$\check{G}_{\le d}$, $\check{n}_d$, $\check{G}_{b, d}$]\label{not4}
Let $\check{G}_{< \infty}$ be the set of the arrow diagrams, that is, the set of isomorphism classes of the oriented Gauss words.  It is clear that $\check{G}_{< \infty}$ consists of countably many elements.  Hence, there exists a bijection between $\check{G}_{< \infty}$ and $\{ {x}^*_i \}_{i \in {\mathbb{N}}}$, where ${x}^*_i$ is a variable.    
Take and fix a bijection $\check{f} : \check{G}_{< \infty}$ $\to$ $\{ {x}^*_i \}_{i \in {\mathbb{N}}}$ satisfying: the number of arrows of $\check{f}^{-1}({x}^*_i)$ is less than or equal to that of $\check{f}^{-1}({x}^*_j)$ if and only if $i \le j$ $(i, j \in {\mathbb{N}})$.  For each positive integer $d$, let $\check{G}_{\le d}$ be the set of arrow diagrams each consisting of at most $d$ arrows and let $\check{n}_d$ $=$ $|\check{G}_{\le d}|$.  Then, it is clear that $\check{f}|_{\check{G}_{\le d}}$ is a bijection from $\check{G}_{\le d}$ to $\{ {x}^*_1, {x}^*_2, \ldots, {x}^*_{\check{n}_d} \}$.  
Further, for each pair of integers $b$ and $d$ ($2 \le b \le d$), let $\check{G}_{b, d}$ $=$ $\check{G}_{\le d} \setminus \check{G}_{\le b-1}$.
Then, $\check{f}|_{\check{G}_{b, d}}$ is a bijection $\check{G}_{b, d}$ $\to$ $\{ {x}^*_{\check{n}_{b-1} +1}, {x}^*_{\check{n}_{b-1} + 2}, \dots, {x}^*_{\check{n}_d} \}$.
\end{notation}
In the rest of this paper, we use the notations in Notation~\ref{not4} unless otherwise denoted, and we freely use this identification between $\check{G}_{< \infty}$ and $\{{x}^*_i \}_{i \in \mathbb{N}}$. 
\begin{example}[$\check{G}_{2, 3}$]\label{eg_23}
It is elementary to confirm the following, and the details of the proof are left to the reader.  
\begin{align*}
\check{G}_{2, 3} = \{
&\bchordth,~
\bchordthb,~
\achordth,~
\ax,~
\bii~,~
\bi~,~
\biii~,~
\bv~,~
\gchordtwo,~
\achordtwo,~
\echordtwo,~
\cchordtwo,~
\hchordtwo,~
\bchordtwo,~\\
&\fchordtwo,~
\dchordtwo,~
\aii~,~
\aiii~,~
\ai~,~
\av~, ~
\ahb,~
\ahd,~
\ahc,~
\aha,~ 
\atra,~  
\atrb
\}.
\end{align*}
\end{example}
\begin{definition}[an arrow diagram $AD_P$ of a spherical curve $P$]\label{dfn_cdp}
Let $P^+$ be an oriented spherical curve, i.e. there is a generic immersion $g: S^1 \to S^2$ such that $g(S^1)=P^+$.   Then, $P^-$ denotes the oriented spherical curve with the opposite orientation.  
We define an arrow diagram of $P^+$ (e.g., Fig.~\ref{def1}) as follows: Let $k$ be the number of the double points of $P^+$, and $m_1, m_2, \ldots, m_k$ mutually distinct positive integers.   Let $K_{P^{\epsilon}}$ be the knot diagram obtained from $P^{\epsilon}$ by replacing every double point with the positive crossing (with respect to the orientation) as shown in Fig.~\ref{negative}.  \begin{figure}[h!]
\includegraphics[width=3cm]{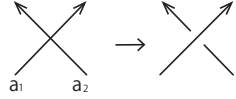}
\caption{Replacement of a double point with a crossing.}\label{negative}
\end{figure}
\begin{figure}[h!]
\includegraphics[width=10cm]{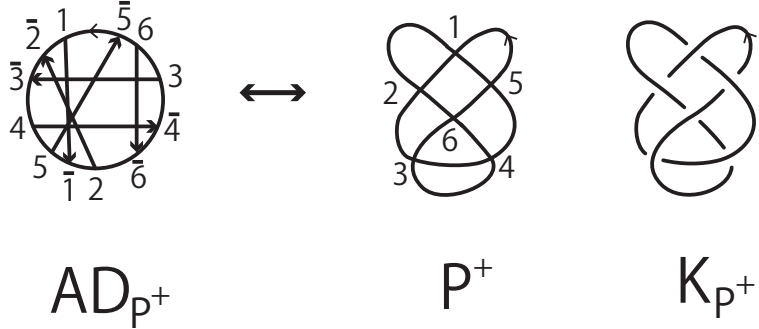}
\caption{An arrow diagram $AD_P$ of an oriented spherical curve $P^+$.}\label{def1}
\end{figure} 

Fix a base point, which is not a double point on $K_{P^{+}}$.  Starting from the base point, proceed along $P^+$ according to the orientation of $K_{P^{+}}$.   
Assign $m_1$ to the first double point that we encounter.  Then, assign $m_2$ to the next double point that we encounter provided it is not the first double point.  Suppose that we have already assigned $m_1$, $m_2, \dots, m_p$.  Then, assign $m_{p+1}$ to the next double point that we encounter if it has not been assigned yet.    
Following the same procedure, we finally label the double points of $P^+$.    
Note that $g^{-1}({\text{double point assigned $m_i$}})$ consists of two points on $S^1$ and we shall assign the pair $({m_i}, \bar{m}_i)$ to them by assigning starting point labeled by $m_i$ (end point labeled by $\bar{m}_i$, resp.) to the over path (under path, resp.).  The arrow diagram represented by $g^{-1}({\text{double point assigned $(m_1, \bar{m}_1)$}}),$ $g^{-1}({\text{double point assigned $(m_2, \bar{m}_2)$}}),$ $\dots,$ $g^{-1}({\text{double point assigned $(m_k, \bar{m}_k)$}})$ on $S^1$ is denoted by $AD_P$ and is called {\it{an arrow diagram of the spherical curve}} $P^+$.  
\end{definition}
\begin{remark}\label{remark2}
It is easy to show that the arrow diagram corresponding to $P^{-\epsilon}$ is ambient isotopic to a reflection of $AD_{P^\epsilon}$, where $(\epsilon, -\epsilon)$ $=$ $(+, -)$ or $(-, +)$.  
\end{remark}

Recall that $AD_{P^\epsilon}$ gives an equivalence class of oriented Gauss words, say $[[v_{P^\epsilon}]]$.  Then, by the definition of the equivalence relation, it is easy to see that the map $P^\epsilon \mapsto [[v_{P^\epsilon}]]$ is well-defined.
\begin{notation}[$x^{*}(AD)$]\label{ad_def}
Let $x^* \in \{ x^*_i \}_{i \in \mathbb{N}}$ (, hence, $x_i$ represents an arrow diagram).  For a given arrow diagram $AD$, fix an oriented Gauss word $G^*$ representing $AD$.  
Let $\sub_{x^*} (G^*)$ $=$ $\{H^*~|~H^* \in \sub(G^*)$, $[[H^*]]=x^* \}$.    
The cardinality of this subset is denoted by $x^*(G^*)$, i.e., $x^*(G^*)$ $=$ $|\sub_{x^*} (G^*)|$.  Let ${G'}^*$ be another oriented Gauss word representing $AD$.   By the definition of the isomorphism of the Gauss words, it is easy to see $x^*({G'}^*)$ $=$ $x^*(G^*)$.  Hence, we shall denote this number by $x^*(AD)$.  
If $AD$ is an arrow diagram of an oriented spherical curve $P^\epsilon$, then $x^*(AD)$ can be denoted by $x^*(P^\epsilon)$.
\end{notation}

Since each equivalence class of the oriented Gauss words is identified with an arrow diagram, we can calculate the number $x^*(AD)$ by using geometric observations.  We explain this philosophy in the next example.  
\begin{example}\label{example2}
We consider the arrow diagram $AD$ in Fig.~\ref{def5} (cf.~Fig.~\ref{def1}).  
Then we label the arrows of $AD$ by $\alpha_i$ $(1 \le i \le 6)$ as in the center figure of Fig.~\ref{def5}.   Consider the subset of the power set of $\{ \alpha_1, \alpha_2, \dots, \alpha_6 \}$, each element of which represents a chord diagram 
isomorphic to $\otimes$.  It is elementary to see that this subset consists of eleven  elements, those are, $\{\alpha_1, \alpha_2\}$, $\{\alpha_1, \alpha_3\}$, $\{\alpha_1, \alpha_4\}$, $\{\alpha_1, \alpha_5\}$$\{\alpha_2, \alpha_3\}$, $\{\alpha_2, \alpha_4\}$, $\{\alpha_2, \alpha_6\}$, $\{\alpha_3, \alpha_4\}$, $\{\alpha_3, \alpha_5\}$, $\{\alpha_4, \alpha_5\}$, and $\{\alpha_5, \alpha_6\}$, and that fact shows that $\ax^*(AD)=11$.  Similarly, we have values of $x^*(AD)$ for several $x^*$'s as in Fig.~\ref{def5}.     
\begin{figure}[h!]
\includegraphics[width=12cm]{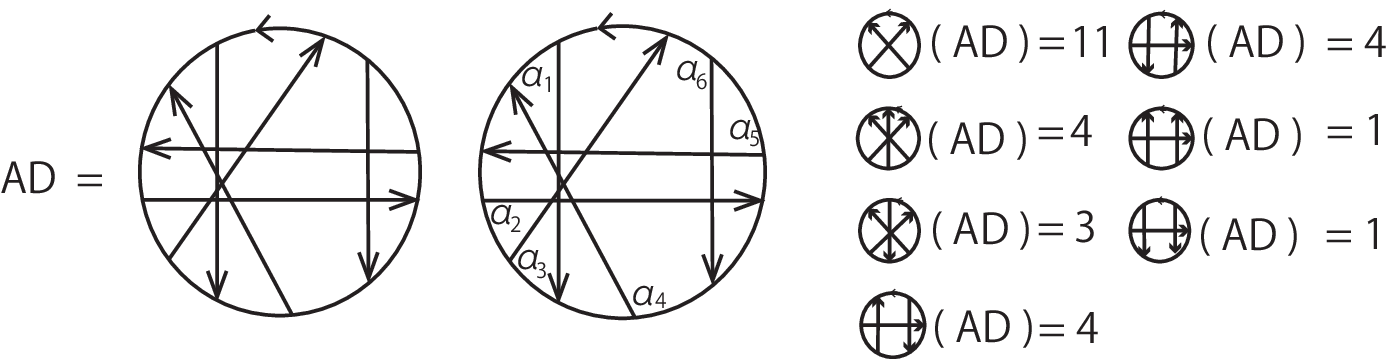}
\caption{$x^*(AD)$.}\label{def5}
\end{figure}
\end{example}
Let $AD$ be an arrow diagram consisting of arrows $\alpha_1, \alpha_2, \dots, \alpha_k$.  Then the arrow diagram consisting of a subset of $\{ \alpha_1, \alpha_2, \dots, \alpha_k \}$ is called a {\it{sub-arrow diagram}} of $AD$.  
\begin{example}\label{example3}
Consider the arrow diagram $AD_{P^+}$ in Fig.~\ref{def6}, which is obtained from $P$ in Fig.~\ref{def1} with an appropriate orientation.  Then we have the values of $x^*(P^{\epsilon})$ for several $x^*$'s as in Fig.~\ref{def6}.    
\begin{figure}[h!]
\includegraphics[width=10cm]{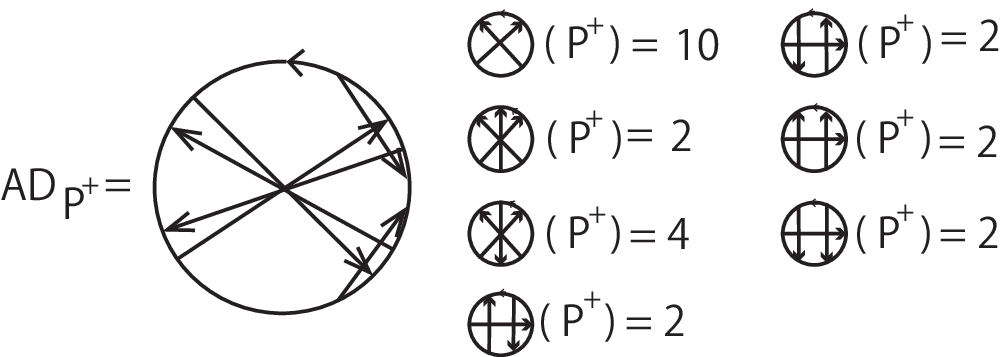}
\caption{$x^*(P^+)$.}\label{def6}
\end{figure}
\end{example}
\begin{definition}[$\tilde{x^*} (z)$, $\tilde{x^*} ({[[z]]})$]\label{tilde_ll}
Let $x^*$ be an arrow diagram.  
We define the function $\tilde{x^*}$ from the set of the oriented Gauss words to $\{ 0, 1 \}$ by
 \[
{\tilde{x^*}} (F^*) = \begin{cases}
1 \quad [[F^*]] = x^* \\
0 \quad [[F^*]] \neq x^*.
\end{cases}
\] 
By definition, it is easy to see $\tilde{x^*}(F^*_1)$ $=$ $\tilde{x^*}(F^*_2)$ for each pair $F^*_1$, $F^*_2$ with $[[F^*_1]]=[[F^*_2]]$.  Hence, we shall denote this number by $\tilde{x^*}([[F^*_1]])$.  If $[[F^*]]$ corresponds to an arrow diagram of an oriented spherical curve $P^{\epsilon}$, then $\tilde{x^*}([[F^*]])$ is denoted by $\tilde{x^*}(P^{\epsilon})$.  
Further, let $\mathbb{Z}[\check{G^*}_{\le l}]$ be the free $\mathbb{Z}$-module generated by the elements of $\check{G^*}_{\le l}$, where $l$ is sufficiently large.  
We linearly extend $\tilde{x^*}$ to the function from $\mathbb{Z}[\check{G^*}_{\le l}]$ to $\mathbb{Z}$.  
It is clear that for any oriented Gauss word $G^*$ with $[[G^*]] = AD$, 
\begin{equation}\label{tilde_x*}
x^*(AD) = \sum_{z^* \in \sub(G^*)} \tilde{x^*} (z^*).
\end{equation}
\end{definition}
\subsection{Relators and deformations}\label{sec_main_result_2}
Let $\mathbb{Z}[\check{G}_{\le l}]$ be the free $\mathbb{Z}$-module defined in Subsection~\ref{ss_def}.   
In this subsection, first we define the elements of $\mathbb{Z}[\check{G}_{\le l}]$ called {\it{relators}} of Type~($\check{\ii}$), Type~($\check{\s}$), Type~($\check{\w}$), Type~($\check{\sss}$), and Type~($\check{\www}$).  
\begin{definition}[Relators, cf.~Fig.~\ref{relator2}]\label{def_relators_arrow}
\begin{itemize}
\item Type ($\check{\ii}$). 
An element $r^*$ of $\mathbb{Z}[\check{G}_{\le l}]$ is called a \emph{Type $(\check{\ii})$ relator} if there exist an oriented Gauss word $S$ and a letter $i$ not in $S$ such that $r^*$ $=$ $[[Si\bar{i}]]$ or $[[S\bar{i}i]]$.  
\item Type ($\check{\s}$).  
An element $r^*$ of $\mathbb{Z}[\check{G}_{\le l}]$ is called a \emph{Type $(\check{\s})$ relator} if there exist an oriented Gauss word $ST$ and letters $i$ and $j$ not in $ST$ such that $r^*$ $=$ $[[Si\bar{j}Tj\bar{i}]]$ $+$ $[[SiT\bar{i}]]$ $+$ $[[S\bar{j}Tj]]$ or $[[S\bar{i}jT\bar{j}i]]$ $+$ $[[S\bar{i}Ti]]$ $+$ $[[SjT\bar{j}]]$.    
\item Type ($\check{\w}$).  
An element $r^*$ of $\mathbb{Z}[\check{G}_{\le l}]$ is called a \emph{Type $(\check{\w})$ relator} if there exist an oriented Gauss word $ST$ and letters $i$ and $j$ not in $ST$ such that $r^*$ $=$ $[[S\bar{i}jTi\bar{j}]]$ $+$ $[[S\bar{i}Ti]]$ $+$ $[[SjT\bar{j}]]$.    
\item Type ($\check{\sss}$).  
An element $r^*$ of $\mathbb{Z}[\check{G}_{\le l}]$ is called a \emph{Type $(\check{\sss})$ relator} if there exist an oriented Gauss word $STU$ and letters $i$, $j$ and $k$ not in $STU$ such that 
\begin{align*}
r^* &= \left( [[S\bar{i}jT\bar{k}iU\bar{j}k]] + [[S\bar{i}jTiU\bar{j}]] + [[S\bar{i}T\bar{k}iUk]] + [[SjT\bar{k}U\bar{j}k]] \right) \\
-&  \left([[Sj\bar{i}Ti\bar{k}Uk\bar{j}]] + [[Sj\bar{i}TiU\bar{j}]] + [[S\bar{i}Ti\bar{k}Uk]] + [[SjT\bar{k}Uk\bar{j}]] \right)
\end{align*}
or
\begin{align*}
& \left( [[Sk\bar{j}Ti\bar{k}Uj\bar{i}]] + [[S\bar{j}TiUj\bar{i}]] + [[SkTi\bar{k}U\bar{i}]] + [[Sk\bar{j}T\bar{k}Uj]] \right)\\
- & \left([[S\bar{j}kT\bar{k}iU\bar{i}j]] + [[S\bar{j}TiU\bar{i}j]] + [[SkT\bar{k}iU\bar{i}]] + [[S\bar{j}kT\bar{k}Uj]] \right). 
\end{align*} 
\item Type ($\check{\www}$).   
An element $r^*$ of $\mathbb{Z}[\check{G}_{\le l}]$ is called a \emph{Type $(\check{\www})$ relator} if there exist an oriented Gauss word $STU$ and letters $i$, $j$, and $k$ not in $STU$ such that 
\begin{align*}
r^*& =\left( [[S\bar{i}\,\bar{j}Ti\bar{k}Ujk]] + [[S\bar{i}\,\bar{j}TiUj]] + [[S\bar{i}Ti\bar{k}Uk]] + [[S\bar{j}T\bar{k}Ujk]] \right) \\
- & \left( [[S\bar{j}\, \bar{i}T\bar{k}iUkj]] + [[S\bar{j}\, \bar{i}TiUj]] + [[S\bar{i}T\bar{k}iUk]] + [[S\bar{j}T\bar{k}Ukj]]  \right).
\end{align*}
or
\begin{align*}
&\left( [[SkjT\bar{k}iU\bar{j}\bar{i}]] + [[SjTiU\bar{j}\bar{i}]] + [[SkT\bar{k}iU\bar{i}]] + [[SkjT\bar{k}U\bar{j}]] \right) \\
- & \left( [[SjkTi\bar{k}U\bar{i}\bar{j}]] + [[SjTiU\bar{i}\bar{j}]] + [[SkTi\bar{k}U\bar{i}]] + [[SjkT\bar{k}U\bar{j}]]  \right).
\end{align*} 
\end{itemize}
\end{definition}
\begin{figure}[h!]
\includegraphics[width=10cm]{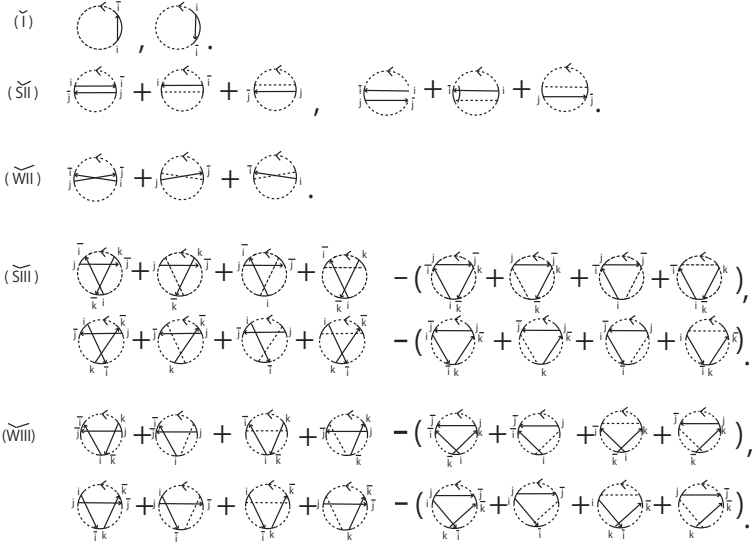}
\caption{Relators.}\label{relator2}
\end{figure}

We note that in Definition~\ref{def_relators_arrow}, if the oriented Gauss word $G^*$ is obtained from an oriented spherical curve $P^\epsilon$, then each relator corresponds to a deformation: Type ($\check{\ii}$) relator to RI, Type ($\check{\s}$) relator to strong~RI\!I, Type ($\check{\w}$) relator to weak~RI\!I, Type ($\check{\sss}$) relator to strong~RI\!I\!I, and Type ($\check{\www}$) relator to weak~RI\!I\!I.  For the precise statement of this note, we introduce the following setting.
\begin{figure}[h!]
\includegraphics[width=12cm]{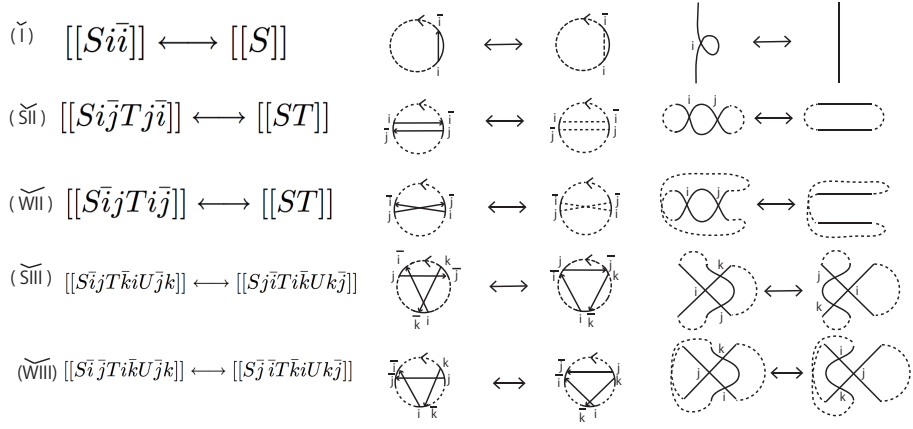}
\caption{Deformations (the mirror images of the center and the corresponding words are omitted).}\label{def3c}
\end{figure}

We first introduce the following notations.    
Let $P^{\epsilon}$ and ${P'}^{\epsilon}$ be two oriented spherical curves.  
If $P^{\epsilon}$ is related to ${P'}^{\epsilon}$ by a single RI (strong~RI\!I, weak~RI\!I, strong~RI\!I\!I, or weak~RI\!I\!I resp.), then there are oriented Gauss words $G^*$ and ${G'}^*$ such that (by exchanging $P$ and $P'$, if necessary), $G^*$ $=$ $Si\bar{i}$ ($S\bar{i}i$, $Si\bar{j}Tj\bar{i}$, $S\bar{i}jT\bar{j}i$, $S\bar{i}jTi\bar{j}$, $S\bar{i}jT\bar{k}iU\bar{j}k$, $Sk\bar{j}Ti\bar{k}Uj\bar{i}$, $S\bar{i}\,\bar{j}Ti\bar{k}Ujk$, or $SkjT\bar{k}iU\bar{j}\bar{i}$ resp.) and ${G'}^*=S$ ($S$, $ST$, $ST$, $ST$, $Sj\bar{i}Ti\bar{k}Uk\bar{j}$, $S\bar{j}kT\bar{k}iU\bar{i}j$, $S\bar{j}\,\bar{i}T\bar{k}iUkj$, or $SjkTi\bar{k}U\bar{i}\bar{j}$ resp.) such that $[[G^*]]=AD_{P^{\epsilon}}$ and $[[{G'}^*]]=AD_{{P'}^{\epsilon}}$ (Fig.~\ref{def3c}).  The subset of $\sub(G^*)$ such that each element has exactly $m$ pairs of letters, each of which arises from $i$, $j$, and $k$ is denoted by $\sub^{(m)} (G)$.    By definition, 
\begin{equation}\label{sub*}
\sub(G^*) = \sub^{(0)}(G^*) \amalg \sub^{(1)} (G^*) \amalg \sub^{(2)} (G^*) \amalg \sub^{(3)} (G^*).   
\end{equation}
Similarly, for an arrow diagram $x^*$, $\sub^{(m)}_{x^*} (G^*)$ denotes the subset of $\sub_{x^*}(G^*)$ consisting of elements, each of which has exactly $m$ pairs of letters, each of which arises from $i$, $j$, and $k$.  Then,  
\begin{equation}\label{subx*}
\sub_{x^*} (G^*) = \sub^{(0)}_{x^*} (G^*) \amalg \sub^{(1)}_{x^*} (G^*) \amalg \sub^{(2)}_{x^*} (G^*) \amalg \sub^{(3)}_{x^*} (G^*).  
\end{equation}

Let $\check{\mathcal{C}}$ be the set of the ambient isotopy classes of oriented spherical curves.  
Next, for each element $\sum_i \alpha_i x^*_i \in \mathbb{Z}[\check{G}_{\le l}]$, we define a function $\check{\mathcal{C}}$ $\to$ $\mathbb{Z}$, also denoted by $\sum_i \alpha_i x^*_i$, and another function denoted by $\sum_i \alpha_i \tilde{x}^*_i$.  
\begin{definition}[$\sum_i \alpha_i x^*_i$, $\sum_i \alpha_i \tilde{x}^*_i$]\label{def_xiv}
Let $b$ and $d$ ($2 \le b \le d$) be integers, $\check{G}_{\le d}$, $\check{G}_{b, d}$, $\{ {x}^*_i \}_{i \in \mathbb{N}}$.   Let $\mathbb{Z}[\check{G}_{\le l}]$ be as in Subsection~\ref{ss_def}.   
Recall that $\check{n}_d$ $=$ $|\check{G}_{\le d}|$ and $\check{G}_{b, d}$ $=$ $\{ {x}^*_i \}_{\check{n}_{b-1} +1 \le i \le \check{n}_d}$ and each $x^*_i$ represents the arrow diagram $\check{f}^{-1}(x^*_i)$ in Notation~\ref{not4}.  For each element
\[ \sum_{ \cineq } \alpha_i x^*_i \in \mathbb{Z}[\check{G}_{\le l}], \]
we define an integer-valued function $\mathcal{\check{C}}$ $\to$ $\mathbb{Z}$, also denoted by $\cFsum$, by
\[P^{\epsilon} \mapsto \cFsum (P^{\epsilon}),\]
where $x^*_i (P^{\epsilon})$ is the integer introduced in Notation~\ref{ad_def}.   

Analogously, for each $\cFsum \in \mathbb{Z}[\check{G}_{\le l}]$, we define the function
\[
\cNsum \cat : \check{G}_{\le l} \to \mathbb{Z}
\]
by
\[
\left( \cNsum \cat \right)(z^*) = \cNsum \cat(z^*),  
\]
where $\tilde{x}_i (z^*)$ is the integer introduced in Definition~\ref{tilde_ll}.
\end{definition}

By using this setting, we can describe the relation between a Type~($\check{\ii}$) relator and a single RI as follows.  
Suppose that an oriented spherical curve $P^{\epsilon}$ is related to another oriented spherical curve ${P'}^{\epsilon}$ by a single RI.  Recall that there exist an oriented letter $i$ and an oriented Gauss word $S$ such that $AD_{P^{\epsilon}}$ $=$ $[[Si\bar{i}]]$ or $[[S\bar{i}i]]$ and $AD_{{P'}^{\epsilon}}$ $=$ $[[S]]$.  Since the arguments are essentially the same, we may suppose, without loss of generality, $AD_{P^{\epsilon}}$ $=$ $[[Si\bar{i}]]$.   Then, let $G^*$ $=$ $Si\bar{i}$.  
Note that in this case in the decomposition (\ref{sub*}), $\sub^{(2)}(G^*)$ $=$ $\emptyset$ and $\sub^{(3)}(G^*)$ $=$ $\emptyset$.  
Then, by (\ref{tilde_x*}) in Definition~\ref{tilde_ll} and (\ref{sub*}), 
\begin{equation*}
\begin{split}
\cFsum &(P^{\epsilon}) = \scFsum \left( \sum_{z^* \in \sub(G^*)} \tilde{x}^*_i (z^*) \right) \\
\!\!\!&= \scFsum \left( \sum_{z^*_0 \in \sub^{(0)}(G^*)} \tilde{x}^*_i (z^*_0) + \sum_{z^*_1 \in \sub^{(1)}(G^*)} \tilde{x}^*_i (z^*_1) \right).    
\end{split}
\end{equation*}
Note that each element $z^*_0 \in$ $\sub^{(0)}(G^*)$ is an oriented sub-Gauss word of $S$.  Then it is clear that $\sub^{(1)}(G^*)$ $=$ $\{ z^*_0 i\bar{i}~|~z^*_0 \in \sub^{(0)}(G^*) \}$.  
Hence, 
\begin{align*}
&\cFsum(P^{\epsilon}) \\
&= \scFsum \left( \sum_{z^*_0 \in \sub^{(0)}(G^*)} \tilde{x}^*_i (z^*_0) + \sum_{z^*_0 \in \sub^{(0)}(G^*)} \tilde{x}^*_i (z^*_0 i\bar{i}) \right) \\
&= \scFsum \left( \sum_{z^*_0 \in \sub^{(0)}(G^*)} \tilde{x}^*_i ([[z^*_0]]) + \sum_{z^*_0 \in \sub^{(0)}(G^*)} \tilde{x}^*_i ([[z^*_0 i\bar{i}]]) \right).      
\end{align*}
On the other hand, since $\sub({G'}^*)$ is identified with $\sub^{(0)}(G^*)$,
\begin{align*}
\cFsum({P'}^{\epsilon}) &= \scFsum \left(\sum_{{z'}^* \in \sub ({G'}^*)} \tilde{x}^*_i ({z'}^*) \right) (\because (\ref{tilde_x*})) \\
&=\scFsum \left(\sum_{z^*_0 \in \sub^{(0)} (G^*)} \tilde{x}^*_i (z^*_0) \right)\\
&= \scFsum \left( \sum_{z^*_0 \in \sub^{(0)} (G^*)} \tilde{x}^*_i \big( [[z^*_0 ]] \big) \right).  
\end{align*}
As a conclusion, the difference of the values is calculated as follows.   
\begin{align*}
\cFsum(P^{\epsilon}) &-\cFsum({P'}^{\epsilon})\\ 
&= \cNsum \sum_{z^*_0 \in \sub^{(0)}(G^*)} \alpha_i \tilde{x}^*_i \left( [[z^*_0 i\bar{i}]] \right).
\end{align*}
We note that this is a linear combination of the values of Type~$\check{(\ii)}$ relators via $\tilde{x}^*_i$. 

For the case Type ($\check{\s}$), ($\check{\w}$), ($\check{\sss}$), or ($\check{\www}$) relators, the arguments are slightly more complicated than that of Type ($\check{\ii}$) relator.  We will explain them in the proof of Theorem~\ref{gg_thm2} and thus, we omit them here.       
\begin{definition}[$\check{R}_{\epsilon_1 \epsilon_2 \epsilon_3 \epsilon_4 \epsilon_5}$]\label{dfn_relator2}
For each $(\epsilon_1, \epsilon_2, \epsilon_3, \epsilon_4, \epsilon_5) \in \{0, 1\}^{5}$, let $\check{R}_{\epsilon_1 \epsilon_2 \epsilon_3 \epsilon_4 \epsilon_5}$ $=$ $\cup_{\epsilon_i = 1} \check{R}_i$ ($\subset \cup_{l \ge 1} \mathbb{Z}[\check{G}_{\le l}]$), where $\check{R}_1$ is the set of the Type ($\check{\ii}$) relators (corresponding to RI), $\check{R}_2$ is the set of the Type ($\check{\s}$) relators (corresponding to strong~RI\!I), $\check{R}_3$ is the set of the Type ($\check{\w}$) relators (corresponding to weak~RI\!I), $\check{R}_4$ is the set of the Type ($\check{\sss}$) relators (corresponding to strong~RI\!I\!I), and $\check{R}_5$ is the set of the Type ($\check{\www}$) relators (corresponding to weak~RI\!I\!I).  
\end{definition}
For integers $b$ and $d$ ($2 \le b \le d$), let $\check{O}_{b, d}$ be the projection $\mathbb{Z}[\check{G}_{\le l}]$ $\to$ $\mathbb{Z}[\check{G}_{b, d}]$.  Here, note that $\check{O}_{b, d}$ is a linear map.  By the definition, we immediately have:
\begin{lemma}\label{lem_relator*}
If $\cineq$, then for any $r^* \in \mathbb{Z}[\check{G}_{\le l}]$, 
\[
\tilde{x}^*_i (r^*) = \tilde{x}^*_i (\check{O}_{b, d} (r^*)).
\]
\end{lemma}
\begin{notation}\label{orep2}
Let $\cRep(b, d)$ $=$ $\check{O}_{b, d}(\cRep)$.  
\end{notation}

By using Lemma~\ref{lem_relator*}, we have the next proposition.   
\begin{proposition}\label{relator_prop*}
For each pair of integers $b$ and $d$ $(2 \le b \le d)$, let $\cNsum \cat$ be a function as in Definition~\ref{def_xiv}.  
For $(\epsilon_1, \epsilon_2, \epsilon_3, \epsilon_4, \epsilon_5)$ $\in \{ 0, 1 \}^5$, let $\cRep$ be the set as in Definition~\ref{dfn_relator2}.
The following two statements are equivalent: 
\begin{enumerate}
\item $\cNsum \alpha_i \tilde{x}^*_i (r^*) = 0 \quad (\forall r^* \in \cRep)$.\label{s1*}
\item $\cNsum \alpha_i \tilde{x}^*_i (r^*) = 0 \quad (\forall r^* \in \cRep(b, d))$.\label{s2*}
\end{enumerate}
\end{proposition}
\begin{proof}
\begin{itemize}
\item $(\ref{s1*}) \Rightarrow (\ref{s2*})$.  This is obvious.  
\item $(\ref{s1*}) \Leftarrow (\ref{s2*})$. Let $r^* \in \cRep$.  By Lemma~\ref{lem_relator*},  
\[
\cNsum \cat (r^*) = \cNsum \cat (\check{O}_{b, d}(r^*)).  
\]
By the condition (\ref{s2*}), 
\[
\cNsum \cat (\check{O}_{b, d}(r^*)) = 0.  
\]
Then, 
\[
\cNsum \cat (r^*) = 0.  
\]
\end{itemize}
\end{proof}

Recall that $\rev$ is the map $\hat{2n}$ $\to$ $\hat{2n}$ such that $\rev(p)$ $\equiv$ $-p+1$ (mod~$2n$).  
Also recall that we have fixed one to one correspondence between the set of isomorphism classes of the oriented Gauss words and $\{ {x}^*_i \}_{i \in \mathbb{N}}$.  
\begin{definition}
We say that the pair $(\cxi, \cxj)$ is a \emph{mirroring} pair if there exist oriented Gauss words $v^*$, $w^*$ such that $[[v^*]]=\cxi$, $[[w^*]]=\cxj$, and $v^*$ $=$ $w^* \circ \rev$.         
We say that the mirroring pair $(\cxi, \cxj)$ is \emph{reflective} if $\cxi(P^\epsilon)$ $=$ $\cxj(P^\epsilon)$ for any oriented spherical curve $P^\epsilon$.        
\end{definition}
\begin{example}
One may think that $(\cxi, \cxj)$ is reflective if and only if $\cxi$ and $\cxj$ would be equivalent.  However, it is false.  For example, a counterexample is given by Fact~\ref{fact1}, which is a special case of \cite[Page~451]{PV}.   It is easier to show that $\ahc(P)$ $=$ $\ahd(P)$ for every spherical curve $P$ by the  argument similar to \cite[Lemma~1]{Based_Ito}.   
\end{example}
\begin{fact}[Polyak-Viro]\label{fact1}
For every spherical curve $P$, $\aha(P)$ $=$ $\ahb(P)$.  
\end{fact}
\begin{remark}
By Remark~\ref{remark2}, we see that for any mirroring pair $(\cxi, \cxj)$ and oriented spherical curve $P^{\epsilon}$, we have $x^{*}_i (P^{\epsilon})$ $=$ $x^{*}_j (P^{- \epsilon})$.  
\end{remark}
\subsection{Main result and corollaries}\label{rr2_main}
\begin{theorem}\label{gg_thm2}
Let $b$ and $d$ $(2 \le b \le d)$ be integers and $\check{G}_{\le d}$, $\{ x^*_i \}_{i \in \mathbb{N}}$, $\mathbb{Z}[\check{G}_{\le l}]$, $\check{n}_d$ $=$ $|\check{G}_{\le d}|$, and $\check{G}_{b, d}$ $=$ $\{ x^*_i \}_{\check{n}_{b-1} +1 \le i \le \check{n}_{d}}$ be as in Subsection~\ref{ss_def}.     
Let $\cFsum$ and $\caFsum$ be functions as in Definition~\ref{def_xiv}.  
For $(\epsilon_1, \epsilon_2, \epsilon_3, \epsilon_4, \epsilon_5)$ $\in \{ 0, 1 \}^5$, let $\cRep(b, d)$ be as in Subsection~\ref{sec_main_result_2}.  
Suppose the following conditions are satisfied: 

\noindent $\bullet$ If $\epsilon_1 =1$, $\displaystyle \caFsum(r^*)=0$ for each $r^* \in \check{R}_{10000}(b, d)$.

\noindent $\bullet$ If $\epsilon_2 =1$, $\displaystyle \caFsum(r^*)=0$ for each $r^* \in \check{R}_{01000}(b, d)$.

\noindent $\bullet$ If $\epsilon_3 =1$, $\displaystyle \caFsum(r^*)=0$ for each $r^* \in \check{R}_{00100}(b, d)$.

\noindent $\bullet$ If $\epsilon_4 =1$, $\displaystyle \caFsum(r^*)=0$ for each $r^* \in \check{R}_{00010}(b, d)$.

\noindent $\bullet$ If $\epsilon_5 =1$, $\displaystyle \caFsum(r^*)=0$ for each $r^* \in \check{R}_{00001}(b, d)$.

Then, $\displaystyle \cFsum$
is an integer-valued invariant of oriented spherical curves under the deformations  corresponding to $\epsilon_j =1$.   

If, in addition to the above, for each mirroring pair $(\cxi, \cxj)$, one of the following two conditions is satisfied: $(1)$ $\alpha_i = \alpha_j$  or $(2)$ $(\cxi, \cxj)$ is reflective, then $\cFsum$ is an integer-valued invariant of spherical curves under the deformations  corresponding to $\epsilon_j =1$.   
\end{theorem}
\begin{definition}[irreducible arrow diagram]
Let $x^*$ be an arrow diagram.  An arrow $\alpha$ in $x^*$ is said to be an {\it isolated} arrow if $\alpha$ does not intersect any other arrow.  
An arrow diagram $x^*$ is said to be \emph{irreducible} if $x^*$ has no isolated arrow.           
The set of the irreducible arrow diagrams is denoted by $\check{\irr}$.  Let $\cirri$ $=$ $\{ i ~|~ \cineq, x^*_i \in \check{\irr} \}$. 
\end{definition}
\begin{example}\label{eg_23irr}
Under the notation of Example~\ref{eg_23}, we note that the set of irreducible elements of $\check{G}_{2, 3}$ consists of seven elements, those are, $\ax$, $\ahb$, $\ahd$, $\ahc$, $\aha$, $\atra$, and $\atrb$.     
\end{example}
Let $\check{O}_{\irr}$ be the projection $\mathbb{Z}[\check{G}_{\le l}]$ $\to$ $\mathbb{Z}[\check{G}_{\le l} \cap \check{\irr}]$.
\begin{lemma}\label{lemma3}
Let $x^*$ $\in$ $\check{\irr}$ and let $r^*$ $\in$ $\check{R}_{\epsilon_1 \epsilon_2 \epsilon_3 \epsilon_4 \epsilon_5}(b, d)$.  
\[
\tilde{x}^* (r^*) = \tilde{x}^* (\check{O}_{\irr}(r^*)).  
\]
\end{lemma}
If we consider the function of the form $\cFsumIrr$ for $\cFsum$ in Theorem~\ref{gg_thm2}, we obtain:  
\begin{corollary}\label{g_thm2}
Let $b$ and $d$ $(2 \le b \le d)$ be integers and $\check{G}_{\le d}$, $\{ x^*_i \}_{i \in \mathbb{N}}$, $\mathbb{Z}[\check{G}_{\le l}]$, $\check{n}_d$ $=$ $|\check{G}_{\le d}|$, and $\check{G}_{b, d}$ $=$ $\{ x^*_i \}_{\check{n}_{b-1} +1 \le i \le \check{n}_{d}}$ be as in Subsection~\ref{ss_def}.     
Let $\cFsumIrr$ and $\caFsumIrr$ be functions as in Definition~\ref{def_xiv}.  
For $(\epsilon_1, \epsilon_2, \epsilon_3, \epsilon_4, \epsilon_5)$ $\in \{ 0, 1 \}^5$, let $\cRep(b, d)$ be as in Subsection~\ref{sec_main_result_2}.      
Suppose the following conditions are satisfied: 

\noindent $\bullet$ If $\epsilon_2 =1$, $\displaystyle \caFsumIrr(r^*)=0$ for each $r^* \in \check{R}_{01000}(b, d)$. 

\vs 

\noindent $\bullet$ If $\epsilon_3 =1$, $\displaystyle \caFsumIrr(r^*)=0$ for each $r^* \in \check{R}_{00100}(b, d)$.

\vs

\noindent $\bullet$ If $\epsilon_4 =1$, $\displaystyle \caFsumIrr(r^*)=0$ for each $r^* \in \check{R}_{00010}(b, d)$.  

\vs

\noindent $\bullet$ If $\epsilon_5 =1$, $\displaystyle \caFsumIrr(r^*)=0$ for each $r^* \in \check{R}_{00001}(b, d)$.   

Then, 
$\displaystyle \cFsumIrr$
is an integer-valued invariant of oriented spherical curves under RI and the deformations  corresponding to $\epsilon_j =1$.    

If, in addition to the above, for each mirroring pair $(\cxi, \cxj)$, one of the following two conditions is satisfied: $(1)$ $\alpha_i = \alpha_j$  or $(2)$ $(\cxi, \cxj)$ is reflective, then $\cFsumIrr$ is an integer-valued invariant of spherical curves under RI and the deformations  corresponding to $\epsilon_j =1$.   
\end{corollary}
\noindent{\bf{Proof of Corollary~\ref{g_thm2} from Theorem~\ref{gg_thm2}.}}
By Theorem~\ref{gg_thm2}, it is enough to show
\begin{equation}\label{eq2_color}
\caFsumIrr (r^*)=0  \quad (\forall r^* \in \check{R}_{10000} (b, d))
\end{equation}
for a proof of Corollary~\ref{g_thm2}.  We first note that if $x^*_i \in \check{\irr}$, then $x^*_i$ has no isolated chords.  On the other hand, let $r^* \in \check{R}_{10000}(b, d)$, that is, there exist an oriented Gauss word $S$ and a letter $j$ such that $r^* =[[Sj\bar{j}]]$ or $[[S\bar{j}j]]$.  Then, the chord corresponding to $j$ is isolated.  
These show that $\tilde{x}^*_i (r^*)=0$.  This shows that (\ref{eq2_color}) holds.    
This fact together with Theorem~\ref{gg_thm2} immediately gives Corollary~\ref{g_thm2}.    
$\hfill \Box$

\begin{definition}[connected arrow diagram]
Let $v^*$ be an oriented Gauss word of length $2m$ and $w^*$ an oriented Gauss word of length $2n$ where $v^*(\hat{2m}) \cap w^*(\hat{2n})$ $=$ $\emptyset$.  
Then we define the oriented Gauss word of length $2(m+n)$, denoted by $v^*w^*$, by $v^*w^*(i)=v^*(i)$ ($1 \le i \le 2m$) and $v^*w^*(2m+i)=w^*(i)$ ($1 \le i \le 2n$).  The arrow diagram $[[v^*w^*]]$ is called a {\it{product}} of arrow diagrams $[[v^*]]$ and $[[w^*]]$.  If an arrow diagram is not a product of two non-empty arrow diagrams, then the arrow diagram is called a {\it{connected}} arrow diagram.  The set of the connected arrow diagrams is denoted by $\check{\conn}$.  Let $\cconni$ $=$ $\{ i ~|~ \cineq, x^*_i \in \check{\conn} \}$.
\end{definition}
It is easy to see that if a connected arrow diagram has at least two oriented chords, then it is irreducible.  Note that the arrow diagram consisting of exactly one chord is connected but not irreducible. 

Recall that $\check{\mathcal{C}}$ is the set of the ambient isotopy classes of the oriented spherical curves. 
\begin{definition}[connected sum]\label{additivity_def_color}
Let $P^{\epsilon_1}_1$ and $P^{\epsilon_2}_2$ $\in\check{\mathcal{C}}$.  We suppose that $P^{\epsilon_1}_1$ and $P^{\epsilon_2}_2$ are obtained from $P_1$ and $P_2$, respectively ($(\epsilon_1, \epsilon_2)$ $\in$ $\{ +, - \}^2$) and suppose that the ambient $2$-spheres are oriented.  
Let $p_i$ be a point on $P^{\epsilon_i}_i$ such that $p_i$ is not a double point ($i=1, 2$).  Let $d_i$ be a sufficiently small disk with center $p_i$ ($i=1, 2$) such that $d_i \cap P^{\epsilon_i}_i$ consists of an arc properly embedded in $d_i$.  Let $\hat{d}_i$ $=$ $cl(S^2 \setminus d_i)$ and ${\hat{P}}^{\epsilon_i}_i$ $=$ $P^{\epsilon_i}_i \cap \hat{d}_i$; let $h :$ $\partial \hat{d}_1$ $\to$ $\partial \hat{d}_2$ be an orientation reversing homeomorphism such that $h(\partial \hat{P}_1)$ $=$ $\partial \hat{P}_2$ such that $h($ the starting point of $\hat{P}^{\epsilon_1}_1 )$ $=$ the end point of $\hat{P}^{\epsilon_2}_2$ and $h($ the end point of $\hat{P}^{\epsilon_1}_1 )$ $=$ the starting point of $\hat{P}^{\epsilon_1}_1$.  Then, $\hat{P}^{\epsilon_1}_1 \cup_h \hat{P}^{\epsilon_2}_2$ gives an oriented spherical curve with the orientation induced from $P^{\epsilon_1}$ and $P^{\epsilon_2}$ in the oriented $2$-sphere $\hat{d}_1 \cup_h \hat{d}_2$.      
The spherical curve $\hat{P}^{\epsilon_1}_1 \cup_h \hat{P}^{\epsilon_2}_2$ in the oriented $2$-sphere, denoted by $P^{\epsilon_1}_1 \sharp_{(p_1, p_2)} P^{\epsilon_2}_2$, is called a \emph{connected sum} of the spherical curves $P^{\epsilon_1}_1$ and $P^{\epsilon_2}_2$ at the pair of points $p_1$ and $p_2$.  
\end{definition}
\begin{definition}[additivity]
Let $I$ be a function on $\check{\mathcal{C}}$.  We say that $I$ is \emph{additive} if $I(P^{\epsilon_1}_1 \sharp_{(p_1, p_2)} P^{\epsilon_2}_2)$ $=$ $I(P^{\epsilon_1}_1)$ $+$ $I(P^{\epsilon_2}_2)$ for any $P^{\epsilon_1}_1 \sharp_{(p_1, p_2)} P^{\epsilon_2}_2$ in Definition~\ref{additivity_def_color}.    
\end{definition}
If we consider the function of the form $\cFsumConn$ for $\cFsum$ in Theorem~\ref{gg_thm2}, we obtain:    
\begin{corollary}\label{cor3b}
Let $b$ and $d$ $(2 \le b \le d)$ be integers and $\check{G}_{\le d}$, $\{ x^*_i \}_{i \in \mathbb{N}}$, $\mathbb{Z}[\check{G}_{\le l}]$, $\check{n}_d$ $=$ $|\check{G}_{\le d}|$, and $\check{G}_{b, d}$ $=$ $\{ x^*_i \}_{\check{n}_{b-1} +1 \le i \le \check{n}_{d}}$ be as in Subsection~\ref{ss_def}.     
Let $\cFsumConn$ and $\caFsumConn$ be functions as in Definition~\ref{def_xiv}.  
For $(\epsilon_1, \epsilon_2, \epsilon_3, \epsilon_4, \epsilon_5)$ $\in \{ 0, 1 \}^5$, let $\Rep(b, d)$ be as in Subsection~\ref{sec_main_result_2}.  
Suppose the following conditions are satisfied: 

\noindent $\bullet$ If $\epsilon_2 =1$, $\displaystyle \caFsumConn(r^*)=0$ for each $r^* \in \check{R}_{01000}(b, d)$.

\noindent $\bullet$ If $\epsilon_3 =1$, $\displaystyle \caFsumConn(r^*)=0$ for each $r^* \in \check{R}_{00100}(b, d)$.

\noindent $\bullet$ If $\epsilon_4 =1$, $\displaystyle \caFsumConn(r^*)=0$ for each $r^* \in \check{R}_{00010}(b, d)$.

\noindent $\bullet$ If $\epsilon_5 =1$, $\displaystyle \caFsumConn(r^*)=0$ for each $r^* \in \check{R}_{00001}(b, d)$.   

Then, 
$\displaystyle \cFsumConn$
is an integer-valued additive invariant of oriented spherical curves under RI and the deformations corresponding to $\epsilon_j = 1$.

If, in addition to the above, for each mirroring pair $(\cxi, \cxj)$, one of the following two conditions is satisfied: $(1)$ $\alpha_i = \alpha_j$  or $(2)$ $(\cxi, \cxj)$ is reflective, then $\cFsumConn$ is an integer-valued invariant of spherical curves under RI and the deformations corresponding to $\epsilon_j =1$.
   
\end{corollary}
\noindent{\bf{Proof of Corollary~\ref{cor3b} from Theorem~\ref{gg_thm2}.}} 
Since $b \ge 2$ and $i \ge \check{n}_{b-1} +1$,  each $x^*_i$ consists of more than one chords, i.e., $x^*_i \in \check{\irr}$ (see the note preceding Definition~\ref{additivity_def_color}). 

By Theorem~\ref{gg_thm2}, it is enough to show
\begin{equation}\label{eq4_color}
\caFsumConn (r^*)=0  \quad (\forall r^* \in \check{R}_{10000} (b, d))
\end{equation}
for a proof of Corollary~\ref{cor3b}.  We first note that if $x^*_i \in \check{\irr}$, then $x^*_i$ has no isolated chords.  On the other hand, let $r^* \in \check{R}_{10000}(b, d)$, that is, there exist an oriented Gauss word $S$ and a letter $j$ such that $r^* = [[Sj\bar{j}]]$ or $[[S\bar{j}j]]$.  Then, the chord corresponding to $j$ is isolated.  
These show that $\tilde{x}^*_i (r^*)=0$.  This shows that (\ref{eq4_color}) holds.

Further, by using geometric observations as in Example~\ref{example2}, it is clear to show that if $x^*_i \in \check{\conn}$, then 
\begin{equation}\label{eq4c}
{x^*_i}(P^{\epsilon_1}_1 \sharp_{(p_1,~p_2)} P^{\epsilon_2}_2) = {x^*_i}(P^{\epsilon_1}_1) + {x^*_i}(P^{\epsilon_2}_2)
\end{equation}
for any $P^{\epsilon_1}_1 \sharp_{(p_1,~p_2)} P^{\epsilon_2}_2$.  This fact implies $\caFsumConn$ is additive.  
These facts together with Theorem~\ref{gg_thm2} immediately give Corollary~\ref{cor3b}.   
$\hfill\Box$

\section{Proof of Theorem~\ref{gg_thm2}.}

$\bullet$ (Proof of Theorem~\ref{gg_thm2} for the case $\epsilon_1 =1$.) 
Let $P$ and $P'$ be two oriented spherical curves where $P^{\epsilon}$ is related to ${P'}^{\epsilon}$ by a single RI, hence, there exist an oriented letter $i$ and an oriented Gauss word $S$ such that $AD_{P^{\epsilon}}$ $=$ $[[Si\bar{i}]]$ or $[[S\bar{i}i]]$ and $AD_{{P'}^{\epsilon}}$ $=$ $[[S]]$.  Since the arguments are essentially the same, we may suppose, without loss of generality, $AD_{P^{\epsilon}}$ $=$ $[[Si\bar{i}]]$. 
As we observed just before Definition~\ref{dfn_relator2}, we have
\begin{align*}
\cFsum({P}^\epsilon) &-\cFsum({P'}^\epsilon) \\
&= \cNsum \sum_{z_0 \in \sub^{(0)}(G^*)} \cat \left( [[z_0 i\bar{i}]] \right).
\end{align*}
By the assumption of this case, for each $z_0 \in \sub^{(0)}(G^*)$, 
\[
\cNsum \cat \left( [[z_0 i\bar{i}]] \right)=0
\]
and this shows that 
\[
\cFsum({P}^\epsilon) = \cFsum({P'}^\epsilon).  
\]

$\bullet$ (Proof of the case $\epsilon_2 =1$.)  Let $P^{\epsilon}$ and ${P'}^\epsilon$ be two oriented spherical curves where $P^{\epsilon}$ is related to ${P'}^\epsilon$ by a single strong~RI\!I, hence, there exist two oriented Gauss words $G=Si\bar{j}Tj \bar{i}$ (or $S\bar{i}jT\bar{j}i$) and $G'=STU$ corresponding to $P^{\epsilon}$ and ${P'}^\epsilon$, respectively, i.e., $AD_{P^\epsilon}$ $=$ $[[Si\bar{j}Tj\bar{i}]]$ (or $AD_{P^{\epsilon}}$ $=$ $[[S\bar{i}jT\bar{j}i]]$) and $AD_{{P'}^\epsilon}$ $=$ $[[ST]]$.  
First, suppose $AD_{P^\epsilon}$ $=$ $[[Si\bar{j}Tj\bar{i}]]$. 
  
By (\ref{tilde_x*}) in Subsection~\ref{ss_def}
 and (\ref{sub*}) in Subsection~\ref{sec_main_result_2}, we obtain (note that $\sub^{(3)} (G^*)$ $=$ $\emptyset$): 

\begin{align*}
&\cFsum (P^{\epsilon}) = \scFsum \left(\sum_{z^* \in \sub(G^*)} \tilde{x}^*_i (z^*) \right) \\
&= \scFsum \left(  \sum_{z^*_0 \in \sub^{(0)}(G^*)} \tilde{x}^*_i (z^*_0) \right)\\
&\qquad\qquad\qquad + \cNsum \sum_{z^*_{12} \in \sub^{(1)}(G^*) \cup \sub^{(2)} (G^*)} \alpha_i \tilde{x}^*_i (z^*_{12})\\
&= \scFsum \left(  \sum_{{z'}^* \in \sub({G'}^*)} \tilde{x}^*_i ({z'}^*) \right)\\
&\qquad\qquad\qquad + \cNsum \sum_{z^*_{12} \in \sub^{(1)}(G^*) \cup \sub^{(2)} (G^*)} \alpha_i \tilde{x}^*_i (z^*_{12})
\end{align*}
($\because$ $\sub^{(0)}(G^*)$ is identified with $\sub({G'}^*)$).

Let $z^*_0 \in \sub^{(0)}(G^*)$.
We note that since $G$ is an oriented Gauss word $z_0$ uniquely admits a decomposition into two sub-words, which are sub-words on $S$ and $T$.
Let $\sigma(z^*_0)$ be the sub-word of $S$ and $\tau(z^*_0)$ the sub-word of $T$ satisfying $z^*_0$ $=$ $\sigma(z^*_0)\tau(z^*_0)$ (for the definition of sub-words, see Definition~\ref{ss_def}).  Under these notations, we define maps
\begin{align*}
z^*_2: &\sub^{(0)}(G^*) \to \sub^{(2)}(G^*); z^*_2 (z^*_0) = \sigma(z^*_0)i\bar{j}\tau(z^*_0)j\bar{i},\\
z^*_1: &\sub^{(0)}(G^*) \to \sub^{(1)}(G^*); z^*_1 (z^*_0) = \sigma(z^*_0)i\tau(z^*_0)\bar{i}~{\textrm{, and}}\\
{z'}^*_1: &\sub^{(0)}(G^*) \to \sub^{(1)}(G^*); {z'}^*_1 (z^*_0) = \sigma(z^*_0)\bar{j}\tau(z^*_0)j.    
\end{align*}
Then, it is easy to see that $\sub^{(1)} (G^*) \cup \sub^{(2)} (G^*)$ admits a decomposition 
\begin{align*}
&\sub^{(1)} (G^*) \cup \sub^{(2)} (G^*) \\
&=  \{ z^*_1 (z^*_0) ~|~ \forall z^*_0 \in \sub^{(0)} (G^*) \} \amalg \{ {z'}^*_1 (z^*_0) ~|~ \forall z^*_0 \in \sub^{(0)} (G^*) \} \amalg \{ z^*_2 (z^*_0) ~|~ \forall z^*_0 \in \sub^{(0)}(G^*) \}.  
\end{align*}
These notations together with the above give: 
\begin{align*}
\cFsum(P^{\epsilon}) &= \cFsum ({P'}^{\epsilon}) \\
& + \cNsum \sum_{z^*_0 \in \sub^{(0)}(G^*)} \cat (z^*_2 (z^*_0) + z^*_1 (z^*_0) + {z'}^*_1 (z^*_0))\\
&= \cFsum ({P'}^{\epsilon}) \\
&+ \cNsum \sum_{z^*_0 \in \sub^{(0)}(G^*)} \cat ([z^*_2 (z^*_0)] + [z^*_1 (z^*_0)] + [{z'}^*_1 (z^*_0)]).
\end{align*}
Here, we note that by the condition for the case $\epsilon_2$ $=$ $1$, for any for any $z^*_0 \in \sub^{(0)}(G^*)$, 
\[
\cNsum \cat ([z^*_2 (z^*_0)] + [z^*_1 (z^*_0)] + [{z'}^*_1 (z^*_0)]) = 0.
\]
Here, one may think that
\[
\tilde{x}^*_i ([z^*_1 (z^*_0)] + [{z'}^*_1 (z^*_0)] + [z^*_2 (z^*_0)]) = 0
\]
for each $z^*_0$ by the condition for the case $\epsilon_2$ $=$ $1$.   
However, the condition says that the equation holds for each $r^* \in \check{R}_{01000}(b, d)$, and we note that 
$[z^*_1 (z^*_0)]$ $+$ $[{z'}^*_1 (z^*_0)]$ $+$ $[z^*_2 (z^*_0)]$ may not be an element of $\check{R}_{01000}(b, d)$ (possibly $\check{f}^{-1}([z^*_2 (z^*_0)]) > \check{n}_d$ or $\check{f}^{-1}([z^*_1 (z^*_0)]) \le \check{n}_{b-1}$), and $\check{O}_{b, d} ([z^*_1 (z^*_0)] + [{z'}^*_1 (z^*_0)] + [z^*_2 (z^*_0)])$ $\neq 0$.  
However, even when this is the case we see that 
\[
\tilde{x}^*_i ([z^*_1 (z^*_0)] + [{z'}^*_1 (z^*_0)] + [z^*_2 (z^*_0)]) = 0
\]by Proposition~\ref{relator_prop*}. 

Thus,
\begin{align*}
\cFsum (P^{\epsilon}) = \cFsum ({P'}^{\epsilon}).
\end{align*}

The proof for the case when $AD_{P^\epsilon}$ $=$ $[[S\bar{i}jT\bar{j}iU]]$ can be carried out as above, and omit it.  

$\bullet$ (Proof of the case $\epsilon_3 =1$.)  
Since the arguments are essentially the same as that of the case $\epsilon_2$ $=$ $1$, we omit this proof. 

$\bullet$ (Proof of the case $\epsilon_4 =1$.)  Let $P^{\epsilon}$ and ${P'}^\epsilon$ be two spherical curves where $P^{\epsilon}$ is related to ${P'}^\epsilon$ by a single strong~RI\!I\!I, hence, there exist two oriented Gauss words $G=S\bar{i}jT\bar{k}iU\bar{j}k$ and $G'=Sj\bar{i}Ti\bar{k}Uk\bar{j}$ corresponding to $P^{\epsilon}$ and ${P'}^\epsilon$, respectively, i.e., $AD_{P^\epsilon}$ $=$ $[[S\bar{i}jT\bar{k}iU\bar{j}k]]$ and $AD_{{P'}^\epsilon}$ $=$ $[[Sj\bar{i}Ti\bar{k}Uk\bar{j}]]$.   
 
By (\ref{tilde_x*}) in Subsection~\ref{ss_def}
 and (\ref{sub*}) in Subsection~\ref{sec_main_result_2}, we obtain: 
\begin{align*}
&\cFsum (P^{\epsilon}) = \scFsum \left( \sum_{z^* \in \sub(G^*)} \tilde{x}^*_i (z^*) \right)  \\
&= \scFsum \left( \sum_{z^*_{01} \in \sub^{(0)}(G^*) \cup \sub^{(1)}(G^*)} \tilde{x}^*_i (z^*_{01}) 
+ \sum_{z^*_{23} \in \sub^{(2)}(G^*) \cup \sub^{(3)}(G^*)} \tilde{x}^*_i (z^*_{23})
 \right).  
\end{align*}
and
\begin{align*}
&\cFsum ({P'}^{\epsilon}) = \scFsum \left( \sum_{{z'}^* \in \sub({G'}^*)} \tilde{x}^*_i ({z'}^*) \right)  \\
&= \scFsum \left( \sum_{{z'}^*_{01} \in \sub^{(0)}(G^*) \cup \sub^{(1)}(G^*)} \tilde{x}^*_i ({z'}^*_{01}) 
+ \sum_{z^*_{23} \in \sub^{(2)}(G^*) \cup \sub^{(3)}(G^*)} \tilde{x}^*_i ({z'}^*_{23})
 \right).
\end{align*}

Since $\sub^{(0)}(G^*)$ ($\sub^{(1)}(G^*)$ resp.) is naturally identified with $\sub^{(0)}({G'}^*)$ ($\sub^{(1)}({G'}^*)$ resp.), the above equations show: 
\begin{align*}
&\cFsum (P^{\epsilon}) - \cFsum ({P'}^{\epsilon})\\ 
&= \cNsum \sum_{z^*_{23} \in \sub^{(2)}(G^*) \cup \sub^{(3)} (G^*)} \cat (z^*)\\
&\qquad - \cNsum \sum_{{z'}^*_{23} \in \sub^{(2)}({G'}^*) \cup \sub^{(3)} ({G'}^*)} \cat ({z'}^*).   
\end{align*}

Let $z^*_0 \in \sub^{(0)}(G^*)$, which is identified with $\sub^{(0)}({G'}^*)$. 
We note that since $G^*$ is an oriented Gauss word $z^*_0$ uniquely admits a decomposition into three sub-words, which are sub-words on $S$, $T$, and $U$. 
Let $\sigma(z^*_0)$ be the sub-word of $S$, $\tau(z^*_0)$ the sub-word of $T$, and $\mu(z^*_0)$ the sub-word of $U$ satisfying $z^*_0$ $=$ $\sigma(z^*_0)\tau(z^*_0)\mu(z^*_0)$.  We define maps 
\begin{align*}
z^*_3: &\sub^{(0)}(G^*) \to \sub^{(3)}(G^*); z^*_3 (z^*_0) = \sigma(z^*_0) \bar{i}j \tau(z^*_0) \bar{k}i \mu(z^*_0) \bar{j}k,\\ 
z^*_{2a}: &\sub^{(0)}(G^*) \to \sub^{(2)}(G^*); z^*_{2a} (z^*_0) = \sigma(z^*_0) \bar{i}j \tau(z^*_0) i \mu(z^*_0) \bar{j}, \\
z^*_{2b}: &\sub^{(0)}(G^*) \to \sub^{(2)}(G^*); z^*_{2b} (z^*_0) = \sigma(z^*_0) \bar{i} \tau(z^*_0) \bar{k}i \mu(z^*_0) k,~{\text{and}}~\\
z^*_{2c}: &\sub^{(0)}(G^*) \to \sub^{(2)}(G^*); z^*_{2c} (z^*_0) = \sigma(z^*_0) j \tau(z^*_0) \bar{k} \mu(z^*_0) \bar{j}k.
\end{align*}
Similarly, let 
\begin{align*}
{z'}^*_3: &\sub^{(0)}({G'}^*) \to \sub^{(3)}({G'}^*); {z'}^*_3 (z^*_0) = \sigma(z^*_0) j\bar{i} 
 \tau(z^*_0) i\bar{k} \mu(z^*_0) k\bar{j},\\
{z'}^*_{2a}: &\sub^{(0)}({G'}^*) \to \sub^{(2)}({G'}^*); {z'}^*_{2a} (z^*_0) = \sigma(z^*_0) j\bar{i} \tau(z^*_0) i \mu(z^*_0) \bar{j}, \\
{z'}^*_{2b}: &\sub^{(0)}({G'}^*) \to \sub^{(2)}({G'}^*); {z'}^*_{2b} (z^*_0) = \sigma(z^*_0) \bar{i} \tau(z^*_0) i\bar{k} \mu(z^*_0) k,~{\text{and}}~ \\
{z'}^*_{2c}: &\sub^{(0)}({G'}^*) \to \sub^{(2)}({G'}^*); {z'}^*_{2c} (z^*_0) = \sigma(z^*_0) j \tau(z^*_0) \bar{k} \mu(z^*_0) k\bar{j}.
\end{align*}
Then, it is easy to see that $\sub^{(2)} (G^*) \cup \sub^{(3)} (G^*)$ admits decompositions 
\begin{align*}
&\sub^{(2)} (G^*) \cup \sub^{(3)} (G^*) \\
&= \{ z^*_3 (z^*_0) ~|~ z^*_0 \in \sub^{(0)}(G^*) \}
\amalg \{ z^*_{2a} (z^*_0) ~|~ z^*_0 \in \sub^{(0)} (G^*) \} \\
& \amalg \{ z^*_{2b} (z^*_0) ~|~ z^*_0 \in \sub^{(0)} (G^*) \}  
\amalg \{ z^*_{2c} (z^*_0) ~|~ z^*_0 \in \sub^{(0)}(G^*) \}
\end{align*}
and 
\begin{align*}
& \sub^{(2)}({G'}^*) \cup \sub^{(3)} ({G'}^*) \\
&= \{ {z'}^*_3 (z^*_0) ~|~ z^*_0 \in \sub^{(0)}({G}^*) \}
\amalg \{ {z'}^*_{2a} (z^*_0) ~|~ z^*_0 \in \sub^{(0)} ({G}^*) \} \\
& \amalg \{ {z'}^*_{2b} (z^*_0) ~|~ z^*_0 \in \sub^{(0)} ({G}^*) \} 
\amalg \{ {z'}^*_{2c} (z^*_0) ~|~ z^*_0 \in \sub^{(0)} ({G}^*) \}.
\end{align*}
By using these notations, we obtain: 
\begin{align*}
&\cFsum (P^{\epsilon}) - \cFsum ({P'}^{\epsilon})
\\
&= \sum_{z^*_0 \in \sub^{(0)}(G^*)}  \caFsum \Big( (z^*_3 (z^*_0) + z^*_{2a} (z^*_0) + z^*_{2b} (z^*_0) + z^*_{2c} (z^*_0))\\ 
&- ({z'}^*_3(z^*_0) + {z'}^*_{2a}(z^*_0) + {z'}^*_{2b}(z^*_0) + {z'}^*_{2c}(z^*_0)) \Big) \\
&= \sum_{z^*_0 \in \sub^{(0)}(G^*)} \caFsum \Big(  ([z^*_3 (z^*_0)] + [z^*_{2a} (z^*_0)] + [z^*_{2b} (z^*_0)] + [z^*_{2c} (z^*_0)])\\ 
& - ([{z'}^*_3(z^*_0)] + [{z'}^*_{2a}(z^*_0)] + [{z'}^*_{2b}(z^*_0)] + [{z'}^*_{2c}(z^*_0)]) \Big).
\end{align*}
Here, we note that 

$([z^*_3 (z^*_0)] + [z^*_{2a} (z^*_0)] + [z^*_{2b} (z^*_0)] + [z^*_{2c} (z^*_0)]) - ([{z'}^*_3(z^*_0)]$ $+$ $[{z'}^*_{2a}(z^*_0)]$ $+$ $[{z'}^*_{2b}(z^*_0)]$ $+$ $[{z'}^*_{2c}(z^*_0)])$ $\in \check{R}_{00010}$.  

Hence, by the assumption of Case $\epsilon_4$ $=$ $1$ and by Proposition~\ref{relator_prop*} (cf.~Proof of the case $\epsilon_2$ $=$ $1$), for each $z^*_0$, 
\begin{align*}
&\caFsum  \Big(  ([z^*_3 (z^*_0)] + [z^*_{2a} (z^*_0)] + [z^*_{2b} (z^*_0)] + [z^*_{2c} (z^*_0)]) \\
&\qquad\qquad\qquad  - ([{z'}^*_3(z^*_0)] + [{z'}^*_{2a}(z^*_0)] + [{z'}^*_{2b}(z^*_0)] + [{z'}^*_{2c}(z^*_0)]) \Big)\\
&\qquad = 0.  
\end{align*}
These show that  
\[ \cFsum ({P}^{\epsilon}) = \cFsum ({P'}^{\epsilon}).
\]

The proof for the case when $AD_{P^\epsilon}$ $=$ $[[Sk\bar{j}Ti\bar{k}Uj\bar{i}]]$ can be carried out as above, and omit it.

$\bullet$ (Proof of the case $\epsilon_5 =1$.)  
Since the arguments are essentially the same as that of the case $\epsilon_4$ $=$ $1$, we omit this proof. 

Note that either condition (1) or (2) is satisfied, for each mirroring pair $(x^*_i, x^*_j)$, we have: 
\begin{align*}
\alpha_i x^*_i (P^{+}) + \alpha_j x^*_j (P^{+})
&= \alpha_i x^*_j (P^{-}) + \alpha_j x^*_i (P^{-}), \\
&(\because x^*_i (P^{\epsilon}) = x^*_j (P^{- \epsilon})~{\textrm{by Remark~\ref{remark2}}})\\
&= \alpha_j x^*_j (P^{-}) + \alpha_i x^*_i (P^{-})~{\textrm{(by Condition~(1) of Theorem~\ref{gg_thm2}}}) \\
\end{align*}
or
\[
x^*_i (P^{+}) = x^*_j (P^{-}) = x^*_i (P^{-})~{\text{(by Condition~(2) of Theorem~\ref{gg_thm2})}}.  
\]
Then $\cFsum (P^{+})$ $=$ $\cFsum (P^-)$ for each spherical curve $P^{\epsilon}$.  
$\hfill\Box$

\section{A method of a computation by a computer program}\label{takamura}
In order to obtain the following program, we use $C++11$ and Mathematica, and to compute a rank of a vector space (i.e., a kernel space) of  functions, each of which is invariant under RI and strong or weak RI\!I\!I.  
\subsection{Nomal Gauss words}\label{un}
In this section, we introduce normal Gauss words for (unoriented) Gauss words.    Since it is worth seeing the unoriented case before we see the oriented case in Section~\ref{oriented_com}, we first discuss it to obtain the set of Gauss words.  
\begin{definition}[normal Gauss word]
For every Gauss word $G$, each letter appears twice in $G$.  Then, the letter firstly (secondary, resp.) appearing is called the \emph{former} (\emph{latter}, resp.) letter.    
Let $G$ be a Gauss word of length $2n$.  Suppose that we read letters in $G$ from the left to the right, and the former letters are labelled in the order $1, 2, 3, \ldots, n$.  Then, the Gauss word $G$ is called a \emph{normal Gauss word}.   The set of normal Gauss words of length $2d$ is denoted by $N_{d}$.  
\end{definition}
\begin{example}
The Gauss word $121233$ ($1234255143$, resp.) is a normal Gauss word.  The Gauss word $131322$ ($2342551431$, resp.) is \emph{not} a normal Gauss word.   The Gauss word obtained by Definition~\ref{dfn_cdp} is a normal Gauss word.  
\end{example}

By using normal Gauss words, $N_{d}$ is obtained from $N_{d-1}$.   First, for an element of $N_{d-1}$, we replace each letter $i$ with $i+1$, and the resulting Gauss word is denoted by $w$.  
Second, we consider $2d$ blanks arranged in a line, and  we put the former of the letter ``$1$'' in the leftmost blank and consider every possibility of the position of the latter of the letter ``$1$''.   Then, we ignore the former and the latter of the letter ``$1$" and put $w$ into $2(d-1)$ ordered blanks.   The process obtains a normal Gauss word in $N_d$.  By the induction, for every $d$, we have $N_d$, which implies $G_{\le d}$ by giving $[w]$ (for the definition of an equivalence class $[w]$, Definition~\ref{ori_g}).    
\begin{figure}[h!]
\centering
\includegraphics[width=7cm]{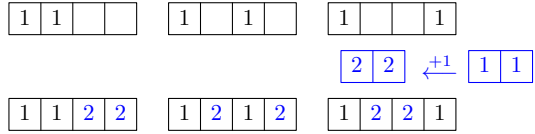}
\caption{$N_2$ is obtained by $N_1$.}
\label{n2}
\end{figure}

Next, we enumerate elements in $\check{R}_{00001}$.   Let us start with a toy model of a linear sum of unoriented Gauss words.  Let
$S$, $T$, and $U$ be sub-words where $STU$ is a Gauss word.  Then, let 
\begin{align*}
r& =\left( [Si\,jTikUjk] + [Si\,jTiUj] + [SiTikUk] + [SjTkUjk] \right) \\
- & \left( [Sj\, iTkiUkj] + [Sj\, iTiUj] + [SiTkiUk] + [SjTkUkj]  \right).
\end{align*}
It is clear that $r$ is obtained from a normal Gauss word of type $12Si1Ti2U$, as shown in Fig.~\ref{fig_t_E}, which is given by adding letters to $123132$.  
\begin{figure}[h!]
\begin{center}
\includegraphics[width=2cm]{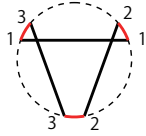}
\caption{The chord diagram corresponding to $12Si1Ti2U$ ($i=3$).  }
\label{fig_t_E}
\end{center}
\end{figure}
As a first example we add a single letter (corresponding to a chord as in Fig.~\ref{fig_t_G}) to $123132$.  Then, we have chord diagrams as in Fig.~\ref{fig_t_F}.
\begin{figure}[h!]
\begin{center}
\includegraphics[width=8cm]{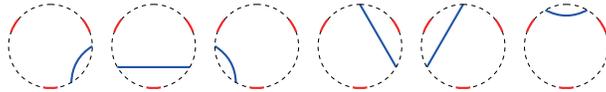}
\caption{The chord diagram corresponding to $123132$}
\label{fig_t_G}
\end{center}
\end{figure}
\begin{figure}[h!]
\begin{center}
\includegraphics[width=10cm]{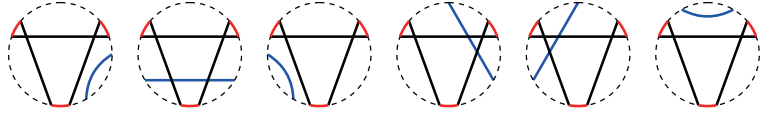}
\caption{Each chord added to  $[123132]$}
\label{fig_t_F}
\end{center}
\end{figure}
Such a word is regarded as a word consisting of  three dots and a Gauss words of length $2$ as follows.  
\begin{equation*}
  \begin{array}{cccccc}
   \rbullet \rbullet \rbullet \textcolor{blue}{4} \textcolor{blue}{4} &
    \rbullet \rbullet \textcolor{blue}{4} \rbullet  \textcolor{blue}{4} &
    \rbullet \rbullet \textcolor{blue}{4} \textcolor{blue}{4} \rbullet&
    \rbullet \textcolor{blue}{4} \rbullet \rbullet \textcolor{blue}{4} &
    \rbullet \textcolor{blue}{4} \rbullet \textcolor{blue}{4} \rbullet&
    \rbullet \textcolor{blue}{4} \textcolor{blue}{4} \rbullet  \rbullet
  \end{array}
 \end{equation*}
 
In general, we obtain a normal Gauss word of type $12Si1Ti2U$ by the following steps:

\begin{itemize}
\item (Step~1) The list of normal Gauss words of length $2(n-3)$ is given.  
\item  (Step~2) For a normal Gauss word $w$ of length $2(n-3)$, we send $m$ each letter to $m+1$.     
\item  (Step~3) List the all possibilities of immersions of three dots into $w$.  
\item  (Step~4) Replace a leftmost (center, rightmost, resp.) dot with two successive letters ``$12$" (``$31$", ``$32$", resp.).  
\item  (Step~5) Fix the normal Gauss word that is isomorphic to the resulting Gauss word.  
\end{itemize}
For every positive integer $n$ ($n \ge 4$), the process obtains the set $\{v~|~v$  is a normal Gauss word of type $12Si1Ti2U$ of length $2n$ $\}$.  

\subsection{Normal oriented Gauss words}\label{oriented_com}
Next, we see a computation using oriented Gauss words.  In our program, denote an oriented Gauss word $w^*$ by \verb+letter_list+ that is a sequence of letters corresponds to $w^*$.  In order to obtain a description of an orientation of letters, let   
\verb+letter_list[i]+ is $0$ ($1$, resp.) if $w^*(i)$ is a starting point (an end point, resp.).  
\begin{example}
Let $w$ be $\bar{1}2\bar{3}13\bar{2}$.  Then, it is clear that $\verb+letter_list+$ $=$ $\{1,2,3,1,3,2\}$, 
and $\verb+orientation_list+$ $=$ $\{1,0,1,0,0,1\}$.  
\end{example}
\begin{definition}
An oriented Gauss word $w^*$ is called a \emph{normal oriented Gauss word} if $w^*$ is a normal Gauss word.  
\end{definition}
The orientations of letters are enumerated by a lexicographic order in the following: 
for a Gauss word of length $2$, all possibilities are $\rightarrow$ and $\leftarrow$ that are represented by 
 $\{0,1\}$ and $\{1, 0\}$, respectively in our program.  

For a Gauss word of length $4$, all possibilities are
   \begin{center}
    \{\{0,1\}, \{0,1\}\},\quad
    \{\{0,1\}, \{1,0\}\},\qquad
    \{\{1,0\}, \{0,1\}\},\quad
    \{\{1,0\}, \{1,0\}\}.  
   \end{center}

\begin{example}
$w^* = \textcolor{red}{1}\textcolor{green}{2}\textcolor{blue}{3}132$ with the information of  orientations
    $\{ \{\textcolor{red}{1},0\}, \{\textcolor{green}{0},1\}, \{\textcolor{blue}{1},0\}\}$ 
    impliy
    $\verb+orientation_list+ = \{\textcolor{red}{1},\textcolor{green}{0}, \textcolor{blue}{1},0,0,1\}$.   
   \end{example}

Then, by applying the process (Step~1)--(Step~5) and the above lexicographic order with respect to $0$ and $1$, we have $\check{G}_{\le n}$ the set of isomorphism classes of oriented Gauss words of length $\le 2n$.  
Similar to an unoriented case (the toy model), we obtain $\{v~|~v$  is a normal oriented Gauss word of type $jkS\bar{i}\,\bar{j}Ti\bar{k}U$ or $\bar{j}\bar{k}Si\,jT\bar{i}kU$ of length $2n$ $\}$, which obtains $\check{R}_{00001}(2, n)$.    

\begin{example}[$\check{R}_{00010}(2,3)$]\label{strong23}
Every element in $\check{G}_{2, 3} \cap \conn$ is listed as follows.  
\begin{center}     
\includegraphics[height=1cm]{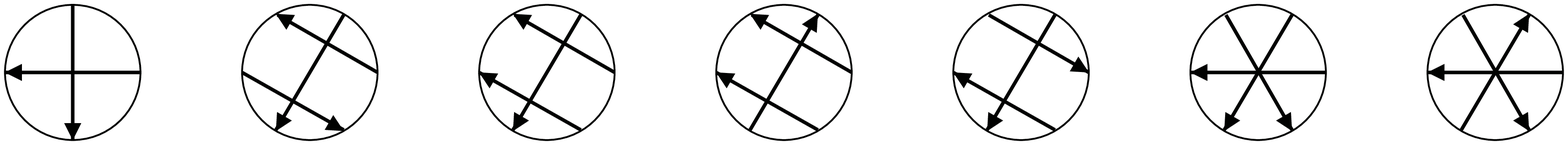}
\end{center}
Then, (if necessary,) we redefine $x^*_1$ $=$ $\ax$, 
$x^*_2$ $=$ $\ahb$, 
$x^*_3$ $=$ $\ahd$,
$x^*_4$ $=$ $\ahc$,
$x^*_5$ $=$ $\aha$,
$x^*_6$ $=$ $\atra$, and
$x^*_7$ $=$ $\atrb$.  
      Also every oriented Gauss word of type $\bar{j}kS\bar{i}\,jT\bar{k}iU$ or $\bar{j}iSk\,\bar{j}Ti\bar{k}U$ in $\check{R}_{00010}(2, 3)$ is listed as follows. 
       \begin{center}\includegraphics[height=1cm]{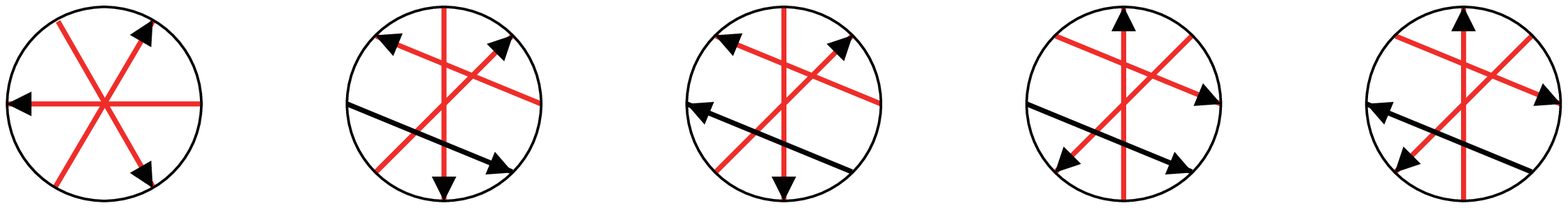}
       \end{center}
      Then, we list relators
      \begin{align*}
r^*_{s1} &= \atrb + \ax + \ax + \ax -  \ \bii -  \  \bchordth -    \bchordth -    \bchordth,\\
r^*_{s1a} &= \atrb + \ax + \ax + \ax - \ \bv   - \achordth - \achordth - \achordth, \\
r^*_{s2} &= \atra + \ahb + \ahb  -  \gchordtwo - \ahb - \gchordtwo, \\
r^*_{s3} &= \atrb + \ahd + \ahc - \achordtwo -\ahb - \echordtwo, \\
r^*_{s4} &= \atra + \aha + \aha - \dchordtwo - \aha - \dchordtwo,~{\textrm{and}}~ \\
r^*_{s5} &= \ahc + \atrb + \ahd - \aha - \bchordtwo - \fchordtwo.
      \end{align*}
      Here, note that $\check{O}_{\irr}(r^*_{s1})$ $=$ $\check{O}_{\irr}(r^*_{s1a})$ and Lemma~\ref{lemma3}.    
      Thus, letting $A$ $=$ $(x^*_i (r^*_{sj}))_{1 \le i \le 7, 1 \le j \le 5}$, we have
      \begin{equation*}
 A  = 
  \begin{bmatrix}
3&0&0&0&0\\
0&1&-1&0&0\\
0&0&1&1&0\\
0&0&1&1&0\\
0&0&0&-1&1\\
0&1&0&0&1\\
1&0&1&1&0
  \end{bmatrix}.
\end{equation*}
      Let $V_s(2,3) =  \{\xx\ |\ \xx A = \oo\}$.  
      It is elementary to show that $\dim V_w(2,3) = 3$ and the set of solutions is $\{\bf{x} =$ $\gamma_1 (1, 0, 3, 0, 0, 0, -3)$ $+ \gamma_2 (0, 1, 1, 0, 1, -1, 0)$ $+$ $\gamma_3 
(0, 0, 1, -1, 0, 0, 0)~|~\gamma_1, \gamma_2, \gamma_3 \in \mathbb{Z} \}$.  

By Corollary~\ref{cor3b}, the following functions are invariant of spherical curves under RI and strong~RI\!I\!I.
\begin{equation}\label{thm3}
\begin{split}
&\ax + 3 \ahd - 3 \atrb,\\
&\aha + \ahb + \ahc - \atra,~{\text{and}}~ \\
&\ahc - \ahd.  
\end{split}
\end{equation}

Note that \cite{IT3} introduced an invariant 
\[
\frac{1}{4} \left \{3 \left( \aha + \ahb + \ahc + \ahd \right) - 3 \left(\atra + \atrb \right) + \ax \right \}
\]
and it can be deduced from 
\begin{align*}
3 \left(\aha+\ahb+\ahc+\ahd \right) -3 \left( \atra + \atrb \right) + \ax\\
= 3 \left(\aha + \ahb + \ahc - \atra \right) + \left( \ax + 3 \ahd - 3 \atrb \right).
     \end{align*}
     \end{example}
\begin{example}[$\check{R}_{00001}(2,3)$]
Every element in $\check{G}_{2, 3} \cap \conn$ is listed as in Example~\ref{strong23}.  
Every oriented Gauss word of type $jkS\bar{i}\,\bar{j}Ti\bar{k}U$ or $\bar{j}\bar{i}Sk\,jT\bar{k}iU$ in $\check{R}_{00001}(2, 3)$ is listed as follows. 
       \begin{center}
       \includegraphics[width=12cm]{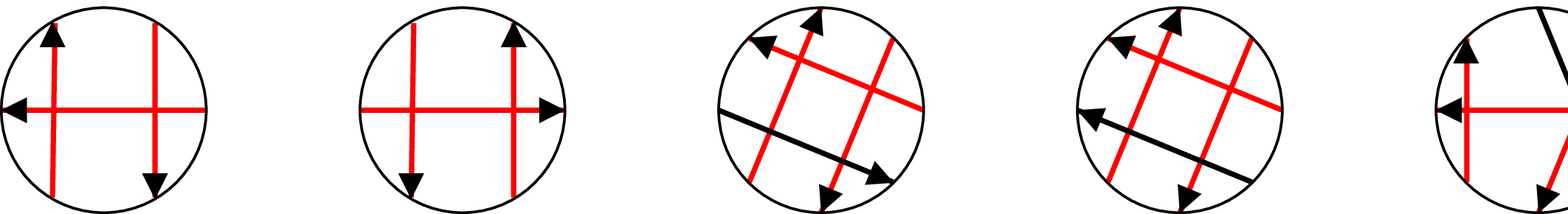}
       \end{center}

      Then, we list relators
      \begin{align*}
r^*_{w1} &= \aha + \achordth + \ax + \ax  - \cchordtwo - \bchordthb - \ax - \bchordthb,\\
r^*_{w2} &= \ahb + \bchordth + \ax + \ax - \hchordtwo - \bchordthb - \ax - \bchordthb, \\
r^*_{w3} &=  \aha+ \ahb + \ahb - \echordtwo - \atrb - \achordtwo, \\
r^*_{w4} &= \aha+ \ahc+ \ahd- \cchordtwo - \atra - \cchordtwo, \\
r^*_{w5} &= \dchordtwo + \aha + \atra - \fchordtwo - \aha - \ahd,\\
r^*_{w6} &= \fchordtwo + \ahd + \atra - \hchordtwo - \ahd - \ahc, \\
r^*_{w7} &= \bchordtwo + \atra + \ahc - \ahd - \ahc  - \hchordtwo, \\
r^*_{w8} &= \dchordtwo + \atra + \aha - \ahc - \aha - \bchordtwo, \\
r^*_{w9} &= \ahb + \ahc + \ahd - \hchordtwo - \atra - \hchordtwo,\\
r^*_{w10} &= \ahb + \aha + \aha - \bchordtwo - \atrb - \fchordtwo, \\
r^*_{w11} &= \achordtwo + \ahd + \atra - \cchordtwo - \ahd - \ahc, \\
r^*_{w12} &= \gchordtwo + \ahb + \atra - \achordtwo - \ahb - \ahd, \\
r^*_{w13} &=  \gchordtwo + \atra + \ahb - \ahc - \ahb  - \echordtwo,~{\textrm{and}}~ \\
r^*_{w14} &= \echordtwo + \atra + \ahc - \ahd - \ahc  - \cchordtwo.
      \end{align*}
      Thus, letting $A$ $=$ $(x^*_i (r^*_{wj}))_{1 \le i \le 7, 1 \le j \le 14}$, we have
      \begin{equation*}
 A =
  \begin{bmatrix}
1&1&0&0&0&0&0&0&0&0&0&0&0&0\\
0&1&2&0&0&0&0&0&1&1&0&0&0&0\\
0&0&0&1&-1&0&-1&0&1&0&0&-1&0&-1\\
0&0&0&1&0&-1&0&-1&1&0&-1&0&-1&0\\
1&0&1&1&0&0&0&0&0&2&0&0&0&0\\
0&0&0&-1&1&1&1&1&-1&0&1&1&1&1\\
0&0&-1&0&0&0&0&0&0&-1&0&0&0&0
  \end{bmatrix}
\end{equation*}
      Let $V_w(2,3) =  \{\xx\ |\ \xx A = \oo\}$.  It is elementary to show that $\dim V_w(2,3)$ $=$ $1$, and the set of the solutions is $\{ {\bf{x}}$ $=$ $\gamma (-1, 1, -1, -1, 1, -1, 3)$ $|~\gamma \in \mathbb{Z}\}$.  
Thus, \[ 
-\ax + \ahb - \ahc - \ahd - \aha - \atra + 3 \atrb
\]
is a weak (1, 3) homotopy invariant of oriented spherical curves.  
However, in \cite{gpv}, Goussarov, Polyak, Viro reviewed that it vanishes on the spherical curves (see \cite[Section~3.2, Page~1058]{gpv}).  
\end{example}
\begin{example}
Let $V^\text{ori}_s(b,b+1)$ ($V^\text{ori}_w(b,b+1)$, resp.) be the kernel space of each matrix $(x^*_i (r^*_j))$ with appropriate induces obtained by $x^*_i \in \check{G}_{b, b+1}$ and $r^*_j \in \check{R}_{00010}(b, b+1)$ ($\check{R}_{00001}(b, b+1)$, resp.).     
     \begin{equation*}
       \begin{array}{|c|c|c|c|c|c|}
	 \hline
	  b& 2 & 3 & 4 & 5 \\
	 \hline
	  \dim V_s^\text{ori}(b,b+1) & 3 & 18 & 145 &   \\
	 \hline	  	  
	  \dim V_w^\text{ori}(b,b+1) & 1 & 3 & 13 & 31  \\
	 \hline	  
       \end{array}
     \end{equation*}
 Surprisingly, if we proceed to a discussion using unoriented Gauss diagrams (cf.~Definition~\ref{ori_g}, Definition~\ref{def_chord}, Section~\ref{un}, and \cite{IT3}),  for every number $b$ of chords ($2 \le b \le 5$), there is no function that is invariant under weak (1, 3) homotopy (some examples for other homotopies, computations are given by \cite{Ito}).  However, by the above table, for oriented versions, i.e., for every $\check{G}_{b, b+1}$ ($2 \le b \le 5$), there exists a function that is invariant under weak (1, 3) homotopy.
   \end{example}

\section{Tables}\label{table_sec}   
Tables~\ref{t3} and \ref{t4} or Table~\ref{t1} list the prime reduced spherical curves up to seven double points.  
Table~\ref{t1} represents values of three invariants $(I_1, I_2, I_3)$ obtained from Section~\ref{oriented_com}.  In Table~\ref{t1}, for any pair of spherical curves $P$ and $P'$ in a box, $(I_1 (P), I_2 (P), I_3 (P))$ $=$ $(I_1 (P'), I_2 (P'), I_3 (P'))$.   
Further, except for the pair $7_4$ and $7_B$, it is easy to show that any pair of spherical curves in the same box in the leftmost column are related by a finite sequence generated by RI and strong~RI\!I\!I.  For the pair $7_4$ and $7_B$, see Fig.~\ref{seq}.  The sequence of Fig.~\ref{seq} is obtained by \cite{ITT_ildt}.  
\begin{figure}[h!]
\includegraphics[width=12cm]{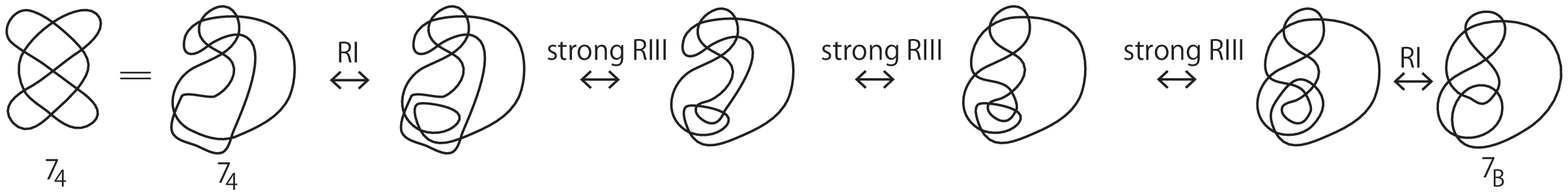}
\caption{A sequence generated by RI and strong~RI\!I\!I between $7_4$ and $7_B$.  The leftmost equality ``$=$" is an ambient isotopy on $S^2$.}\label{seq}
\end{figure}
Tables~\ref{t0a}--\ref{t0c} are the list of $\check{G}_{4, 5} \cap \conn$.  Tables~\ref{t0d1}--\ref{t0d13} are the $13$ tuples (corresponding to the formula $\dim V_w^\text{ori}(4, 5)$ $=$ $13$ preceding Conjecture~\ref{conjecture1}) consisting of the coefficients in the order of $\check{G}_{4, 5} \cap \conn$ (Tables~\ref{t0a}--\ref{t0c}).

Further, the list of $15922$ oriented Gauss diagrams in $\check{G}_{5, 6} \cap \conn$ and   
the $31$ tuples (corresponding to the formula $\dim V_w^\text{ori}(5, 6)$ $=$ $31$ preceding Conjecture~\ref{conjecture1}) consisting of the coefficients in the order of $\check{G}_{5, 6} \cap \conn$ are given by the webpage:\\
\texttt{http://www.chuo-u.gr.jp/\textasciitilde{}takamura/sc/index.html}. 

\begin{table}[h!]
\caption{Prime spherical curves, positive knot diagrams, and values of $x^*$.  In this table, for each $P$ and for each $x^*$, $x^*(P^{+})$ $=$ $x^*(P^{-})$.
}\label{t3}
\includegraphics[width=12cm]{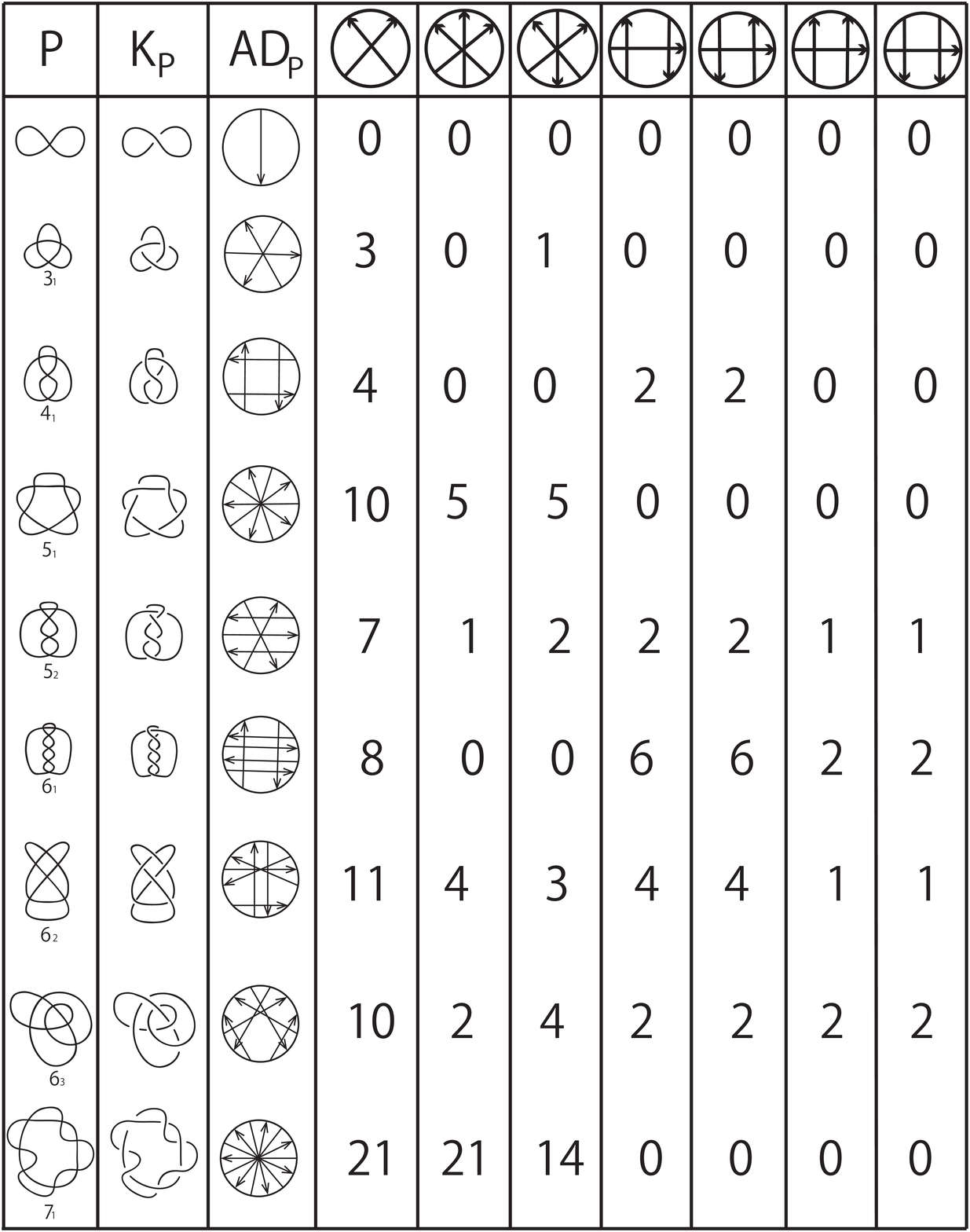}
\end{table}
\begin{table}[h!]
\caption{Prime spherical curves, positive knot diagrams, and values of $x^*$.  In this table, for each $P$ and for each $x^*$, $x^*(P^{+})$ $=$ $x^*(P^{-})$.    
}\label{t4}
\includegraphics[width=12cm]{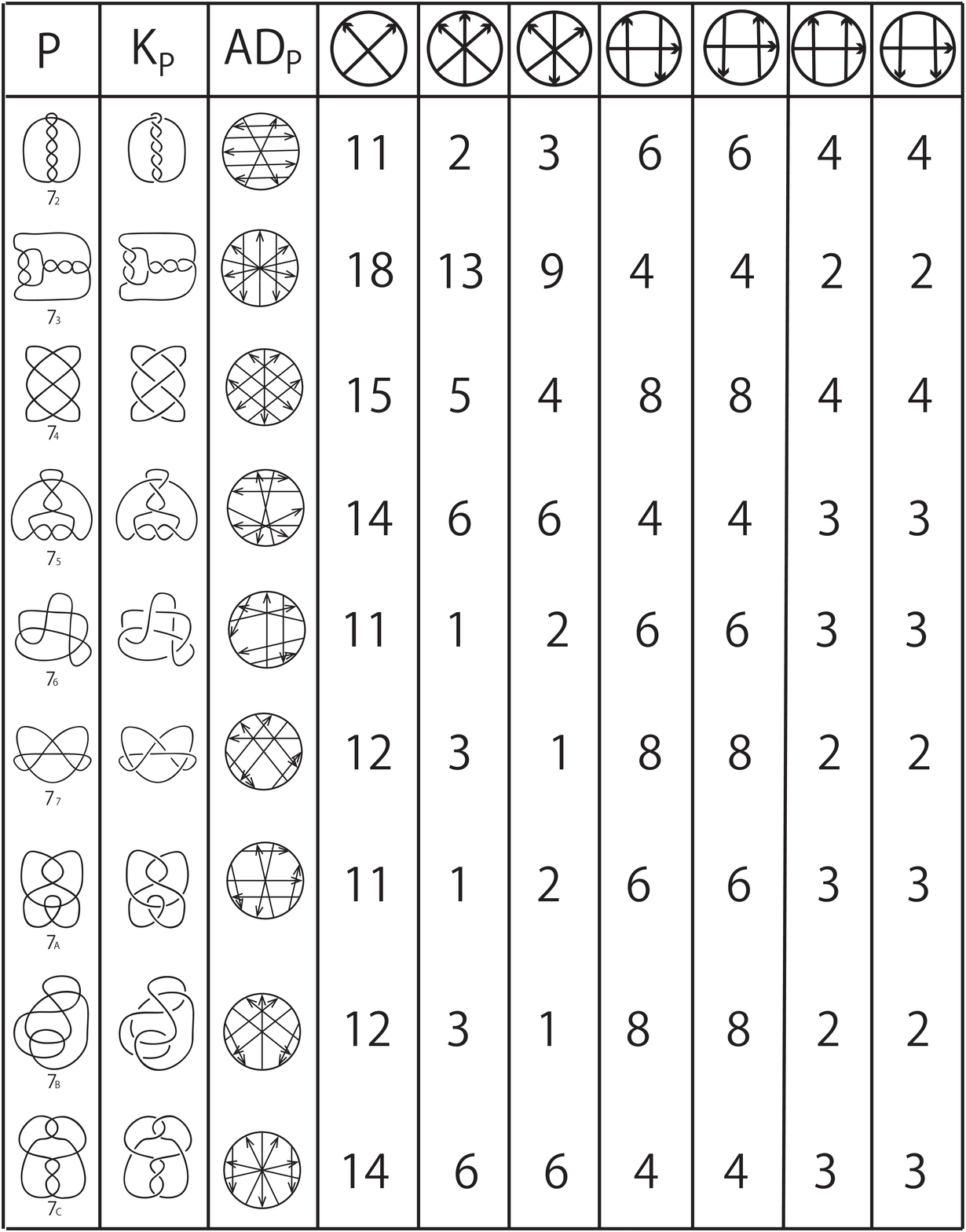}
\end{table}  
\begin{table}[h!]
\caption{Values of three invariants obtained from Example~\ref{strong23}.  Any pair of prime spherical curves in the same box in the leftmost column are related by a finite sequence generated by RI and strong~RI\!I\!I.}\label{t1}
\includegraphics[width=12cm]{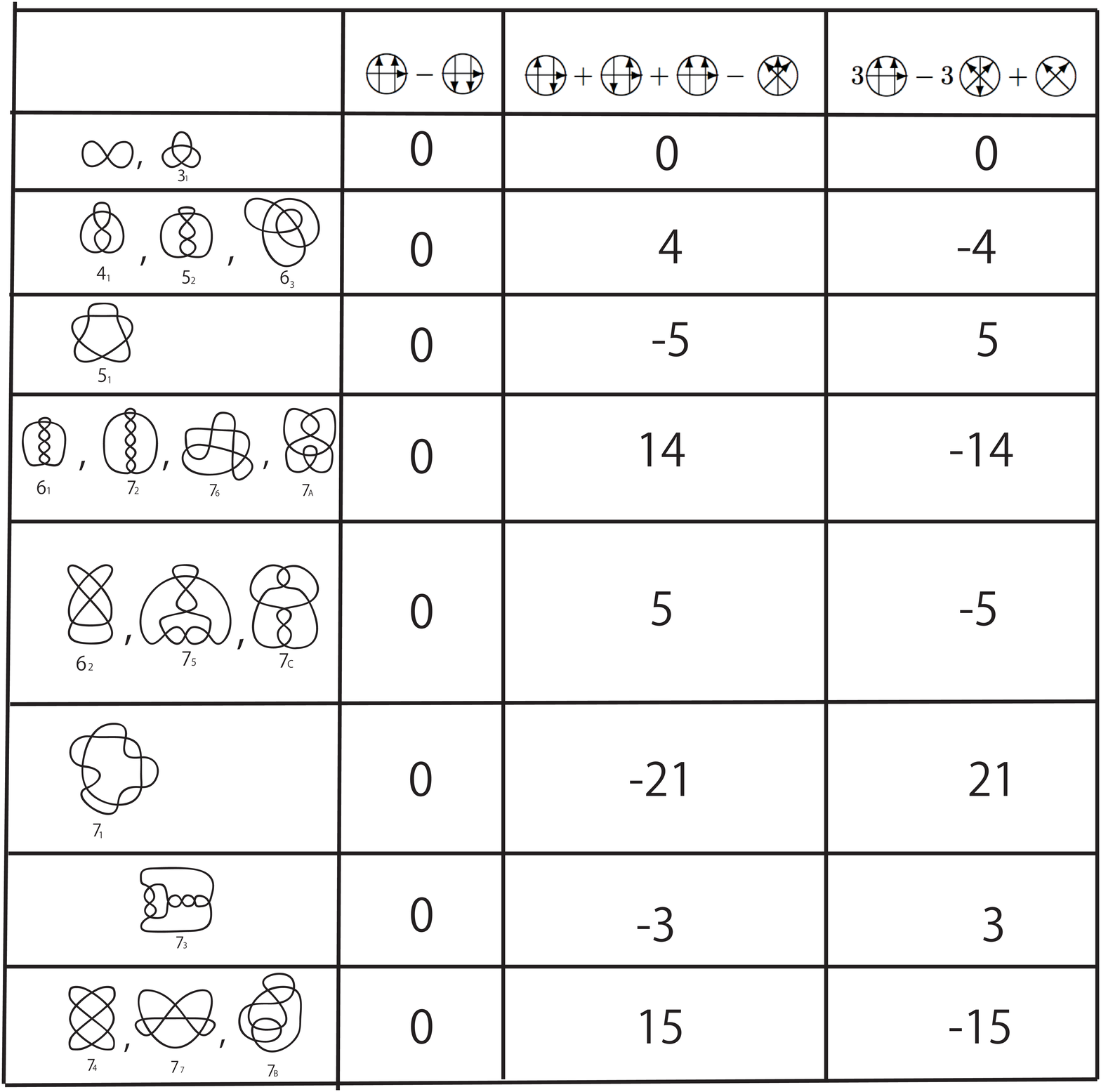}
\end{table}

\begin{table}
\caption{The ordered elements of $\check{G}_{4, 5}$ from no.$1$ to no.$300$}\label{t0a}
\includegraphics[width=12cm]{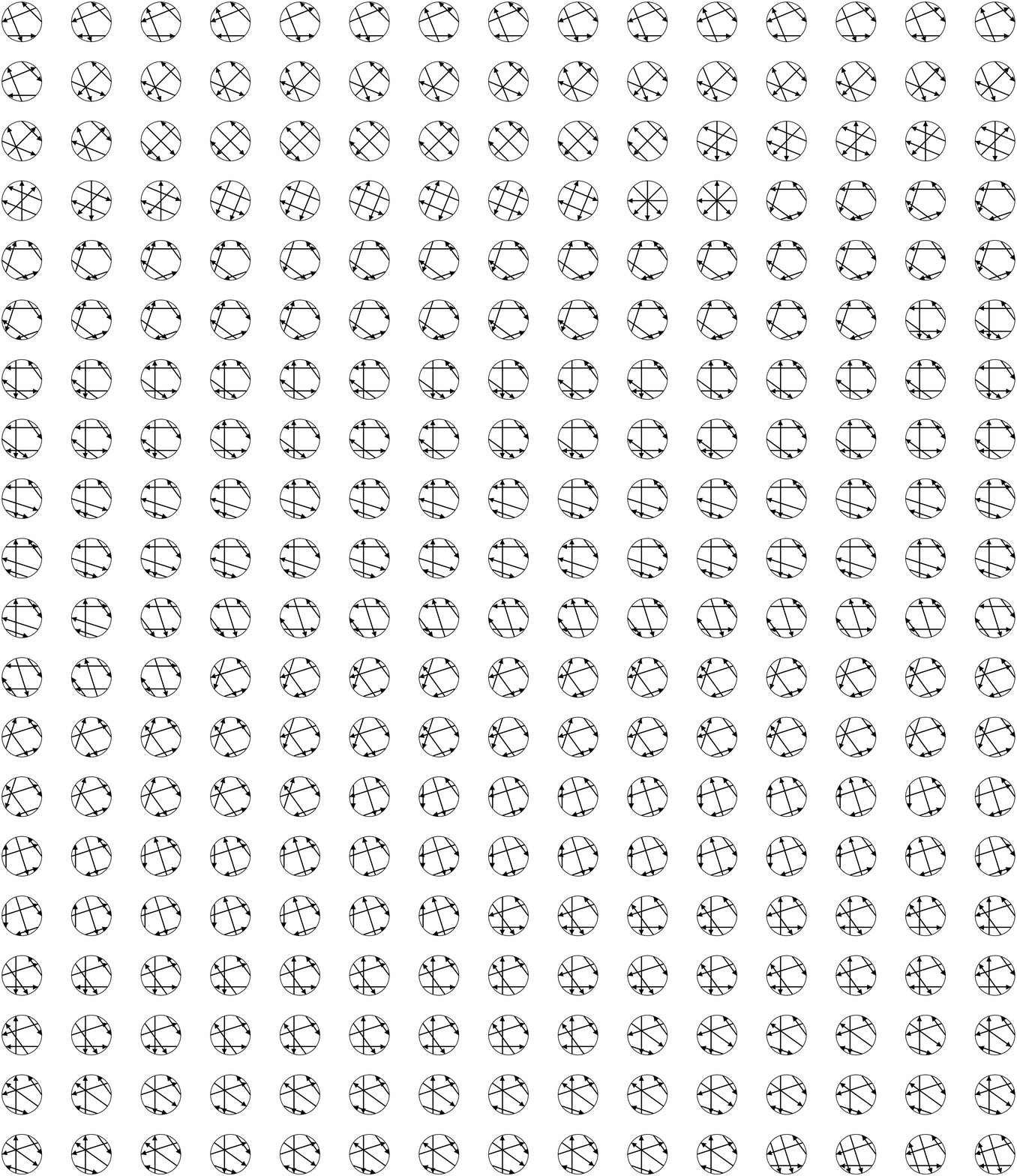}
\end{table}
\begin{table}
\caption{The ordered elements of $\check{G}_{4, 5}$ from no.$301$ to no.$600$}\label{t0b}
\includegraphics[width=12cm]{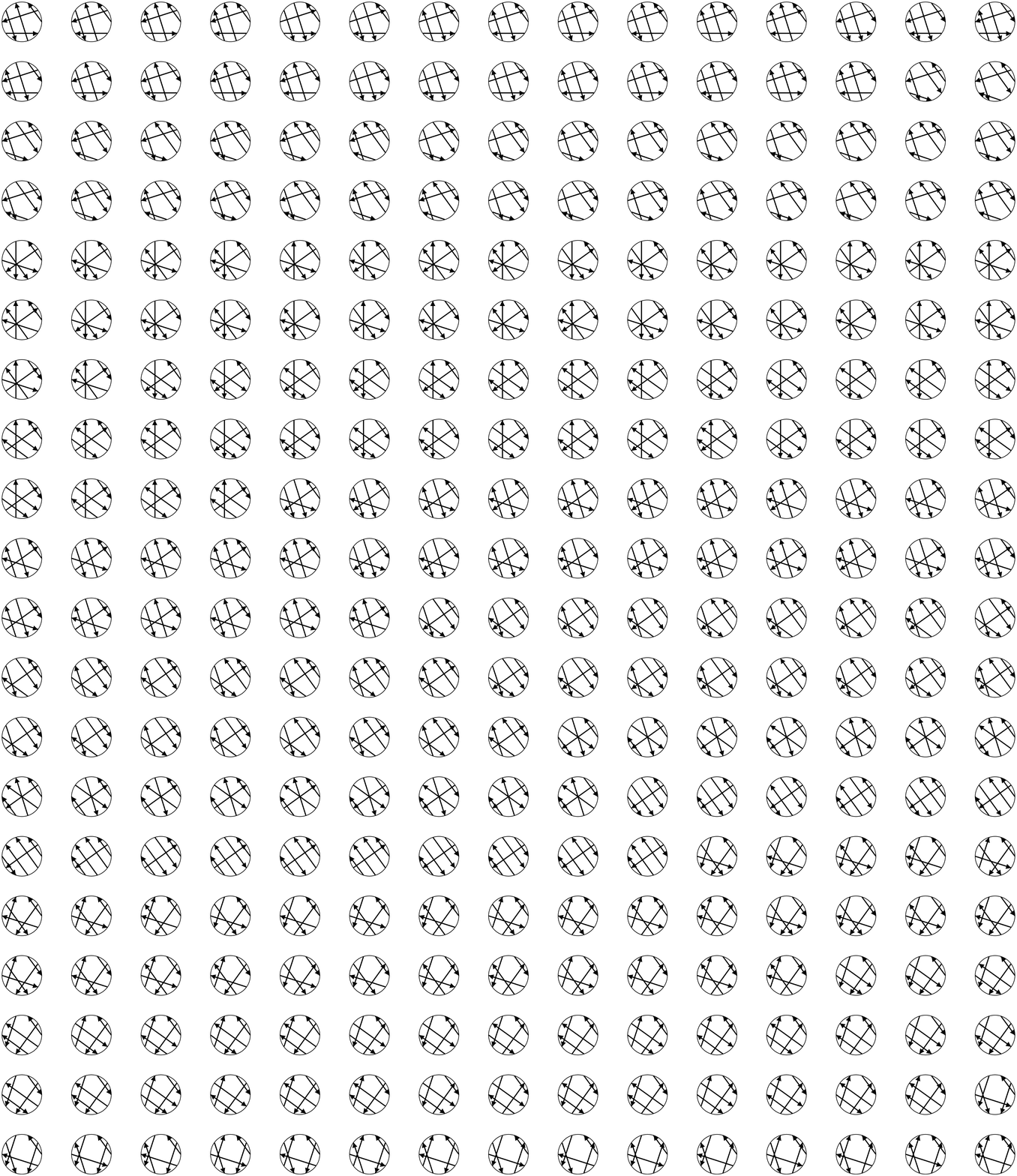}
\end{table}
\begin{table}
\caption{The ordered elements of $\check{G}_{4, 5}$ from no.$601$ to no.$852$}\label{t0c}
\includegraphics[width=12cm]{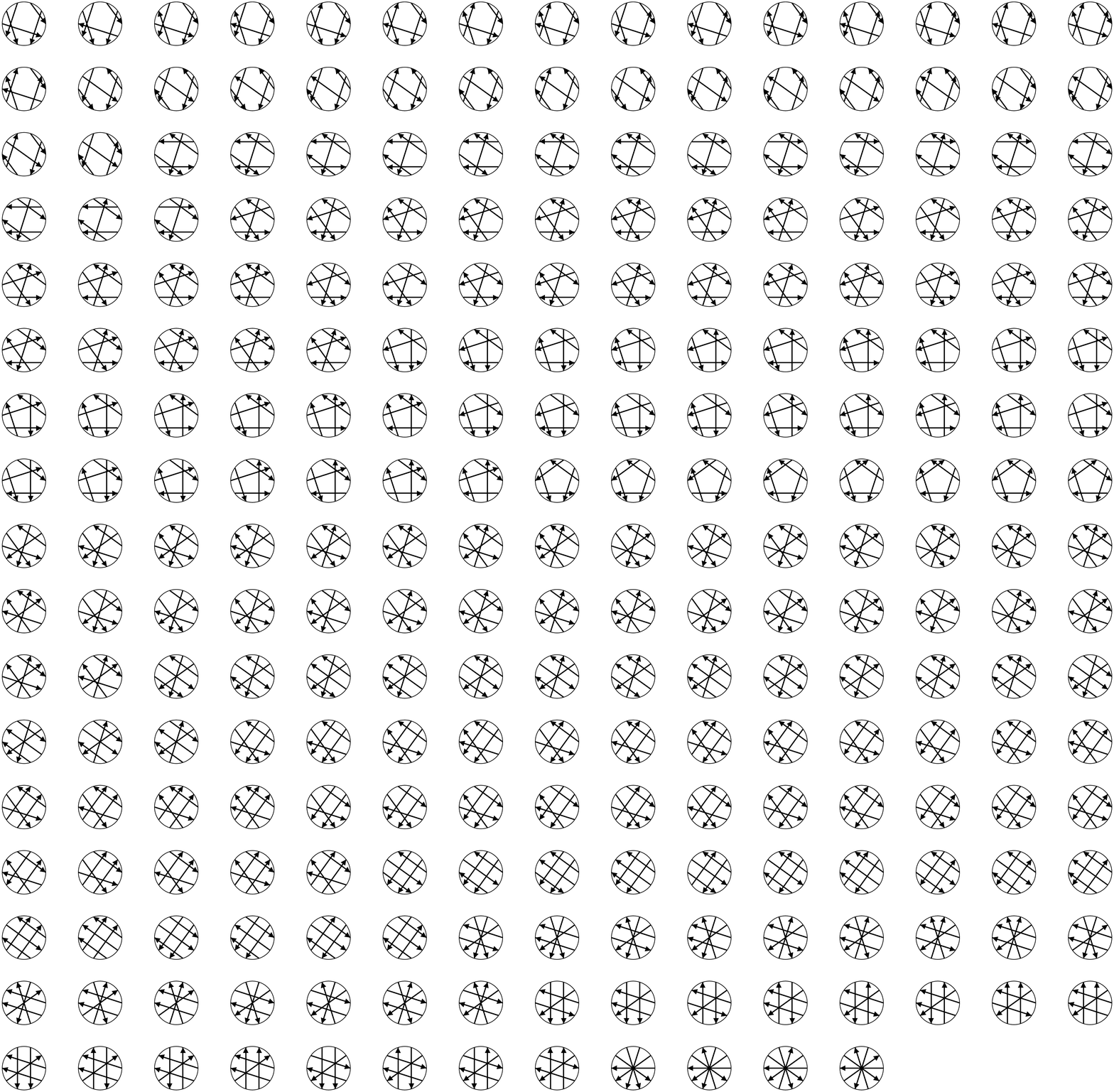}
\end{table}

\begin{table}[h!]
\caption{The first one of the $13$ tuples (corresponding to $\dim V_w^\text{ori}(4, 5)$ $=$ $13$) consisting of the coefficients in the order of $\check{G}_{4, 5} \cap \conn$ (Tables~\ref{t0a}--\ref{t0c})}\label{t0d1}
{\tiny $\begin{array}{*{15}c}
-19 &82 &-7 &44 &55 &-161 &15 &79 &-34 &15 &65 &-58 &44 &106 &-31\\
-19 &82 &44 &18 &2 &-34 &-30 &44 &58 &-55 &44 &-30 &79 &115 &18\\
44 &-31 &-17 &0 &31 &92 &0 &-16 &-17 &31 &-16 &6 &14 &64 &14\\
0 &-16 &14 &0 &-26 &0 &0 &28 &28 &32 &32 &-33 &7 &-80 &80\\
17 &-13 &52 &-30 &30 &-26 &-17 &-13 &-58 &84 &-7 &7 &7 &-7 &84\\
-58 &-13 &-17 &-26 &30 &-30 &52 &-13 &17 &80 &-80 &7 &-33 &72 &-7\\
-211 &94 &-52 &9 &125 &-30 &-99 &52 &116 &-121 &105 &8 &-144 &83 &-30\\
-31 &121 &-8 &-8 &3 &-61 &14 &83 &12 &-104 &61 &-19 &-98 &106 &-41\\
72 &-98 &-94 &94 &-52 &56 &52 &-30 &-99 &103 &91 &-121 &105 &-79 &-83\\
83 &-30 &30 &34 &-8 &-8 &-22 &-10 &14 &83 &-61 &-57 &61 &-19 &19\\
15 &-41 &4 &87 &-30 &-87 &-51 &0 &73 &0 &55 &-30 &-61 &0 &0\\
87 &-51 &4 &-29 &71 &-47 &9 &-13 &-77 &23 &63 &1 &-61 &-13 &25\\
-25 &89 &15 &-31 &-31 &15 &89 &-25 &25 &-13 &-61 &1 &63 &23 &-77\\
-13 &9 &-47 &71 &-29 &72 &-30 &-98 &30 &-92 &32 &96 &-10 &-77 &17\\
81 &-47 &87 &-45 &-61 &45 &-68 &52 &68 &-26 &66 &-32 &-96 &36 &103\\
-17 &-81 &21 &-83 &15 &83 &-41 &-97 &125 &39 &-49 &64 &-126 &-46 &90\\
67 &-93 &-31 &31 &-20 &56 &32 &-42 &71 &-81 &-57 &93 &-82 &82 &20\\
-46 &-23 &67 &-13 &-49 &64 &-74 &12 &16 &171 &-129 &-157 &119 &-186 &96\\
186 &-100 &-142 &82 &142 &-130 &193 &-129 &-135 &119 &-107 &91 &97 &-33 &96\\
-84 &-144 &84 &126 &-40 &-130 &40 &-107 &69 &97 &-55 &-97 &125 &44 &-84\\
81 &-113 &-46 &90 &67 &-93 &-26 &48 &-55 &69 &32 &-42 &71 &-81 &-44\\
58 &-65 &87 &20 &-46 &-23 &67 &0 &-32 &29 &-69 &12 &16 &70 &-28\\
-121 &53 &-92 &32 &113 &-27 &-64 &4 &81 &-47 &64 &-22 &-59 &43 &-70\\
54 &91 &-49 &66 &-32 &-109 &49 &86 &0 &-81 &21 &-60 &-8 &85 &-43\\
99 &-17 &-144 &58 &-99 &49 &52 &2 &-71 &17 &20 &-58 &99 &-13 &-52\\
58 &-55 &61 &100 &-14 &55 &-93 &-96 &42 &115 &-61 &-64 &14 &-55 &-31\\
96 &-14 &99 &-68 &-97 &58 &-99 &42 &63 &2 &-71 &32 &53 &-58 &99\\
-42 &-71 &58 &-55 &42 &71 &-14 &55 &-60 &-81 &42 &115 &-50 &-71 &14\\
-55 &16 &45 &-14 &99 &-17 &-84 &20 &-111 &39 &52 &2 &-71 &17 &48\\
-20 &67 &-47 &-52 &58 &-55 &61 &66 &-46 &93 &-65 &-96 &42 &115 &-61\\
-74 &2 &-93 &29 &96 &-14 &0 &11 &-9 &-54 &-6 &-1 &-33 &92 &-6\\
-57 &-37 &48 &-6 &21 &53 &-16 &-16 &53 &21 &-6 &48 &-37 &-57 &-6\\
92 &-33 &-1 &-6 &-54 &-9 &11 &0 &0 &42 &-4 &-34 &-34 &-56 &-6\\
92 &-54 &46 &68 &-16 &0 &-4 &-34 &-6 &0 &42 &-14 &-54 &-6 &-54\\
-6 &92 &-54 &48 &48 &-16 &0 &-6 &-14 &-6 &-43 &90 &31 &-30 &23\\
-60 &-37 &26 &61 &-94 &-71 &64 &-39 &38 &67 &-26 &87 &-46 &-75 &74\\
-49 &42 &19 &-52 &-87 &76 &53 &-90 &83 &-82 &-23 &70 &-43 &36 &55\\
-30 &23 &-20 &-47 &26 &61 &-66 &-77 &64 &-39 &32 &51 &-26 &87 &-62\\
-81 &74 &-49 &36 &47 &-52 &-87 &66 &93 &-90 &83 &-58 &-77 &70 &113\\
-59 &-206 &130 &-50 &62 &91 &-81 &-92 &64 &63 &-101 &47 &3 &-70 &86\\
-27 &43 &116 &-66 &12 &-50 &-49 &21 &32 &-22 &-51 &63 &17 &-93 &54\\
0 &59 &-55 &-23 &49 &0 &-30 &0 &4 &-30 &0 &23 &-37 &59 &-23\\
4 &0 &-6 &-1 &-33 &92 &21 &53 &-16 &8 &-152 &-33 &-8 &53 &8\\
-1 &21 &-6 &-14 &-76 &21 &117 &-54 &52 &9 &-47 &27 &77 &-54 &-106\\
3 &-75 &38 &82 &82 &38 &-75 &3 &-106 &-54 &77 &27 &-47 &9 &52\\
-54 &117 &21 &-76 &-14 &0 &0 &52 &4 &-79 &35 &-21 &25 &19 &13\\
-43 &-61 &0 &-4 &56 &4 &4 &56 &-4 &0 &-61 &-43 &13 &19 &25\\
-21 &35 &-79 &4 &52 &0 &0 &13 &17 &9 &17 &9 &13 &-195 &-195\\
6 &0 &32 &6 &-12 &24 &-24 &64 &-6 &-34 &-82 &-18 &0 &-10 &54\\
0 &0 &54 &-10 &0 &-18 &-82 &-34 &-6 &64 &-24 &24 &-12 &6 &32\\
0 &6 &6 &14 &0 &6 &-12 &-20 &-6 &64 &-42 &-18 &20 &0 &6\\
0 &-12 &-6 &6 &0 &32 &6 &0 &4 &-44 &-12 &0 &12 &-36 &-12\\
0 &-10 &54 &0 &0 &54 &-10 &0 &-12 &-36 &12 &0 &-12 &-44 &4\\
0 &6 &32 &0 &6 &6 &14 &0 &6 &0 &-14 &-12 &0 &-14 &-12\\
20 &0 &6 &0 &0 &0 &-28 &32 &0 &32 &-36 &0 &0 &32 &0\\
0 &0 &32 &-28 &0 &-36 &0 &-28 &32 &8 &-28 &0 &-28 &0 &32\\
0 &0 &0 &32 &-28 &0 &8 &0 &32 &0 &0 &-160 &\end{array}$
}
\end{table}

\begin{table}[h!]
\caption{The second one of the $13$ tuples (corresponding to $\dim V_w^\text{ori}(4, 5)$ $=$ $13$) consisting of the coefficients in the order of $\check{G}_{4, 5} \cap \conn$ (Tables~\ref{t0a}--\ref{t0c})}\label{t0d2}
{\tiny
$\begin{array}{*{15}c}
-33 &46 &19 &36 &53 &-195 &5 &69 &-14 &5 &-13 &-38 &36 &110 &-37\\
-33 &46 &36 &-18 &-26 &-14 &-90 &36 &38 &-53 &36 &-82 &69 &65 &-26\\
36 &-37 &-27 &-40 &61 &52 &-40 &-48 &-19 &53 &-32 &-30 &34 &40 &34\\
0 &-32 &34 &0 &-46 &0 &0 &148 &148 &16 &0 &5 &-3 &-32 &32\\
35 &-39 &12 &-10 &2 &-6 &45 &-39 &-38 &36 &3 &-3 &-3 &3 &28\\
-30 &-39 &45 &-14 &10 &-2 &4 &-39 &35 &40 &-40 &-3 &5 &72 &-45\\
-97 &74 &-28 &-21 &71 &-26 &-137 &100 &180 &-139 &83 &-48 &-120 &81 &-10\\
-21 &35 &-8 &-40 &89 &-7 &-38 &73 &-36 &-112 &71 &-17 &-14 &54 &-19\\
72 &-70 &-74 &74 &-28 &24 &28 &-26 &-137 &133 &145 &-139 &83 &-85 &-81\\
81 &-10 &10 &6 &-8 &-40 &46 &34 &-38 &73 &-71 &-75 &71 &-17 &17\\
21 &-19 &-12 &37 &14 &-37 &-33 &8 &35 &0 &29 &6 &-39 &0 &-8\\
37 &-33 &-4 &1 &61 &-45 &-21 &41 &-103 &21 &45 &11 &-79 &41 &27\\
-51 &115 &5 &-69 &-69 &5 &107 &-43 &27 &41 &-87 &19 &53 &13 &-103\\
41 &-13 &-53 &61 &1 &80 &-18 &-78 &18 &-76 &8 &72 &-6 &-95 &27\\
91 &-21 &77 &-15 &-79 &15 &-84 &20 &84 &-22 &78 &-8 &-72 &4 &93\\
-27 &-91 &23 &-81 &21 &81 &-19 &-99 &103 &37 &-59 &80 &-98 &-26 &62\\
41 &-55 &3 &5 &-36 &56 &-8 &-6 &101 &-99 &-43 &63 &-86 &94 &28\\
-58 &-37 &57 &-7 &-11 &40 &-62 &12 &8 &153 &-91 &-151 &85 &-166 &104\\
150 &-84 &-170 &102 &170 &-102 &171 &-107 &-165 &101 &-89 &25 &83 &-19 &104\\
-36 &-88 &20 &106 &-40 &-102 &40 &-105 &39 &99 &-37 &-99 &103 &36 &-52\\
83 &-107 &-26 &62 &41 &-55 &18 &-8 &-45 &63 &-8 &-6 &101 &-99 &-36\\
54 &-99 &109 &28 &-58 &-37 &57 &-16 &-8 &47 &-63 &12 &8 &74 &-12\\
-91 &31 &-76 &8 &99 &-33 &-64 &-4 &83 &-13 &64 &-2 &-81 &17 &-74\\
10 &97 &-35 &78 &-8 &-103 &35 &66 &0 &-83 &15 &-68 &8 &79 &-17\\
57 &-35 &-64 &46 &-73 &19 &76 &-26 &-61 &35 &68 &-46 &57 &-31 &-60\\
38 &-53 &39 &60 &-42 &69 &-23 &-80 &30 &65 &-39 &-72 &42 &-53 &27\\
64 &-34 &57 &-20 &-75 &46 &-73 &14 &77 &-26 &-61 &16 &87 &-46 &57\\
-14 &-77 &38 &-53 &22 &77 &-42 &69 &-20 &-83 &30 &65 &-22 &-93 &42\\
-53 &16 &79 &-34 &57 &-35 &-60 &20 &-77 &45 &76 &-26 &-61 &35 &56\\
-36 &73 &-45 &-60 &38 &-53 &39 &54 &-18 &71 &-43 &-80 &30 &65 &-39\\
-54 &30 &-71 &39 &64 &-34 &-16 &25 &61 &-66 &-2 &-11 &-43 &52 &14\\
-11 &-63 &64 &-2 &-1 &47 &-48 &-48 &39 &7 &-2 &64 &-55 &-19 &14\\
52 &-51 &-3 &-2 &-66 &69 &17 &-16 &-16 &78 &-12 &-54 &10 &-72 &14\\
52 &-82 &58 &76 &-48 &-16 &-12 &10 &14 &-16 &78 &6 &-66 &-2 &-66\\
14 &52 &-82 &64 &64 &-48 &-16 &-2 &6 &14 &-81 &70 &53 &-42 &21\\
-20 &-15 &14 &95 &-82 &-85 &64 &-45 &34 &33 &-14 &85 &-66 &-57 &46\\
-19 &14 &9 &-12 &-93 &76 &79 &-62 &49 &-38 &-29 &18 &-81 &52 &77\\
-42 &21 &-12 &-29 &14 &95 &-70 &-95 &64 &-45 &24 &41 &-14 &85 &-50\\
-75 &46 &-19 &-4 &29 &-12 &-93 &70 &79 &-62 &49 &-22 &-39 &18 &99\\
-81 &-90 &94 &-102 &66 &97 &-83 &-76 &64 &85 &-71 &69 &-63 &-74 &66\\
-33 &17 &28 &-30 &36 &2 &-35 &15 &8 &-2 &-17 &5 &-5 &1 &10\\
0 &9 &-13 &-5 &3 &-8 &14 &8 &-4 &6 &0 &5 &-7 &9 &-5\\
-12 &0 &-2 &-11 &-43 &52 &-1 &47 &-48 &0 &64 &-51 &-32 &39 &-16\\
-3 &7 &-2 &-10 &-52 &-1 &63 &14 &60 &-5 &-45 &-15 &55 &-2 &-62\\
9 &-81 &2 &70 &70 &2 &-81 &9 &-62 &-2 &55 &-15 &-45 &-5 &60\\
14 &63 &-1 &-52 &-10 &0 &0 &-20 &12 &19 &-7 &1 &11 &-7 &-9\\
-1 &9 &-16 &4 &8 &-4 &-4 &8 &4 &-16 &9 &-1 &-9 &-7 &11\\
1 &-7 &19 &12 &-20 &0 &0 &-9 &3 &-5 &3 &-5 &-9 &55 &55\\
-30 &8 &56 &-30 &-4 &8 &-48 &40 &-2 &18 &-54 &34 &16 &-46 &34\\
0 &0 &42 &-38 &0 &34 &-46 &10 &-2 &40 &-56 &16 &-4 &-30 &48\\
0 &-14 &-30 &42 &16 &-30 &-4 &-36 &-2 &40 &-22 &34 &12 &0 &-14\\
0 &-4 &-2 &-30 &8 &56 &-30 &0 &20 &-52 &36 &0 &4 &-52 &36\\
16 &-46 &34 &0 &0 &42 &-38 &0 &36 &-60 &12 &0 &36 &-44 &12\\
0 &-30 &48 &0 &-14 &-30 &42 &16 &-30 &0 &-34 &36 &0 &-34 &36\\
12 &0 &-14 &0 &0 &0 &4 &16 &0 &16 &-4 &0 &-16 &16 &-16\\
0 &-16 &16 &4 &-16 &-4 &0 &4 &16 &0 &4 &0 &4 &-16 &16\\
-16 &-16 &0 &16 &4 &0 &-16 &0 &16 &0 &-16 &0 &\end{array}$
}
\end{table}
\begin{table}[h!]
\caption{The third one of the $13$ tuples (corresponding to $\dim V_w^\text{ori}(4, 5)$ $=$ $13$) consisting of the coefficients in the order of $\check{G}_{4, 5} \cap \conn$ (Tables~\ref{t0a}--\ref{t0c})}\label{t0d3}
{\tiny
$\begin{array}{*{15}c}
-43 &50 &33 &44 &63 &-217 &7 &71 &-2 &7 &-39 &-26 &44 &122 &-23\\
-43 &50 &44 &-14 &-14 &-2 &-110 &44 &58 &-31 &44 &-94 &71 &75 &-30\\
44 &-23 &-25 &-48 &71 &60 &-48 &-16 &-9 &55 &-48 &-42 &46 &80 &46\\
0 &-48 &46 &0 &-74 &0 &0 &188 &188 &32 &0 &-9 &-1 &-48 &48\\
41 &-37 &20 &-14 &-2 &6 &23 &-37 &-58 &68 &1 &-1 &-1 &1 &52\\
-42 &-37 &23 &-10 &14 &2 &4 &-37 &41 &64 &-64 &-1 &-9 &72 &-31\\
-139 &78 &-20 &-31 &85 &-14 &-139 &100 &164 &-145 &81 &-8 &-128 &75 &-14\\
-23 &65 &-8 &-40 &75 &-21 &-34 &91 &-36 &-120 &85 &-11 &-66 &74 &-17\\
72 &-82 &-78 &78 &-20 &24 &20 &-14 &-139 &143 &131 &-145 &81 &-71 &-75\\
75 &-14 &14 &18 &-8 &-40 &26 &38 &-34 &91 &-85 &-81 &85 &-11 &11\\
7 &-17 &-12 &63 &2 &-63 &-43 &16 &49 &0 &47 &-14 &-53 &0 &-16\\
63 &-43 &4 &-5 &63 &-55 &1 &27 &-101 &31 &39 &9 &-69 &27 &17\\
-65 &129 &7 &-55 &-55 &7 &113 &-49 &17 &27 &-85 &25 &55 &15 &-101\\
27 &17 &-71 &63 &-5 &88 &-30 &-98 &30 &-76 &16 &80 &-10 &-85 &25\\
89 &-39 &79 &-21 &-69 &21 &-84 &36 &84 &-26 &66 &-16 &-80 &20 &95\\
-25 &-89 &29 &-75 &7 &75 &-17 &-121 &117 &47 &-57 &96 &-126 &-30 &74\\
43 &-53 &25 &-9 &-36 &56 &-16 &-10 &111 &-105 &-49 &69 &-98 &114 &20\\
-62 &-31 &43 &-37 &7 &48 &-58 &28 &0 &163 &-105 &-149 &95 &-186 &112\\
154 &-84 &-158 &98 &158 &-114 &169 &-105 &-143 &95 &-99 &51 &89 &-25 &112\\
-68 &-96 &36 &110 &-40 &-114 &40 &-99 &45 &89 &-31 &-121 &117 &44 &-68\\
105 &-121 &-30 &74 &43 &-53 &22 &0 &-47 &61 &-16 &-10 &111 &-105 &-44\\
58 &-89 &111 &20 &-62 &-31 &43 &-32 &16 &37 &-61 &28 &0 &70 &-12\\
-113 &45 &-76 &16 &105 &-35 &-64 &4 &73 &-23 &64 &-6 &-67 &19 &-70\\
22 &99 &-41 &66 &-16 &-101 &41 &70 &0 &-73 &13 &-60 &-8 &77 &-19\\
43 &-9 &-80 &42 &-75 &9 &84 &-14 &-63 &41 &68 &-74 &75 &-21 &-84\\
58 &-31 &21 &68 &-30 &63 &-21 &-96 &26 &75 &-53 &-80 &62 &-63 &9\\
96 &-46 &43 &-20 &-73 &42 &-75 &26 &71 &-14 &-63 &32 &93 &-74 &75\\
-26 &-95 &58 &-31 &10 &63 &-30 &63 &-28 &-73 &26 &75 &-34 &-95 &62\\
-63 &16 &85 &-46 &43 &-9 &-52 &4 &-87 &31 &84 &-14 &-63 &41 &64\\
-36 &75 &-55 &-84 &58 &-31 &21 &50 &-14 &85 &-41 &-96 &26 &75 &-53\\
-74 &34 &-85 &53 &96 &-46 &-32 &35 &47 &-70 &10 &-9 &-41 &60 &10\\
-17 &-61 &48 &10 &-19 &45 &-16 &-16 &29 &-3 &10 &48 &-45 &-33 &10\\
60 &-57 &7 &10 &-70 &63 &19 &-32 &-32 &90 &-4 &-50 &14 &-88 &10\\
60 &-70 &30 &68 &-16 &-32 &-4 &14 &10 &-32 &90 &2 &-70 &10 &-70\\
10 &60 &-70 &48 &48 &-16 &-32 &10 &2 &10 &-83 &90 &55 &-46 &-1\\
-28 &3 &10 &101 &-94 &-79 &64 &-31 &22 &27 &-10 &95 &-78 &-67 &58\\
-9 &26 &-5 &-20 &-111 &92 &77 &-74 &43 &-34 &-15 &22 &-83 &52 &79\\
-46 &-1 &12 &-23 &10 &101 &-82 &-85 &64 &-31 &16 &27 &-10 &95 &-62\\
-89 &58 &-9 &4 &23 &-20 &-111 &82 &101 &-74 &43 &-26 &-37 &22 &105\\
-67 &-126 &98 &-98 &62 &99 &-73 &-76 &64 &55 &-77 &71 &-37 &-70 &70\\
-35 &19 &52 &-34 &44 &-18 &-41 &13 &16 &-6 &-11 &23 &-7 &-37 &22\\
0 &19 &-15 &1 &9 &-16 &2 &16 &4 &-14 &0 &-1 &-13 &19 &1\\
-12 &0 &10 &-9 &-41 &60 &-19 &45 &-16 &24 &-8 &-57 &-24 &29 &-8\\
7 &-3 &10 &2 &-76 &13 &77 &-6 &52 &-31 &-23 &-29 &101 &-6 &-90\\
11 &-83 &22 &66 &66 &22 &-83 &11 &-90 &-6 &101 &-29 &-23 &-31 &52\\
-6 &77 &13 &-76 &2 &0 &0 &20 &4 &-23 &11 &-13 &17 &-5 &5\\
-19 &-21 &0 &-4 &24 &4 &4 &24 &-4 &0 &-21 &-19 &5 &-5 &17\\
-13 &11 &-23 &4 &20 &0 &0 &5 &9 &1 &9 &1 &5 &-75 &-75\\
-42 &16 &80 &-42 &20 &-8 &-72 &80 &-22 &30 &-66 &14 &32 &-58 &38\\
0 &0 &54 &-42 &0 &14 &-50 &14 &-22 &80 &-88 &8 &20 &-42 &64\\
0 &-10 &-42 &62 &32 &-42 &20 &-68 &-22 &80 &-42 &14 &20 &0 &-10\\
0 &20 &-22 &-42 &16 &80 &-42 &0 &20 &-76 &36 &0 &12 &-52 &36\\
32 &-58 &38 &0 &0 &54 &-42 &0 &36 &-68 &28 &0 &36 &-60 &4\\
0 &-42 &64 &0 &-10 &-42 &62 &32 &-42 &0 &-46 &36 &0 &-46 &36\\
20 &0 &-10 &0 &0 &0 &-12 &32 &0 &32 &12 &-32 &-32 &32 &0\\
-32 &0 &32 &-12 &-32 &12 &0 &-12 &32 &24 &-12 &0 &-12 &-32 &32\\
0 &-32 &0 &32 &-12 &0 &-8 &0 &32 &-32 &0 &0 &\end{array}$
}
\end{table}
\begin{table}[h!]
\caption{The fourth one of the $13$ tuples (corresponding to $\dim V_w^\text{ori}(4, 5)$ $=$ $13$) consisting of the coefficients in the order of $\check{G}_{4, 5} \cap \conn$ (Tables~\ref{t0a}--\ref{t0c})}\label{t0d4}
{\tiny
$\begin{array}{*{15}c}
1 &0 &-1 &-1 &-1 &4 &0 &-1 &-1 &0 &2 &0 &-1 &-2 &0\\
1 &0 &-1 &0 &1 &-1 &3 &-1 &-1 &0 &-1 &2 &-1 &0 &1\\
-1 &0 &1 &1 &-2 &0 &1 &0 &0 &-1 &1 &2 &-1 &-1 &-1\\
0 &1 &-1 &0 &2 &0 &0 &-4 &-4 &0 &0 &0 &0 &0 &0\\
-1 &1 &0 &0 &1 &-1 &-1 &1 &1 &-1 &0 &0 &0 &0 &0\\
0 &1 &-1 &0 &0 &-1 &1 &1 &-1 &-1 &1 &0 &0 &-1 &1\\
1 &-1 &0 &1 &-1 &0 &3 &-3 &-3 &3 &-1 &0 &2 &-1 &0\\
0 &0 &0 &1 &-2 &0 &1 &-2 &2 &2 &-2 &0 &1 &-1 &0\\
-1 &1 &1 &-1 &0 &0 &0 &0 &3 &-3 &-3 &3 &-1 &1 &1\\
-1 &0 &0 &0 &0 &1 &-1 &-1 &1 &-2 &2 &2 &-2 &0 &0\\
0 &0 &1 &-1 &-1 &1 &0 &-1 &0 &0 &0 &0 &1 &0 &1\\
-1 &0 &0 &0 &-1 &1 &0 &-1 &2 &-1 &0 &0 &1 &-1 &0\\
2 &-3 &0 &1 &1 &0 &-2 &1 &0 &-1 &2 &-1 &-1 &0 &2\\
-1 &-1 &2 &-1 &0 &-2 &1 &2 &-1 &1 &0 &-1 &0 &2 &-1\\
-2 &1 &-1 &0 &1 &0 &2 &-1 &-2 &1 &-1 &0 &1 &0 &-2\\
1 &2 &-1 &1 &0 &-1 &0 &3 &-2 &-1 &1 &-2 &2 &0 &-1\\
-1 &0 &-1 &1 &1 &-1 &1 &0 &-3 &2 &1 &-1 &2 &-2 &0\\
1 &1 &0 &1 &-1 &-1 &1 &-1 &0 &-3 &2 &2 &-1 &4 &-3\\
-2 &1 &3 &-2 &-3 &2 &-3 &2 &2 &-1 &2 &-1 &-1 &0 &-3\\
2 &1 &0 &-2 &1 &2 &-1 &2 &-1 &-1 &0 &3 &-2 &-1 &1\\
-2 &2 &0 &-1 &-1 &0 &-1 &1 &1 &-1 &1 &0 &-3 &2 &1\\
-1 &2 &-2 &0 &1 &1 &0 &1 &-1 &-1 &1 &-1 &0 &-1 &0\\
2 &-1 &1 &0 &-2 &1 &1 &0 &-1 &0 &-1 &0 &1 &0 &1\\
0 &-2 &1 &-1 &0 &2 &-1 &-1 &0 &1 &0 &1 &0 &-1 &0\\
0 &0 &1 &-1 &2 &0 &-3 &1 &0 &-1 &-1 &2 &-1 &0 &2\\
-1 &0 &0 &-1 &1 &-2 &0 &3 &-1 &0 &1 &1 &-2 &1 &0\\
-2 &1 &0 &0 &1 &-1 &2 &-1 &-2 &1 &0 &0 &-2 &2 &-1\\
0 &2 &-1 &0 &0 &-1 &1 &-2 &1 &2 &-1 &0 &0 &2 &-2\\
1 &0 &-2 &1 &0 &0 &1 &0 &2 &-1 &-3 &1 &0 &-1 &-1\\
1 &-1 &1 &2 &-1 &0 &0 &-1 &0 &-2 &1 &3 &-1 &0 &1\\
1 &-1 &1 &-1 &-2 &1 &1 &-1 &-1 &1 &0 &0 &0 &0 &-1\\
0 &2 &-1 &0 &1 &-1 &0 &0 &0 &0 &0 &-1 &1 &1 &-1\\
0 &1 &-1 &0 &1 &-2 &0 &1 &1 &-2 &1 &0 &-1 &2 &-1\\
0 &2 &0 &-2 &0 &1 &1 &-1 &-1 &1 &-2 &0 &1 &0 &1\\
-1 &0 &2 &-1 &-1 &0 &1 &0 &0 &-1 &2 &-2 &-1 &1 &1\\
0 &-1 &0 &-3 &2 &2 &-1 &0 &0 &0 &0 &-2 &2 &1 &-1\\
-1 &0 &1 &0 &3 &-2 &-2 &1 &0 &0 &0 &0 &2 &-1 &-2\\
1 &1 &-1 &0 &0 &-3 &2 &2 &-1 &0 &0 &0 &0 &-2 &1\\
2 &-1 &-1 &1 &0 &0 &3 &-2 &-2 &1 &0 &0 &0 &0 &-2\\
1 &1 &-1 &3 &-1 &-2 &1 &1 &-1 &0 &1 &-2 &1 &1 &-1\\
1 &0 &0 &0 &-2 &0 &1 &0 &0 &0 &-1 &0 &1 &0 &0\\
0 &0 &0 &-1 &1 &1 &-1 &-1 &0 &0 &0 &1 &0 &0 &-1\\
1 &0 &0 &0 &0 &0 &1 &-1 &0 &-1 &-1 &1 &1 &0 &1\\
-1 &0 &0 &0 &1 &-1 &0 &0 &-1 &1 &0 &1 &-2 &0 &1\\
-1 &2 &0 &-1 &-1 &0 &2 &-1 &1 &0 &-2 &1 &0 &1 &-1\\
0 &0 &-1 &1 &0 &0 &0 &0 &0 &0 &0 &0 &0 &0 &0\\
0 &0 &0 &0 &0 &0 &0 &0 &0 &0 &0 &0 &0 &0 &0\\
0 &0 &0 &0 &0 &0 &0 &0 &0 &0 &0 &0 &0 &0 &0\\
2 &-1 &-3 &2 &0 &0 &1 &-1 &0 &-1 &2 &-1 &-2 &2 &0\\
0 &0 &-1 &1 &0 &-1 &1 &0 &0 &-1 &2 &-1 &0 &2 &-2\\
0 &0 &2 &-2 &-2 &2 &0 &1 &0 &-1 &1 &-1 &0 &0 &0\\
0 &0 &0 &2 &-1 &-3 &2 &0 &-1 &2 &-1 &0 &0 &1 &-1\\
-2 &2 &0 &0 &0 &-1 &1 &0 &-1 &2 &-1 &0 &-1 &1 &0\\
0 &2 &-2 &0 &0 &2 &-2 &-2 &2 &0 &1 &-1 &0 &1 &-1\\
0 &0 &0 &0 &0 &0 &0 &0 &-1 &0 &0 &0 &1 &0 &0\\
0 &0 &0 &0 &1 &0 &-1 &0 &0 &-1 &0 &0 &0 &1 &0\\
0 &1 &-1 &0 &0 &0 &1 &-1 &0 &0 &0 &0 &\end{array}$
}
\end{table}
\begin{table}[h!]
\caption{The fifth one of the $13$ tuples (corresponding to $\dim V_w^\text{ori}(4, 5)$ $=$ $13$) consisting of the coefficients in the order of $\check{G}_{4, 5} \cap \conn$ (Tables~\ref{t0a}--\ref{t0c})}\label{t0d5}
{\tiny
$\begin{array}{*{15}c}
-3 &-6 &5 &0 &-5 &11 &-1 &-9 &6 &-1 &-19 &10 &0 &-10 &9\\
-3 &-6 &0 &-2 &2 &6 &-2 &0 &-6 &9 &0 &-2 &-9 &-13 &-2\\
0 &9 &7 &0 &-1 &-12 &0 &8 &7 &-1 &0 &-2 &6 &0 &6\\
0 &0 &6 &0 &-2 &0 &0 &12 &12 &0 &0 &-1 &-1 &0 &0\\
1 &3 &-4 &2 &-2 &6 &-9 &3 &-2 &4 &1 &-1 &-1 &1 &4\\
-2 &3 &-9 &6 &-2 &2 &-4 &3 &1 &0 &0 &-1 &-1 &-8 &9\\
5 &-10 &12 &-7 &-11 &10 &13 &-12 &-20 &15 &-15 &16 &16 &-13 &2\\
1 &1 &0 &0 &-5 &3 &-2 &-5 &4 &8 &-3 &5 &-10 &-6 &7\\
-8 &6 &10 &-10 &12 &-8 &-12 &10 &13 &-9 &-21 &15 &-15 &17 &13\\
-13 &2 &-2 &2 &0 &0 &-6 &6 &-2 &-5 &3 &7 &-3 &5 &-5\\
-9 &7 &4 &-1 &-6 &1 &-3 &0 &1 &0 &7 &-6 &3 &0 &0\\
-1 &-3 &4 &3 &-9 &1 &9 &-5 &11 &-1 &-9 &1 &11 &-5 &-7\\
-1 &-7 &-1 &9 &9 &-1 &-7 &-1 &-7 &-5 &11 &1 &-9 &-1 &11\\
-5 &9 &1 &-9 &3 &-8 &2 &6 &-2 &12 &0 &-8 &-2 &19 &-7\\
-15 &1 &-9 &3 &11 &-3 &12 &-4 &-12 &6 &-14 &0 &8 &4 &-17\\
7 &15 &-3 &13 &-9 &-13 &7 &7 &-11 &-1 &7 &0 &2 &2 &-6\\
-13 &11 &17 &-9 &4 &-8 &-8 &6 &-9 &7 &7 &-11 &6 &2 &-4\\
2 &9 &-13 &-13 &15 &-8 &14 &4 &-8 &-21 &15 &19 &-9 &22 &-16\\
-22 &12 &26 &-14 &-26 &14 &-31 &23 &25 &-17 &13 &-5 &-7 &-1 &-16\\
4 &16 &-4 &-18 &8 &14 &-8 &21 &-11 &-15 &9 &7 &-11 &-4 &4\\
1 &7 &2 &-6 &-13 &11 &6 &0 &9 &-11 &-8 &6 &-9 &7 &4\\
-6 &15 &-9 &-4 &2 &9 &-13 &-8 &16 &-11 &11 &4 &-8 &-10 &4\\
7 &-3 &12 &0 &-15 &5 &8 &4 &-15 &1 &-8 &2 &13 &-5 &10\\
-2 &-13 &7 &-14 &0 &19 &-7 &-10 &0 &15 &-3 &12 &-8 &-11 &5\\
-13 &7 &8 &-6 &13 &-7 &-4 &2 &9 &-7 &-4 &6 &-13 &11 &4\\
-6 &9 &-11 &-4 &2 &-9 &11 &8 &-6 &-13 &11 &8 &-2 &9 &-7\\
-8 &2 &-13 &4 &7 &-6 &13 &2 &-9 &2 &9 &0 &-11 &6 &-13\\
6 &9 &-6 &9 &-6 &-9 &2 &-9 &4 &15 &-6 &-13 &6 &17 &-2\\
9 &-8 &-11 &2 &-13 &7 &12 &-4 &9 &-9 &-4 &2 &9 &-7 &-8\\
12 &-13 &9 &4 &-6 &9 &-11 &-6 &2 &-3 &7 &8 &-6 &-13 &11\\
6 &-6 &11 &-3 &-8 &2 &0 &-5 &-9 &10 &2 &7 &7 &-12 &2\\
-1 &3 &-8 &2 &-3 &-3 &8 &8 &-3 &-3 &2 &-8 &3 &-1 &2\\
-12 &7 &7 &2 &10 &-9 &-5 &0 &0 &-6 &4 &6 &-2 &8 &2\\
-12 &10 &-10 &-12 &8 &0 &4 &-2 &2 &0 &-6 &-6 &10 &2 &10\\
2 &-12 &10 &-8 &-8 &8 &0 &2 &-6 &2 &5 &-6 &-1 &2 &-9\\
4 &11 &-6 &-11 &10 &17 &-8 &9 &-10 &-13 &6 &-9 &2 &5 &-6\\
7 &2 &-5 &4 &9 &-4 &-11 &6 &-13 &14 &9 &-10 &5 &-4 &-9\\
2 &-9 &12 &9 &-6 &-11 &6 &19 &-8 &9 &-8 &-13 &6 &-9 &2\\
7 &-6 &7 &4 &-9 &4 &9 &-6 &-3 &6 &-13 &6 &11 &-10 &-15\\
13 &10 &-14 &14 &-10 &-13 &15 &12 &-8 &-17 &11 &-9 &11 &10 &-10\\
5 &-5 &-4 &6 &-4 &-2 &7 &-3 &0 &2 &5 &-1 &1 &-5 &-2\\
0 &-5 &9 &1 &1 &0 &-6 &0 &4 &-6 &0 &-1 &3 &-5 &1\\
4 &0 &2 &7 &7 &-12 &-3 &-3 &8 &0 &-16 &7 &0 &-3 &0\\
7 &-3 &2 &2 &4 &5 &-11 &-6 &-4 &-7 &9 &-5 &5 &2 &6\\
3 &5 &-2 &-6 &-6 &-2 &5 &3 &6 &2 &5 &-5 &9 &-7 &-4\\
-6 &-11 &5 &4 &2 &0 &0 &12 &-4 &-7 &3 &-5 &1 &-5 &5\\
-3 &-5 &8 &-4 &0 &4 &4 &0 &-4 &8 &-5 &-3 &5 &-5 &1\\
-5 &3 &-7 &-4 &12 &0 &0 &5 &1 &1 &1 &1 &5 &-35 &-35\\
-2 &0 &0 &-2 &4 &0 &0 &0 &2 &-2 &-2 &6 &0 &-2 &-2\\
0 &0 &-2 &-2 &0 &6 &-2 &-2 &2 &0 &0 &0 &4 &-2 &0\\
0 &-2 &-2 &6 &0 &-2 &4 &-4 &2 &0 &-18 &6 &4 &0 &-2\\
0 &4 &2 &-2 &0 &0 &-2 &0 &-4 &-4 &4 &0 &4 &4 &4\\
0 &-2 &-2 &0 &0 &-2 &-2 &0 &4 &4 &4 &0 &4 &-4 &-4\\
0 &-2 &0 &0 &-2 &-2 &6 &0 &-2 &0 &-6 &4 &0 &-6 &4\\
4 &0 &-2 &0 &0 &0 &4 &0 &0 &0 &-4 &0 &0 &0 &0\\
0 &0 &0 &4 &0 &-4 &0 &4 &0 &0 &4 &-8 &4 &0 &0\\
0 &0 &0 &0 &4 &-8 &0 &0 &0 &0 &0 &0 &\end{array}$
}
\end{table}
\begin{table}[h!]
\caption{The sixth one of the $13$ tuples (corresponding to $\dim V_w^\text{ori}(4, 5)$ $=$ $13$) consisting of the coefficients in the order of $\check{G}_{4, 5} \cap \conn$ (Tables~\ref{t0a}--\ref{t0c})}\label{t0d6}
{\tiny
$\begin{array}{*{15}c}
3 &0 &-3 &-2 &-1 &3 &-1 &1 &-2 &-1 &9 &-4 &-2 &0 &-3\\
3 &0 &-2 &2 &0 &-2 &4 &-2 &0 &-3 &-2 &4 &1 &3 &2\\
-2 &-3 &-3 &2 &-1 &4 &2 &-4 &-3 &-1 &2 &2 &-4 &-2 &-4\\
0 &2 &-4 &0 &2 &0 &0 &-12 &-12 &0 &0 &1 &1 &2 &-2\\
-3 &1 &0 &0 &0 &-2 &3 &1 &2 &-4 &-1 &1 &1 &-1 &-4\\
2 &1 &3 &-2 &0 &0 &0 &1 &-3 &-2 &2 &1 &1 &2 &-3\\
3 &2 &-6 &5 &1 &-4 &-3 &4 &4 &-1 &3 &-6 &-4 &3 &0\\
-1 &-3 &0 &2 &1 &1 &0 &-1 &-2 &2 &-3 &-1 &6 &0 &-1\\
2 &0 &-2 &2 &-6 &4 &4 &-4 &-3 &1 &5 &-1 &3 &-5 &-3\\
3 &0 &0 &-2 &0 &2 &2 &-2 &0 &-1 &1 &1 &-3 &-1 &1\\
3 &-1 &-2 &-1 &4 &1 &3 &0 &-3 &0 &-5 &4 &-1 &0 &0\\
-1 &3 &-2 &-1 &3 &1 &-5 &1 &-1 &-1 &3 &-1 &-3 &1 &5\\
1 &1 &-1 &-3 &-3 &-1 &1 &1 &5 &1 &-3 &-1 &3 &-1 &-1\\
1 &-5 &1 &3 &-1 &2 &0 &0 &0 &-2 &-2 &0 &2 &-7 &3\\
5 &1 &1 &1 &-3 &-1 &-4 &0 &4 &-2 &4 &2 &0 &-4 &5\\
-3 &-5 &1 &-3 &3 &3 &-1 &1 &1 &-1 &-1 &-4 &4 &0 &0\\
3 &-1 &-7 &5 &-2 &2 &2 &-2 &1 &-1 &-1 &1 &2 &-4 &2\\
0 &-3 &3 &7 &-7 &2 &-4 &-2 &4 &5 &-3 &-3 &-1 &-6 &6\\
4 &-2 &-8 &4 &6 &0 &9 &-7 &-7 &3 &-3 &-1 &-1 &3 &6\\
0 &-2 &-2 &4 &-2 &0 &0 &-7 &3 &3 &-1 &1 &1 &0 &2\\
-5 &1 &0 &0 &3 &-1 &-2 &0 &-3 &3 &2 &-2 &1 &-1 &0\\
0 &-3 &1 &2 &0 &-3 &3 &4 &-8 &5 &-3 &-2 &4 &2 &0\\
-1 &1 &-2 &-2 &3 &-1 &-2 &-2 &5 &1 &0 &2 &-5 &1 &-2\\
-2 &5 &-3 &4 &2 &-5 &1 &2 &0 &-5 &1 &-2 &2 &3 &-1\\
3 &-3 &0 &0 &-3 &3 &0 &0 &-3 &3 &0 &0 &3 &-3 &0\\
0 &-3 &3 &0 &0 &3 &-3 &0 &0 &3 &-3 &0 &0 &-3 &3\\
0 &0 &3 &0 &-1 &0 &-3 &0 &1 &0 &-3 &0 &1 &0 &3\\
0 &-1 &0 &-3 &2 &3 &0 &3 &-2 &-3 &0 &3 &-2 &-3 &0\\
-3 &2 &3 &0 &3 &-3 &-2 &2 &-3 &3 &0 &0 &-3 &3 &2\\
-2 &3 &-3 &0 &0 &-3 &3 &0 &0 &1 &-1 &0 &0 &3 &-3\\
0 &0 &-1 &1 &0 &0 &0 &3 &3 &-2 &-2 &-5 &-1 &4 &-2\\
1 &1 &4 &-2 &3 &-1 &-4 &-4 &-1 &3 &-2 &4 &1 &1 &-2\\
4 &-1 &-5 &-2 &-2 &3 &3 &0 &0 &2 &-4 &0 &2 &-2 &-2\\
4 &-2 &4 &6 &-4 &0 &-4 &2 &-2 &0 &2 &2 &-2 &-2 &-2\\
-2 &4 &-2 &4 &4 &-4 &0 &-2 &2 &-2 &-1 &-2 &-1 &0 &5\\
0 &-3 &2 &3 &-2 &-5 &0 &-3 &4 &5 &-2 &1 &2 &1 &0\\
-3 &-2 &1 &0 &-1 &0 &3 &2 &3 &-4 &-5 &2 &-1 &0 &1\\
0 &5 &-4 &-3 &2 &3 &0 &-7 &0 &-3 &4 &5 &-2 &1 &2\\
1 &0 &-3 &-4 &3 &0 &-1 &0 &-1 &2 &3 &-2 &-3 &2 &3\\
-5 &0 &2 &-4 &4 &5 &-5 &-2 &0 &5 &-3 &3 &-3 &-2 &2\\
-1 &1 &0 &0 &0 &2 &-3 &1 &-2 &2 &1 &-1 &-1 &3 &-2\\
0 &1 &-3 &1 &-3 &0 &4 &0 &-2 &4 &0 &-1 &-1 &1 &1\\
-2 &0 &-2 &-5 &-1 &4 &3 &-1 &-4 &2 &8 &-1 &0 &-1 &2\\
-5 &3 &-2 &0 &0 &-3 &3 &0 &0 &3 &-3 &3 &-3 &0 &0\\
-3 &3 &0 &0 &0 &0 &3 &-3 &0 &0 &-3 &3 &-3 &3 &0\\
0 &3 &-3 &0 &0 &0 &0 &-6 &2 &1 &-3 &3 &-1 &3 &-1\\
5 &1 &-4 &4 &-2 &-2 &-2 &-2 &4 &-4 &1 &5 &-1 &3 &-1\\
3 &-3 &1 &2 &-6 &0 &0 &-3 &-1 &1 &-1 &1 &-3 &15 &15\\
2 &0 &0 &2 &-4 &0 &2 &-2 &-2 &0 &2 &-4 &0 &4 &0\\
0 &0 &0 &4 &0 &-4 &2 &0 &-2 &-2 &2 &0 &-4 &2 &0\\
0 &2 &2 &-2 &0 &2 &-4 &6 &-2 &-2 &8 &-4 &0 &0 &2\\
0 &-4 &-2 &2 &0 &0 &2 &0 &2 &4 &-2 &-4 &-4 &-2 &-2\\
0 &4 &0 &0 &0 &0 &4 &0 &-2 &-2 &-4 &-4 &-2 &4 &2\\
0 &2 &0 &0 &2 &2 &-2 &0 &2 &0 &4 &-2 &-4 &4 &-2\\
0 &0 &2 &0 &-4 &0 &0 &0 &0 &0 &0 &0 &0 &0 &0\\
0 &0 &0 &0 &0 &0 &0 &0 &0 &0 &0 &0 &0 &0 &0\\
0 &0 &0 &0 &0 &0 &0 &0 &0 &0 &0 &0 &\end{array}$
}
\end{table}
\begin{table}[h!]
\caption{The seventh one of the $13$ tuples (corresponding to $\dim V_w^\text{ori}(4, 5)$ $=$ $13$) consisting of the coefficients in the order of $\check{G}_{4, 5} \cap \conn$ (Tables~\ref{t0a}--\ref{t0c})}\label{t0d7}
{\tiny
$\begin{array}{*{15}c}
1 &-2 &0 &-1 &-2 &7 &0 &-3 &1 &0 &-1 &2 &-1 &-4 &2\\
1 &-2 &-1 &0 &2 &1 &3 &-1 &-2 &2 &-1 &3 &-3 &-2 &0\\
-1 &2 &1 &1 &-2 &-1 &1 &1 &1 &-2 &1 &2 &-1 &-1 &-1\\
0 &1 &-1 &0 &2 &0 &0 &-4 &-4 &0 &0 &0 &0 &1 &-1\\
-1 &1 &0 &0 &0 &0 &-1 &1 &1 &-1 &0 &0 &0 &0 &-1\\
1 &1 &-1 &0 &0 &0 &0 &1 &-1 &-1 &1 &0 &0 &-3 &2\\
4 &-3 &1 &0 &-2 &1 &5 &-4 &-6 &5 &-3 &2 &4 &-3 &1\\
0 &-2 &1 &1 &-2 &0 &1 &-3 &2 &4 &-3 &1 &0 &-2 &1\\
-3 &3 &3 &-3 &1 &-1 &-1 &1 &5 &-5 &-5 &5 &-3 &3 &3\\
-3 &1 &-1 &-1 &1 &1 &-1 &-1 &1 &-3 &3 &3 &-3 &1 &-1\\
-1 &1 &0 &-1 &0 &1 &1 &0 &-1 &0 &-1 &0 &1 &0 &0\\
-1 &1 &0 &0 &-2 &2 &0 &-1 &3 &-1 &-1 &-1 &3 &-1 &-1\\
2 &-4 &0 &2 &2 &0 &-4 &2 &-1 &-1 &3 &-1 &-1 &-1 &3\\
-1 &0 &2 &-2 &0 &-3 &1 &3 &-1 &3 &-1 &-3 &1 &3 &-1\\
-3 &1 &-3 &1 &3 &-1 &3 &-1 &-3 &1 &-3 &1 &3 &-1 &-3\\
1 &3 &-1 &3 &-1 &-3 &1 &4 &-4 &-2 &2 &-4 &4 &2 &-2\\
-2 &2 &0 &0 &2 &-2 &0 &0 &-4 &4 &2 &-2 &4 &-4 &-2\\
2 &2 &-2 &0 &0 &-2 &2 &0 &0 &-6 &4 &6 &-4 &6 &-4\\
-6 &4 &6 &-4 &-6 &4 &-6 &4 &6 &-4 &4 &-2 &-4 &2 &-4\\
2 &4 &-2 &-4 &2 &4 &-2 &4 &-2 &-4 &2 &4 &-4 &-2 &2\\
-4 &4 &2 &-2 &-2 &2 &0 &0 &2 &-2 &0 &0 &-4 &4 &2\\
-2 &4 &-4 &-2 &2 &2 &-2 &0 &0 &-2 &2 &0 &0 &-3 &1\\
4 &-2 &3 &-1 &-4 &2 &2 &0 &-3 &1 &-2 &0 &3 &-1 &3\\
-1 &-4 &2 &-3 &1 &4 &-2 &-2 &0 &3 &-1 &2 &0 &-3 &1\\
-2 &1 &3 &-2 &2 &-1 &-3 &2 &2 &-1 &-3 &2 &-2 &1 &3\\
-2 &2 &-1 &-3 &2 &-2 &1 &3 &-2 &-2 &1 &3 &-2 &2 &-1\\
-3 &2 &-2 &1 &3 &-2 &2 &-1 &-3 &2 &2 &-1 &-3 &2 &-2\\
1 &3 &-2 &2 &-1 &-3 &2 &-2 &1 &3 &-2 &-2 &1 &3 &-2\\
2 &-1 &-3 &2 &-2 &1 &2 &-1 &3 &-2 &-3 &2 &2 &-1 &-2\\
1 &-3 &2 &3 &-2 &2 &-1 &-2 &1 &-3 &2 &3 &-2 &-2 &1\\
2 &-1 &3 &-2 &-3 &2 &1 &-1 &-2 &2 &0 &0 &1 &-1 &-1\\
1 &2 &-2 &0 &0 &-1 &1 &1 &-1 &0 &0 &-2 &2 &1 &-1\\
-1 &1 &0 &0 &2 &-2 &-1 &1 &1 &-3 &0 &2 &0 &2 &-1\\
-1 &3 &-2 &-2 &1 &1 &0 &0 &-1 &1 &-3 &0 &2 &0 &2\\
-1 &-1 &3 &-2 &-2 &1 &1 &0 &0 &-1 &3 &-3 &-2 &2 &-1\\
1 &0 &0 &-4 &4 &3 &-3 &2 &-2 &-1 &1 &-3 &3 &2 &-2\\
1 &-1 &0 &0 &4 &-4 &-3 &3 &-2 &2 &1 &-1 &3 &-2 &-3\\
2 &-1 &0 &1 &0 &-4 &3 &4 &-3 &2 &-1 &-2 &1 &-3 &2\\
3 &-2 &1 &0 &-1 &0 &4 &-3 &-4 &3 &-2 &1 &2 &-1 &-4\\
3 &4 &-3 &4 &-3 &-4 &3 &3 &-2 &-3 &2 &-3 &2 &3 &-2\\
2 &-1 &-2 &1 &-2 &1 &2 &-1 &-1 &0 &1 &0 &1 &0 &-1\\
0 &0 &0 &0 &0 &0 &0 &0 &0 &0 &0 &0 &0 &0 &0\\
0 &0 &0 &0 &1 &-1 &0 &-1 &1 &0 &-1 &1 &1 &-1 &0\\
0 &0 &0 &0 &2 &-1 &-1 &0 &-2 &1 &1 &1 &-3 &0 &2\\
-1 &3 &0 &-2 &-2 &0 &3 &-1 &2 &0 &-3 &1 &1 &1 &-2\\
0 &-1 &-1 &2 &0 &0 &0 &0 &0 &0 &0 &0 &0 &0 &0\\
0 &0 &0 &0 &0 &0 &0 &0 &0 &0 &0 &0 &0 &0 &0\\
0 &0 &0 &0 &0 &0 &0 &0 &0 &0 &0 &0 &0 &0 &0\\
1 &0 &-3 &2 &0 &-1 &2 &-1 &0 &-1 &2 &-1 &-1 &2 &-1\\
0 &0 &-1 &2 &-1 &-1 &2 &-1 &0 &-1 &2 &-1 &0 &2 &-3\\
0 &1 &1 &-2 &-1 &2 &0 &1 &0 &-1 &1 &-1 &0 &0 &1\\
-1 &0 &0 &1 &0 &-3 &2 &0 &-1 &2 &-1 &0 &-1 &2 &-1\\
-1 &2 &-1 &0 &0 &-1 &2 &-1 &-1 &2 &-1 &0 &-1 &2 &-1\\
0 &2 &-3 &0 &1 &1 &-2 &-1 &2 &0 &1 &-1 &0 &1 &-1\\
0 &0 &1 &-1 &0 &0 &0 &0 &0 &0 &0 &0 &0 &0 &0\\
0 &0 &0 &0 &0 &0 &0 &0 &0 &0 &0 &0 &0 &0 &0\\
0 &0 &0 &0 &0 &0 &0 &0 &0 &0 &0 &0 &\end{array}$
}
\end{table}
\begin{table}[h!]
\caption{The eighth one of the $13$ tuples (corresponding to $\dim V_w^\text{ori}(4, 5)$ $=$ $13$) consisting of the coefficients in the order of $\check{G}_{4, 5} \cap \conn$ (Tables~\ref{t0a}--\ref{t0c})}\label{t0d8}
{\tiny
$\begin{array}{*{15}c}
-1 &1 &1 &1 &1 &-4 &0 &1 &0 &0 &-2 &0 &1 &2 &0\\
-1 &1 &1 &-1 &0 &0 &-2 &1 &0 &-1 &1 &-2 &1 &1 &-1\\
1 &0 &0 &-1 &1 &1 &-1 &-1 &0 &1 &-1 &-1 &1 &1 &1\\
-1 &-1 &1 &0 &-2 &0 &0 &4 &4 &0 &0 &0 &0 &-1 &1\\
1 &-1 &0 &0 &0 &0 &1 &-1 &-1 &1 &0 &0 &0 &0 &1\\
-1 &-1 &1 &0 &0 &0 &0 &-1 &1 &1 &-1 &0 &0 &1 &0\\
-2 &1 &0 &-1 &1 &0 &-2 &1 &3 &-2 &1 &0 &-2 &1 &0\\
-1 &1 &0 &-1 &2 &0 &-1 &1 &0 &-2 &1 &0 &-1 &1 &0\\
1 &-1 &-1 &1 &0 &0 &0 &0 &-2 &2 &2 &-2 &1 &-1 &-1\\
1 &0 &0 &0 &0 &-1 &1 &1 &-1 &1 &-1 &-1 &1 &0 &0\\
0 &0 &0 &1 &0 &-1 &-1 &0 &1 &0 &1 &0 &-1 &0 &0\\
1 &-1 &0 &0 &1 &-1 &0 &1 &-2 &0 &1 &0 &-1 &1 &0\\
-1 &2 &0 &-1 &-1 &0 &2 &-1 &0 &1 &-1 &0 &1 &0 &-2\\
1 &0 &-1 &1 &0 &1 &0 &-1 &0 &-1 &0 &1 &0 &-1 &0\\
1 &0 &1 &0 &-1 &0 &-1 &0 &1 &0 &1 &0 &-1 &0 &1\\
0 &-1 &0 &-1 &0 &1 &0 &-1 &2 &0 &-1 &1 &-2 &0 &1\\
0 &-1 &1 &0 &0 &1 &-1 &0 &1 &-2 &0 &1 &-1 &2 &0\\
-1 &0 &1 &-1 &0 &0 &-1 &1 &0 &2 &-1 &-2 &1 &-2 &1\\
2 &-1 &-2 &1 &2 &-1 &2 &-1 &-2 &1 &-1 &0 &1 &0 &1\\
0 &-1 &0 &1 &0 &-1 &0 &-1 &0 &1 &0 &-1 &2 &0 &-1\\
1 &-2 &0 &1 &0 &-1 &1 &0 &0 &1 &-1 &0 &1 &-2 &0\\
1 &-1 &2 &0 &-1 &0 &1 &-1 &0 &0 &-1 &1 &0 &1 &0\\
-1 &0 &-1 &0 &1 &0 &-1 &0 &1 &0 &1 &0 &-1 &0 &-1\\
0 &1 &0 &1 &0 &-1 &0 &1 &0 &-1 &0 &-1 &0 &1 &0\\
1 &0 &-1 &0 &-1 &0 &1 &0 &-1 &0 &1 &0 &1 &0 &-1\\
0 &-1 &0 &1 &0 &1 &0 &-1 &0 &1 &0 &-1 &0 &-1 &0\\
1 &0 &1 &0 &-1 &0 &-1 &0 &1 &0 &-1 &0 &1 &0 &1\\
0 &-1 &0 &-1 &0 &1 &0 &1 &0 &-1 &0 &1 &0 &-1 &0\\
-1 &0 &1 &0 &1 &0 &-1 &0 &-1 &0 &1 &0 &-1 &0 &1\\
0 &1 &0 &-1 &0 &-1 &0 &1 &0 &1 &0 &-1 &0 &1 &0\\
-1 &0 &-1 &0 &1 &0 &0 &0 &1 &-1 &0 &0 &-1 &1 &0\\
0 &-1 &1 &0 &0 &1 &-1 &-1 &1 &0 &0 &1 &-1 &0 &0\\
1 &-1 &0 &0 &-1 &1 &0 &0 &0 &1 &0 &-1 &0 &-1 &0\\
1 &-1 &1 &1 &-1 &0 &0 &0 &0 &0 &1 &0 &-1 &0 &-1\\
0 &1 &-1 &1 &1 &-1 &0 &0 &0 &0 &-1 &1 &1 &-1 &0\\
0 &0 &0 &1 &-1 &-1 &1 &0 &0 &0 &0 &1 &-1 &-1 &1\\
0 &0 &0 &0 &-1 &1 &1 &-1 &0 &0 &0 &0 &-1 &1 &1\\
-1 &0 &0 &0 &0 &1 &-1 &-1 &1 &0 &0 &0 &0 &1 &-1\\
-1 &1 &0 &0 &0 &0 &-1 &1 &1 &-1 &0 &0 &0 &0 &1\\
-1 &-1 &1 &-1 &1 &1 &-1 &-1 &1 &1 &-1 &1 &-1 &-1 &1\\
0 &0 &0 &0 &0 &0 &0 &0 &0 &0 &0 &0 &0 &0 &0\\
0 &0 &0 &0 &0 &0 &0 &0 &0 &0 &0 &0 &0 &0 &0\\
0 &0 &0 &0 &-1 &1 &0 &1 &-1 &0 &1 &-1 &-1 &1 &0\\
0 &0 &0 &0 &-1 &0 &1 &0 &1 &0 &-1 &0 &1 &0 &-1\\
0 &-1 &0 &1 &1 &0 &-1 &0 &-1 &0 &1 &0 &-1 &0 &1\\
0 &1 &0 &-1 &0 &0 &0 &0 &0 &0 &0 &0 &0 &0 &0\\
0 &0 &0 &0 &0 &0 &0 &0 &0 &0 &0 &0 &0 &0 &0\\
0 &0 &0 &0 &0 &0 &0 &0 &0 &0 &0 &0 &0 &0 &0\\
0 &0 &1 &-1 &0 &0 &-1 &1 &0 &0 &-1 &1 &0 &0 &1\\
-1 &-1 &1 &0 &0 &1 &-1 &0 &0 &1 &-1 &0 &0 &-1 &1\\
0 &0 &0 &1 &0 &-1 &0 &-1 &0 &1 &-1 &1 &1 &-1 &0\\
0 &0 &0 &0 &0 &1 &-1 &0 &0 &-1 &1 &0 &0 &-1 &1\\
0 &0 &1 &-1 &-1 &1 &0 &0 &1 &-1 &0 &0 &1 &-1 &0\\
0 &-1 &1 &0 &0 &0 &1 &0 &-1 &0 &-1 &1 &0 &-1 &1\\
1 &-1 &0 &0 &0 &0 &0 &0 &0 &0 &0 &0 &0 &0 &0\\
0 &0 &0 &0 &0 &0 &0 &0 &0 &0 &0 &0 &0 &0 &0\\
0 &0 &0 &0 &0 &0 &0 &0 &0 &0 &0 &0 &\end{array}$
}
\end{table}
\begin{table}[h!]
\caption{The ninth one of the $13$ tuples (corresponding to $\dim V_w^\text{ori}(4, 5)$ $=$ $13$) consisting of the coefficients in the order of $\check{G}_{4, 5} \cap \conn$ (Tables~\ref{t0a}--\ref{t0c})}\label{t0d9}
{\tiny
$\begin{array}{*{15}c}
0 &2 &-1 &0 &1 &-4 &0 &2 &-2 &0 &4 &-3 &0 &3 &-2\\
0 &2 &0 &1 &-1 &-2 &-1 &0 &1 &-3 &0 &-1 &2 &3 &1\\
0 &-2 &-1 &0 &1 &2 &0 &-2 &-1 &1 &0 &0 &0 &0 &0\\
0 &0 &0 &0 &0 &0 &0 &0 &0 &0 &0 &0 &0 &-1 &1\\
0 &0 &1 &-1 &1 &-1 &0 &0 &-1 &1 &0 &0 &0 &0 &1\\
-1 &0 &0 &-1 &1 &-1 &1 &0 &0 &1 &-1 &0 &0 &3 &-2\\
-4 &3 &-2 &1 &3 &-2 &-4 &3 &5 &-4 &3 &-2 &-4 &3 &-1\\
0 &2 &-1 &0 &1 &-1 &0 &2 &-1 &-3 &2 &-1 &0 &2 &-1\\
3 &-3 &-3 &3 &-2 &2 &2 &-2 &-4 &4 &4 &-4 &3 &-3 &-3\\
3 &-1 &1 &1 &-1 &0 &0 &0 &0 &2 &-2 &-2 &2 &-1 &1\\
1 &-1 &0 &1 &0 &-1 &-1 &0 &1 &0 &1 &0 &-1 &0 &0\\
1 &-1 &0 &0 &2 &-1 &-1 &0 &-2 &1 &1 &1 &-3 &0 &2\\
-1 &3 &0 &-2 &-2 &0 &3 &-1 &2 &0 &-3 &1 &1 &1 &-2\\
0 &-1 &-1 &2 &0 &3 &-1 &-3 &1 &-3 &1 &3 &-1 &-3 &1\\
3 &-1 &3 &-1 &-3 &1 &-3 &1 &3 &-1 &3 &-1 &-3 &1 &3\\
-1 &-3 &1 &-3 &1 &3 &-1 &-3 &3 &2 &-2 &3 &-3 &-2 &2\\
2 &-2 &-1 &1 &-2 &2 &1 &-1 &3 &-3 &-2 &2 &-3 &3 &2\\
-2 &-2 &2 &1 &-1 &2 &-2 &-1 &1 &6 &-4 &-6 &4 &-6 &4\\
6 &-4 &-6 &4 &6 &-4 &6 &-4 &-6 &4 &-4 &2 &4 &-2 &4\\
-2 &-4 &2 &4 &-2 &-4 &2 &-4 &2 &4 &-2 &-3 &3 &2 &-2\\
3 &-3 &-2 &2 &2 &-2 &-1 &1 &-2 &2 &1 &-1 &3 &-3 &-2\\
2 &-3 &3 &2 &-2 &-2 &2 &1 &-1 &2 &-2 &-1 &1 &3 &-1\\
-4 &2 &-3 &1 &4 &-2 &-2 &0 &3 &-1 &2 &0 &-3 &1 &-3\\
1 &4 &-2 &3 &-1 &-4 &2 &2 &0 &-3 &1 &-2 &0 &3 &-1\\
3 &-2 &-3 &2 &-2 &1 &2 &-1 &-3 &2 &3 &-2 &2 &-1 &-2\\
1 &-3 &2 &3 &-2 &2 &-1 &-2 &1 &3 &-2 &-3 &2 &-2 &1\\
2 &-1 &3 &-2 &-3 &2 &-2 &1 &2 &-1 &-3 &2 &3 &-2 &2\\
-1 &-2 &1 &-3 &2 &3 &-2 &2 &-1 &-2 &1 &3 &-2 &-3 &2\\
-2 &1 &2 &-1 &3 &-2 &-2 &1 &-3 &2 &2 &-1 &-3 &2 &2\\
-1 &3 &-2 &-2 &1 &-3 &2 &2 &-1 &3 &-2 &-2 &1 &3 &-2\\
-2 &1 &-3 &2 &2 &-1 &0 &1 &1 &-2 &0 &-1 &-1 &2 &0\\
-1 &-1 &2 &0 &1 &1 &-2 &-2 &1 &1 &0 &2 &-1 &-1 &0\\
2 &-1 &-1 &0 &-2 &1 &1 &0 &0 &2 &0 &-2 &0 &-2 &0\\
2 &-2 &2 &2 &-2 &0 &0 &0 &0 &0 &2 &0 &-2 &0 &-2\\
0 &2 &-2 &2 &2 &-2 &0 &0 &0 &0 &-2 &2 &1 &-1 &2\\
-2 &-1 &1 &3 &-3 &-2 &2 &-3 &3 &2 &-2 &2 &-2 &-1 &1\\
-2 &2 &1 &-1 &-3 &3 &2 &-2 &3 &-3 &-2 &2 &-2 &1 &2\\
-1 &2 &-1 &-2 &1 &3 &-2 &-3 &2 &-3 &2 &3 &-2 &2 &-1\\
-2 &1 &-2 &1 &2 &-1 &-3 &2 &3 &-2 &3 &-2 &-3 &2 &4\\
-3 &-5 &4 &-3 &2 &4 &-3 &-3 &2 &4 &-3 &2 &-1 &-3 &2\\
-2 &1 &3 &-2 &1 &0 &-2 &1 &1 &0 &-2 &1 &0 &-1 &1\\
0 &1 &-1 &-1 &1 &0 &0 &0 &0 &0 &0 &1 &-1 &1 &-1\\
0 &0 &0 &-1 &-1 &2 &1 &1 &-2 &0 &0 &-1 &0 &1 &0\\
-1 &1 &0 &0 &-2 &0 &2 &0 &2 &0 &-2 &0 &2 &0 &-2\\
0 &-2 &0 &2 &2 &0 &-2 &0 &-2 &0 &2 &0 &-2 &0 &2\\
0 &2 &0 &-2 &0 &0 &0 &0 &0 &0 &0 &0 &0 &0 &0\\
0 &0 &0 &0 &0 &0 &0 &0 &0 &0 &0 &0 &0 &0 &0\\
0 &0 &0 &0 &0 &0 &0 &0 &0 &0 &0 &0 &0 &0 &0\\
0 &-1 &1 &0 &0 &1 &-1 &0 &0 &1 &-1 &0 &0 &-1 &1\\
0 &0 &1 &-1 &0 &0 &-1 &1 &0 &0 &-1 &1 &0 &0 &1\\
-1 &0 &0 &0 &0 &0 &0 &0 &0 &0 &0 &0 &0 &0 &0\\
0 &0 &0 &0 &-1 &1 &0 &0 &1 &-1 &0 &0 &1 &-1 &0\\
0 &-1 &1 &0 &0 &1 &-1 &0 &0 &-1 &1 &0 &0 &-1 &1\\
0 &0 &1 &-1 &0 &0 &0 &0 &0 &0 &0 &0 &0 &0 &0\\
0 &0 &0 &0 &0 &0 &0 &0 &0 &0 &0 &0 &0 &0 &0\\
0 &0 &0 &0 &0 &0 &0 &0 &0 &0 &0 &0 &0 &0 &0\\
0 &0 &0 &0 &0 &0 &0 &0 &0 &0 &0 &0 &\end{array}$
}
\end{table}
\begin{table}[h!]
\caption{The tenth one of the $13$ tuples (corresponding to $\dim V_w^\text{ori}(4, 5)$ $=$ $13$) consisting of the coefficients in the order of $\check{G}_{4, 5} \cap \conn$ (Tables~\ref{t0a}--\ref{t0c})}\label{t0d10}
{\tiny
$\begin{array}{*{15}c}
0 &-1 &1 &0 &-1 &3 &0 &-2 &1 &0 &-3 &2 &0 &-2 &2\\
0 &-1 &0 &0 &1 &1 &0 &0 &-1 &2 &0 &0 &-2 &-2 &0\\
0 &2 &1 &0 &-1 &-2 &0 &2 &1 &-1 &0 &0 &0 &0 &0\\
0 &0 &0 &0 &0 &0 &0 &0 &0 &0 &0 &0 &0 &0 &0\\
0 &0 &0 &0 &0 &0 &0 &0 &0 &0 &0 &0 &0 &0 &0\\
0 &0 &0 &0 &0 &0 &0 &0 &0 &0 &0 &0 &0 &-2 &2\\
2 &-2 &2 &-2 &-2 &2 &3 &-3 &-3 &3 &-3 &3 &3 &-3 &0\\
0 &0 &0 &0 &0 &0 &0 &-1 &1 &1 &-1 &1 &-1 &-1 &1\\
-2 &2 &2 &-2 &2 &-2 &-2 &2 &3 &-3 &-3 &3 &-3 &3 &3\\
-3 &0 &0 &0 &0 &0 &0 &0 &0 &-1 &1 &1 &-1 &1 &-1\\
-1 &1 &0 &0 &0 &0 &0 &0 &0 &0 &0 &0 &0 &0 &0\\
0 &0 &0 &0 &-2 &0 &2 &0 &2 &0 &-2 &0 &2 &0 &-2\\
0 &-2 &0 &2 &2 &0 &-2 &0 &-2 &0 &2 &0 &-2 &0 &2\\
0 &2 &0 &-2 &0 &-2 &0 &2 &0 &2 &0 &-2 &0 &3 &-1\\
-3 &1 &-3 &1 &3 &-1 &2 &0 &-2 &0 &-2 &0 &2 &0 &-3\\
1 &3 &-1 &3 &-1 &-3 &1 &2 &-2 &-2 &2 &-1 &1 &1 &-1\\
-2 &2 &2 &-2 &1 &-1 &-1 &1 &-2 &2 &2 &-2 &1 &-1 &-1\\
1 &2 &-2 &-2 &2 &-1 &1 &1 &-1 &-4 &2 &4 &-2 &5 &-3\\
-5 &3 &5 &-3 &-5 &3 &-6 &4 &6 &-4 &2 &0 &-2 &0 &-3\\
1 &3 &-1 &-3 &1 &3 &-1 &4 &-2 &-4 &2 &2 &-2 &-1 &1\\
-2 &2 &1 &-1 &-2 &2 &1 &-1 &2 &-2 &-1 &1 &-2 &2 &1\\
-1 &2 &-2 &-1 &1 &2 &-2 &-1 &1 &-2 &2 &1 &-1 &-2 &0\\
3 &-1 &2 &0 &-3 &1 &2 &0 &-3 &1 &-2 &0 &3 &-1 &2\\
0 &-3 &1 &-2 &0 &3 &-1 &-2 &0 &3 &-1 &2 &0 &-3 &1\\
-2 &2 &1 &-1 &2 &-2 &-1 &1 &2 &-2 &-1 &1 &-2 &2 &1\\
-1 &2 &-2 &-1 &1 &-2 &2 &1 &-1 &-2 &2 &1 &-1 &2 &-2\\
-1 &1 &-2 &1 &2 &-1 &2 &-1 &-2 &1 &2 &-1 &-2 &1 &-2\\
1 &2 &-1 &2 &-1 &-2 &1 &-2 &1 &2 &-1 &-2 &1 &2 &-1\\
2 &-1 &-2 &1 &-2 &2 &1 &-1 &2 &-2 &-1 &1 &2 &-2 &-1\\
1 &-2 &2 &1 &-1 &2 &-2 &-1 &1 &-2 &2 &1 &-1 &-2 &2\\
1 &-1 &2 &-2 &-1 &1 &0 &-1 &-1 &2 &0 &1 &1 &-2 &0\\
1 &1 &-2 &0 &-1 &-1 &2 &2 &-1 &-1 &0 &-2 &1 &1 &0\\
-2 &1 &1 &0 &2 &-1 &-1 &0 &0 &-2 &0 &2 &0 &2 &0\\
-2 &2 &-2 &-2 &2 &0 &0 &0 &0 &0 &-2 &0 &2 &0 &2\\
0 &-2 &2 &-2 &-2 &2 &0 &0 &0 &0 &2 &-1 &-2 &1 &-2\\
1 &2 &-1 &-2 &1 &2 &-1 &2 &-1 &-2 &1 &-2 &1 &2 &-1\\
2 &-1 &-2 &1 &2 &-1 &-2 &1 &-2 &1 &2 &-1 &2 &-1 &-2\\
1 &-2 &1 &2 &-1 &-2 &1 &2 &-1 &2 &-1 &-2 &1 &-2 &1\\
2 &-1 &2 &-1 &-2 &1 &2 &-1 &-2 &1 &-2 &1 &2 &-1 &-3\\
3 &3 &-3 &3 &-3 &-3 &3 &2 &-2 &-2 &2 &-2 &2 &2 &-2\\
1 &-1 &-1 &1 &-1 &1 &1 &-1 &0 &0 &0 &0 &0 &0 &0\\
0 &0 &0 &0 &0 &0 &0 &0 &0 &0 &0 &0 &0 &0 &0\\
0 &0 &0 &1 &1 &-2 &-1 &-1 &2 &0 &0 &1 &0 &-1 &0\\
1 &-1 &0 &0 &1 &1 &-2 &0 &-1 &-1 &2 &-1 &0 &0 &1\\
1 &0 &0 &-1 &-1 &0 &0 &1 &1 &0 &0 &-1 &2 &-1 &-1\\
0 &-2 &1 &1 &0 &-1 &0 &2 &-1 &0 &1 &-1 &0 &0 &1\\
-1 &0 &1 &-2 &0 &1 &1 &0 &-2 &1 &0 &-1 &1 &0 &0\\
-1 &1 &0 &-1 &2 &0 &-1 &0 &0 &0 &0 &0 &0 &0 &0\\
0 &0 &0 &0 &0 &0 &0 &0 &0 &0 &0 &0 &0 &0 &0\\
0 &0 &0 &0 &0 &0 &0 &0 &0 &0 &0 &0 &0 &0 &0\\
0 &0 &0 &0 &0 &0 &0 &0 &0 &0 &0 &0 &0 &0 &0\\
0 &0 &0 &0 &0 &0 &0 &0 &0 &0 &0 &0 &0 &0 &0\\
0 &0 &0 &0 &0 &0 &0 &0 &0 &0 &0 &0 &0 &0 &0\\
0 &0 &0 &0 &0 &0 &0 &0 &0 &0 &0 &0 &0 &0 &0\\
0 &0 &0 &0 &0 &0 &0 &0 &0 &0 &0 &0 &0 &0 &0\\
0 &0 &0 &0 &0 &0 &0 &0 &0 &0 &0 &0 &0 &0 &0\\
0 &0 &0 &0 &0 &0 &0 &0 &0 &0 &0 &0 &\end{array}$
}
\end{table}
\begin{table}[h!]
\caption{The eleventh one of the $13$ tuples (corresponding to $\dim V_w^\text{ori}(4, 5)$ $=$ $13$) consisting of the coefficients in the order of $\check{G}_{4, 5} \cap \conn$ (Tables~\ref{t0a}--\ref{t0c})}\label{t0d11}
{\tiny
$\begin{array}{*{15}c}
0 &0 &0 &0 &0 &0 &0 &0 &0 &0 &0 &0 &0 &0 &0\\
0 &0 &0 &0 &0 &0 &0 &0 &0 &0 &0 &0 &0 &0 &0\\
0 &0 &0 &0 &0 &0 &0 &0 &0 &0 &0 &0 &0 &0 &0\\
0 &0 &0 &0 &0 &0 &0 &0 &0 &0 &0 &0 &0 &-1 &1\\
0 &0 &1 &-1 &1 &-1 &0 &0 &-1 &1 &0 &0 &0 &0 &1\\
-1 &0 &0 &-1 &1 &-1 &1 &0 &0 &1 &-1 &0 &0 &0 &1\\
-1 &0 &0 &-1 &1 &0 &0 &-1 &1 &0 &0 &1 &-1 &0 &0\\
-1 &1 &0 &0 &1 &-1 &0 &0 &1 &-1 &0 &0 &-1 &1 &0\\
0 &0 &0 &0 &0 &0 &0 &0 &0 &0 &0 &0 &0 &0 &0\\
0 &0 &0 &0 &0 &0 &0 &0 &0 &0 &0 &0 &0 &0 &0\\
0 &0 &0 &1 &0 &-1 &-1 &0 &1 &0 &1 &0 &-1 &0 &0\\
1 &-1 &0 &0 &0 &0 &0 &0 &0 &0 &0 &0 &0 &0 &0\\
0 &0 &0 &0 &0 &0 &0 &0 &0 &0 &0 &0 &0 &0 &0\\
0 &0 &0 &0 &0 &0 &0 &0 &0 &0 &0 &0 &0 &0 &0\\
0 &0 &0 &0 &0 &0 &0 &0 &0 &0 &0 &0 &0 &0 &0\\
0 &0 &0 &0 &0 &0 &0 &0 &0 &0 &0 &0 &0 &0 &0\\
0 &0 &0 &0 &0 &0 &0 &0 &0 &0 &0 &0 &0 &0 &0\\
0 &0 &0 &0 &0 &0 &0 &0 &0 &0 &0 &0 &0 &0 &0\\
0 &0 &0 &0 &0 &0 &0 &0 &0 &0 &0 &0 &0 &0 &0\\
0 &0 &0 &0 &0 &0 &0 &0 &0 &0 &0 &0 &0 &0 &0\\
0 &0 &0 &0 &0 &0 &0 &0 &0 &0 &0 &0 &0 &0 &0\\
0 &0 &0 &0 &0 &0 &0 &0 &0 &0 &0 &0 &0 &0 &0\\
0 &0 &0 &0 &0 &0 &0 &0 &0 &0 &0 &0 &0 &0 &0\\
0 &0 &0 &0 &0 &0 &0 &0 &0 &0 &0 &0 &0 &0 &0\\
0 &0 &0 &0 &0 &0 &0 &0 &0 &0 &0 &0 &0 &0 &0\\
0 &0 &0 &0 &0 &0 &0 &0 &0 &0 &0 &0 &0 &0 &0\\
0 &0 &0 &0 &0 &0 &0 &0 &0 &0 &0 &0 &0 &0 &0\\
0 &0 &0 &0 &0 &0 &0 &0 &0 &0 &0 &0 &0 &0 &0\\
0 &0 &0 &0 &0 &0 &0 &0 &0 &0 &0 &0 &0 &0 &0\\
0 &0 &0 &0 &0 &0 &0 &0 &0 &0 &0 &0 &0 &0 &0\\
0 &0 &0 &0 &0 &0 &0 &0 &0 &0 &0 &0 &0 &0 &0\\
0 &0 &0 &0 &0 &0 &0 &0 &0 &0 &0 &0 &0 &0 &0\\
0 &0 &0 &0 &0 &0 &0 &0 &0 &0 &0 &0 &0 &0 &0\\
0 &0 &0 &0 &0 &0 &0 &0 &0 &0 &0 &0 &0 &0 &0\\
0 &0 &0 &0 &0 &0 &0 &0 &0 &0 &0 &0 &0 &0 &0\\
0 &0 &0 &0 &0 &0 &0 &0 &0 &0 &0 &0 &0 &0 &0\\
0 &0 &0 &0 &0 &0 &0 &0 &0 &0 &0 &0 &0 &0 &0\\
0 &0 &0 &0 &0 &0 &0 &0 &0 &0 &0 &0 &0 &0 &0\\
0 &0 &0 &0 &0 &0 &0 &0 &0 &0 &0 &0 &0 &0 &0\\
0 &-1 &1 &1 &-1 &0 &0 &0 &0 &1 &-1 &-1 &1 &0 &0\\
0 &0 &1 &-1 &-1 &1 &0 &0 &0 &0 &-1 &1 &1 &-1 &0\\
0 &1 &-1 &-1 &1 &0 &0 &0 &0 &0 &0 &1 &-1 &1 &-1\\
0 &0 &0 &0 &0 &0 &0 &0 &0 &0 &0 &0 &0 &0 &0\\
0 &0 &0 &0 &-1 &0 &1 &0 &1 &0 &-1 &0 &1 &0 &-1\\
0 &-1 &0 &1 &1 &0 &-1 &0 &-1 &0 &1 &0 &-1 &0 &1\\
0 &1 &0 &-1 &0 &0 &-1 &1 &0 &0 &1 &-1 &0 &0 &1\\
-1 &0 &0 &-1 &1 &0 &0 &1 &-1 &0 &0 &-1 &1 &0 &0\\
-1 &1 &0 &0 &1 &-1 &0 &0 &0 &0 &0 &0 &0 &0 &0\\
0 &0 &0 &0 &0 &0 &0 &0 &0 &0 &0 &0 &0 &0 &0\\
0 &0 &0 &0 &0 &0 &0 &0 &0 &0 &0 &0 &0 &0 &0\\
0 &0 &0 &0 &0 &0 &0 &0 &0 &0 &0 &0 &0 &0 &0\\
0 &0 &0 &0 &0 &0 &0 &0 &0 &0 &0 &0 &0 &0 &0\\
0 &0 &0 &0 &0 &0 &0 &0 &0 &0 &0 &0 &0 &0 &0\\
0 &0 &0 &0 &0 &0 &0 &0 &0 &0 &0 &0 &0 &0 &0\\
0 &0 &0 &0 &0 &0 &0 &0 &0 &0 &0 &0 &0 &0 &0\\
0 &0 &0 &0 &0 &0 &0 &0 &0 &0 &0 &0 &0 &0 &0\\
0 &0 &0 &0 &0 &0 &0 &0 &0 &0 &0 &0 &\end{array}$
}
\end{table}

\clearpage

\begin{table}[h!]
\caption{The twelfth one of the $13$ tuples (corresponding to $\dim V_w^\text{ori}(4, 5)$ $=$ $13$) consisting of the coefficients in the order of $\check{G}_{4, 5} \cap \conn$ (Tables~\ref{t0a}--\ref{t0c})}\label{t0d12}
{\tiny
$\begin{array}{*{15}c}
0 &0 &0 &0 &0 &0 &0 &0 &0 &0 &0 &0 &0 &0 &0\\
0 &0 &0 &0 &0 &0 &0 &0 &0 &0 &0 &0 &0 &0 &0\\
0 &0 &0 &0 &0 &0 &0 &0 &0 &0 &0 &0 &0 &0 &0\\
0 &0 &0 &0 &0 &0 &0 &0 &0 &0 &0 &0 &0 &0 &0\\
0 &0 &0 &0 &0 &0 &0 &0 &0 &0 &0 &0 &0 &0 &0\\
0 &0 &0 &0 &0 &0 &0 &0 &0 &0 &0 &0 &0 &0 &0\\
0 &0 &0 &0 &0 &0 &0 &0 &0 &0 &0 &0 &0 &0 &0\\
0 &0 &0 &0 &0 &0 &0 &0 &0 &0 &0 &0 &0 &0 &0\\
0 &0 &0 &0 &0 &0 &0 &0 &0 &0 &0 &0 &0 &0 &0\\
0 &0 &0 &0 &0 &0 &0 &0 &0 &0 &0 &0 &0 &0 &0\\
0 &0 &-1 &1 &1 &-1 &1 &1 &-1 &-1 &-1 &1 &-1 &1 &-1\\
1 &1 &-1 &0 &0 &0 &0 &0 &0 &0 &0 &0 &0 &0 &0\\
0 &0 &0 &0 &0 &0 &0 &0 &0 &0 &0 &0 &0 &0 &0\\
0 &0 &0 &0 &0 &1 &-1 &-1 &1 &1 &-1 &-1 &1 &-1 &1\\
1 &-1 &-1 &1 &1 &-1 &-1 &1 &1 &-1 &-1 &1 &1 &-1 &1\\
-1 &-1 &1 &1 &-1 &-1 &1 &0 &0 &0 &0 &0 &0 &0 &0\\
0 &0 &0 &0 &0 &0 &0 &0 &0 &0 &0 &0 &0 &0 &0\\
0 &0 &0 &0 &0 &0 &0 &0 &0 &1 &-1 &1 &-1 &-1 &1\\
-1 &1 &-1 &1 &-1 &1 &1 &-1 &1 &-1 &-1 &1 &-1 &1 &1\\
-1 &1 &-1 &1 &-1 &1 &-1 &-1 &1 &-1 &1 &0 &0 &0 &0\\
0 &0 &0 &0 &0 &0 &0 &0 &0 &0 &0 &0 &0 &0 &0\\
0 &0 &0 &0 &0 &0 &0 &0 &0 &0 &0 &0 &0 &0 &0\\
0 &0 &0 &0 &0 &0 &0 &0 &0 &0 &0 &0 &0 &0 &0\\
0 &0 &0 &0 &0 &0 &0 &0 &0 &0 &0 &0 &0 &0 &0\\
0 &0 &0 &0 &0 &0 &0 &0 &0 &0 &0 &0 &0 &0 &0\\
0 &0 &0 &0 &0 &0 &0 &0 &0 &0 &0 &0 &0 &0 &0\\
0 &0 &0 &0 &0 &0 &0 &0 &0 &0 &0 &0 &0 &0 &0\\
0 &0 &0 &0 &0 &0 &0 &0 &0 &0 &0 &0 &0 &0 &0\\
0 &0 &0 &0 &0 &0 &0 &0 &0 &0 &0 &0 &0 &0 &0\\
0 &0 &0 &0 &0 &0 &0 &0 &0 &0 &0 &0 &0 &0 &0\\
0 &0 &0 &0 &0 &0 &0 &0 &0 &0 &0 &0 &0 &0 &0\\
0 &0 &0 &0 &0 &0 &0 &0 &0 &0 &0 &0 &0 &0 &0\\
0 &0 &0 &0 &0 &0 &0 &0 &0 &0 &-2 &2 &2 &-2 &0\\
0 &0 &-2 &2 &0 &0 &-2 &2 &0 &0 &0 &0 &0 &0 &0\\
0 &0 &0 &0 &0 &0 &0 &0 &0 &0 &0 &0 &0 &0 &0\\
0 &0 &0 &0 &0 &0 &0 &0 &0 &0 &0 &0 &0 &0 &0\\
0 &0 &0 &0 &0 &0 &0 &0 &0 &0 &0 &0 &0 &0 &0\\
0 &0 &0 &0 &0 &0 &0 &0 &0 &0 &0 &0 &0 &0 &0\\
0 &0 &0 &0 &0 &0 &0 &0 &0 &0 &0 &0 &0 &0 &0\\
0 &0 &0 &0 &0 &0 &0 &0 &0 &0 &0 &0 &0 &0 &0\\
0 &0 &0 &0 &0 &0 &0 &0 &0 &0 &0 &0 &0 &0 &0\\
0 &1 &-1 &1 &-1 &-1 &1 &1 &-1 &1 &1 &-1 &-1 &1 &1\\
-1 &-1 &0 &0 &0 &0 &0 &0 &0 &0 &0 &0 &0 &0 &0\\
0 &0 &0 &0 &0 &0 &0 &0 &0 &0 &0 &0 &0 &0 &0\\
0 &0 &0 &0 &0 &0 &0 &0 &0 &0 &0 &0 &0 &0 &0\\
0 &0 &0 &0 &0 &0 &0 &0 &0 &0 &0 &0 &0 &0 &0\\
0 &0 &0 &0 &0 &0 &0 &0 &0 &0 &0 &0 &0 &0 &0\\
0 &0 &0 &0 &0 &0 &0 &0 &0 &0 &0 &0 &0 &0 &0\\
0 &0 &0 &0 &0 &0 &0 &0 &0 &0 &0 &0 &0 &0 &0\\
0 &0 &0 &0 &0 &0 &0 &0 &0 &0 &0 &0 &0 &0 &0\\
0 &0 &0 &0 &0 &0 &0 &0 &0 &0 &0 &0 &0 &0 &0\\
0 &0 &0 &0 &0 &0 &0 &0 &0 &0 &0 &0 &0 &0 &0\\
0 &0 &0 &0 &0 &0 &0 &0 &0 &0 &0 &0 &0 &0 &0\\
0 &0 &0 &0 &0 &0 &0 &0 &0 &0 &0 &0 &0 &0 &0\\
0 &0 &0 &0 &0 &0 &0 &0 &0 &0 &0 &0 &0 &0 &0\\
0 &0 &0 &0 &0 &0 &0 &0 &0 &0 &0 &0 &0 &0 &0\\
0 &0 &0 &0 &0 &0 &0 &0 &0 &0 &0 &0 &\end{array}$
}
\end{table}
\clearpage
\begin{table}[h!]
\caption{The thirteenth one of the $13$ tuples (corresponding to $\dim V_w^\text{ori}(4, 5)$ $=$ $13$) consisting of the coefficients in the order of $\check{G}_{4, 5} \cap \conn$ (Tables~\ref{t0a}--\ref{t0c})}\label{t0d13}
{\tiny
$\begin{array}{*{15}c}
0 &1 &-1 &0 &1 &-2 &0 &1 &-1 &0 &2 &-1 &0 &1 &-1\\
0 &1 &0 &0 &-1 &-1 &0 &0 &1 &-1 &0 &0 &1 &1 &0\\
0 &-1 &0 &0 &0 &0 &0 &0 &0 &0 &0 &0 &0 &0 &0\\
0 &0 &0 &0 &0 &0 &0 &0 &0 &0 &0 &0 &0 &0 &0\\
0 &0 &0 &0 &0 &0 &0 &0 &0 &0 &0 &0 &0 &0 &0\\
0 &0 &0 &0 &0 &0 &0 &0 &0 &0 &0 &0 &0 &1 &-1\\
-1 &1 &-1 &1 &1 &-1 &-1 &1 &1 &-1 &1 &-1 &-1 &1 &-1\\
1 &1 &-1 &1 &-1 &-1 &1 &1 &-1 &-1 &1 &-1 &1 &1 &-1\\
1 &-1 &-1 &1 &-1 &1 &1 &-1 &-1 &1 &1 &-1 &1 &-1 &-1\\
1 &-1 &1 &1 &-1 &1 &-1 &-1 &1 &1 &-1 &-1 &1 &-1 &1\\
1 &-1 &0 &0 &0 &0 &0 &0 &0 &0 &0 &0 &0 &0 &0\\
0 &0 &0 &0 &0 &0 &0 &0 &0 &0 &0 &0 &0 &0 &0\\
0 &0 &0 &0 &0 &0 &0 &0 &0 &0 &0 &0 &0 &0 &0\\
0 &0 &0 &0 &0 &1 &-1 &-1 &1 &-1 &1 &1 &-1 &-1 &1\\
1 &-1 &1 &-1 &-1 &1 &-1 &1 &1 &-1 &1 &-1 &-1 &1 &1\\
-1 &-1 &1 &-1 &1 &1 &-1 &-1 &1 &1 &-1 &1 &-1 &-1 &1\\
1 &-1 &-1 &1 &-1 &1 &1 &-1 &1 &-1 &-1 &1 &-1 &1 &1\\
-1 &-1 &1 &1 &-1 &1 &-1 &-1 &1 &2 &-2 &-2 &2 &-2 &2\\
2 &-2 &-2 &2 &2 &-2 &2 &-2 &-2 &2 &-2 &2 &2 &-2 &2\\
-2 &-2 &2 &2 &-2 &-2 &2 &-2 &2 &2 &-2 &-1 &1 &1 &-1\\
1 &-1 &-1 &1 &1 &-1 &-1 &1 &-1 &1 &1 &-1 &1 &-1 &-1\\
1 &-1 &1 &1 &-1 &-1 &1 &1 &-1 &1 &-1 &-1 &1 &1 &-1\\
-1 &1 &-1 &1 &1 &-1 &-1 &1 &1 &-1 &1 &-1 &-1 &1 &-1\\
1 &1 &-1 &1 &-1 &-1 &1 &1 &-1 &-1 &1 &-1 &1 &1 &-1\\
1 &-1 &-1 &1 &-1 &1 &1 &-1 &-1 &1 &1 &-1 &1 &-1 &-1\\
1 &-1 &1 &1 &-1 &1 &-1 &-1 &1 &1 &-1 &-1 &1 &-1 &1\\
1 &-1 &1 &-1 &-1 &1 &-1 &1 &1 &-1 &-1 &1 &1 &-1 &1\\
-1 &-1 &1 &-1 &1 &1 &-1 &1 &-1 &-1 &1 &1 &-1 &-1 &1\\
-1 &1 &1 &-1 &1 &-1 &-1 &1 &-1 &1 &1 &-1 &-1 &1 &1\\
-1 &1 &-1 &-1 &1 &-1 &1 &1 &-1 &1 &-1 &-1 &1 &1 &-1\\
-1 &1 &-1 &1 &1 &-1 &0 &0 &0 &0 &0 &0 &0 &0 &0\\
0 &0 &0 &0 &0 &0 &0 &0 &0 &0 &0 &0 &0 &0 &0\\
0 &0 &0 &0 &0 &0 &0 &0 &0 &0 &0 &0 &0 &0 &0\\
0 &0 &0 &0 &0 &0 &0 &0 &0 &0 &0 &0 &0 &0 &0\\
0 &0 &0 &0 &0 &0 &0 &0 &0 &0 &-1 &1 &1 &-1 &1\\
-1 &-1 &1 &1 &-1 &-1 &1 &-1 &1 &1 &-1 &1 &-1 &-1 &1\\
-1 &1 &1 &-1 &-1 &1 &1 &-1 &1 &-1 &-1 &1 &-1 &1 &1\\
-1 &1 &-1 &-1 &1 &1 &-1 &-1 &1 &-1 &1 &1 &-1 &1 &-1\\
-1 &1 &-1 &1 &1 &-1 &-1 &1 &1 &-1 &1 &-1 &-1 &1 &1\\
-1 &-1 &1 &-1 &1 &1 &-1 &-1 &1 &1 &-1 &1 &-1 &-1 &1\\
-1 &1 &1 &-1 &1 &-1 &-1 &1 &1 &-1 &-1 &1 &-1 &1 &1\\
-1 &0 &0 &0 &0 &0 &0 &0 &0 &0 &0 &0 &0 &0 &0\\
0 &0 &0 &0 &0 &0 &0 &0 &0 &0 &0 &0 &0 &0 &0\\
0 &0 &0 &0 &0 &0 &0 &0 &0 &0 &0 &0 &0 &0 &0\\
0 &0 &0 &0 &0 &0 &0 &0 &0 &0 &0 &0 &0 &0 &0\\
0 &0 &0 &0 &0 &0 &0 &0 &0 &0 &0 &0 &0 &0 &0\\
0 &0 &0 &0 &0 &0 &0 &0 &0 &0 &0 &0 &0 &0 &0\\
0 &0 &0 &0 &0 &0 &0 &0 &0 &0 &0 &0 &0 &0 &0\\
0 &0 &0 &0 &0 &0 &0 &0 &0 &0 &0 &0 &0 &0 &0\\
0 &0 &0 &0 &0 &0 &0 &0 &0 &0 &0 &0 &0 &0 &0\\
0 &0 &0 &0 &0 &0 &0 &0 &0 &0 &0 &0 &0 &0 &0\\
0 &0 &0 &0 &0 &0 &0 &0 &0 &0 &0 &0 &0 &0 &0\\
0 &0 &0 &0 &0 &0 &0 &0 &0 &0 &0 &0 &0 &0 &0\\
0 &0 &0 &0 &0 &0 &0 &0 &0 &0 &0 &0 &0 &0 &0\\
0 &0 &0 &0 &0 &0 &0 &0 &0 &0 &0 &0 &0 &0 &0\\
0 &0 &0 &0 &0 &0 &0 &0 &0 &0 &0 &0 &0 &0 &0\\
0 &0 &0 &0 &0 &0 &0 &0 &0 &0 &0 &0 &\end{array}$
}
\end{table}

\section*{Acknowledgements}
The authors would like to thank Dr.~Yuka Kotorii for discussions.  
N.~I. would like to thank Mr.~Yusuke Takimura, for creating and providing many figures of spherical curves for Tables~\ref{t3}--\ref{t1}, for his support and useful discussions.  N.~I. also would like to thank Professor Yukari Funakoshi, Ms.~Megumi Hashizume, and Professor Tsuyoshi Kobayashi for their comments on an earlier version of this paper.       
This work was partially supported by Aoyama Gakuin University-Supported Program ``Early Eagle Program".

\end{document}